\documentclass[12pt,letterpaper,reqno]{amsart}
\usepackage{titlesec}
\titleformat{\section}
  {\centering\normalfont\large\bfseries}{\thesection.}{1em}{}
  \titleformat{\subsection}
  {\normalfont\normalsize\bfseries}{\thesubsection.}{1em}{}
\usepackage{comment}
\usepackage{appendix}
\usepackage{graphicx} 
\usepackage{amsfonts}
\usepackage{amsmath}
\usepackage{pgfplots}
\usepackage{hyperref}
\allowdisplaybreaks
\usepackage{amsthm}
\usepackage{subcaption}
\usepackage{multicol}
\usepackage{xcolor}
\usepackage{mathtools}
\usepackage{fancyhdr,amssymb}
\usepackage{url}
\usepackage{enumerate}
\usepackage{dirtytalk}
\usepackage{tikz-3dplot}

\newcommand{\N}{\mathbb{N}}
\newcommand{\Z}{\mathbb{Z}}

\newcommand{\R}{\mathbb{R}}
\newcommand{\C}{\mathbb{C}}

\newcommand{\Ld}{\Lambda}

\theoremstyle{plain}

\newtheorem{theorem}{Theorem}[section]

\newtheorem{definition}[theorem]{Definition}

\newtheorem{remark}[theorem]{Remark}

\newtheorem{conjecture}[theorem]{Conjecture}

\newtheorem{question}[theorem]{Question}

\tdplotsetmaincoords{70}{160}

\title{\(\ell^{p}\) improving estimates for multilinear forms motivated by distance graphs}

\author{Eyvindur Palsson}\email{palsson@vt.edu}\address{Department of Mathematics, Virginia Tech, Blacksburg, VA 24061}

\author{Jennifer Smucker}\email{jennifer21@vt.edu}\address{Department of Mathematics, Virginia Tech, Blacksburg, VA 24061}

\begin{document}

\begin{abstract}

We undertake a systematic study of the mapping properties of forms based on distance graphs in \(\Z^{d}\) to see how the structure of a graph, \(G\), affects the \(\ell^{p}\) improving estimates of the form, \(\Lambda_{G}\), based on \(G\). This extends previous work on \(\ell^{p}\) improving properties for the spherical averaging operator, which corresponds to a distance graph of a single distance. We obtain \(\ell^{p}\) improving estimates for the collection of forms based on all graphs with 2, 3, and 4 vertices, as well as chains and simplexes of any size in \(\Z^{d}\). Surprisingly, certain mapping properties only seem to depend on the number of vertices in the graph, not its structure, and forms based on subgraphs of a graph, \(G\), do not necessarily inherit all mapping properties from \(G\).

\end{abstract}


\maketitle

\section{Introduction}

Spherical averages on \(\R^{d}\) are operators \(A_{\lambda}\) such that
        \[A_{\lambda}(f)(x)=f\ast d\sigma_{\lambda}(x)=\int_{S_{\lambda}}f(x-y)d\sigma_{\lambda}(y)\]
where \(d\sigma_{\lambda}\) is the normalized invariant measure on the sphere \(S_{\lambda}=\{x:|x|^{2}=\lambda\}\). Spherical averages are trivially bounded in \(L^{p}\) to \(L^{p}\) for \(p\geq 1\). They are fundamental in analysis and partial differential equations and show up in a wide range of problems such as the solution to the three-dimensional wave equation and in the Falconer distance problem on the interface of geometric measure theory and harmonic analysis.

To get pointwise estimates for the spherical averaging operator, one can take the worst radius, \(\lambda\), at each point, \(x\), giving the spherical maximal function: \[\displaystyle A_{\ast}(f)(x)=\sup_{0<\lambda<\infty}|A_{\lambda}(f)(x)|.\] A classic estimate\footnote{Here $A\lesssim B$ means there exists a constant $C$ s.t. $A\leq CB$ } for the spherical maximal function is that \[\|A_{\ast}(f)\|_{L^{p}(\R^{d})}\lesssim\|f\|_{L^{p}(\R^{d})}\]
for \(p>\frac{d}{d-1}\) and \(d\geq 2\), shown by Stein for \(d\geq 3\) in \cite{S76} and Bourgain for \(d=2\) in \cite{B86}.
\\

Spherical averages can also be defined in discrete settings. For \(\lambda\in\N\) and functions \(f:\Z^{d}\to\C\), the spherical averaging operator is defined as \[A_{\lambda}f(x)=\frac{1}{N_{\lambda}}\sum_{y\in\Z^{d}:|y|^{2}=\lambda}f(x-y)\] where \(N_{\lambda}=|\{y\in\Z^{d}:|y|^{2}=\lambda\}|\) is not zero. \(A_{\lambda}\) is the linear operator given by convolution with the discrete probability measure \(\sigma_{\lambda}:=\frac{1}{N_{\lambda}}1_{\{y\in\Z^{d}|y|^2=\lambda\}}.\)  Discrete spherical averages are also trivially bounded from \(\ell^{p}\) to \(\ell^{p}\) for \(p\geq 1\).\\

The discrete analogue of the spherical maximal function is \(\displaystyle A_{\ast}f(n)=\sup_{0<\lambda<\infty}|A_{\lambda}f(n)|\). It was first introduced by Magyar \cite{M97} but then Magyar, Stein, and Wainger \cite{MSW02} showed the sharp result that \(A_{\ast}\) is bounded on \(\ell^{p}\) for \(p>\frac{d}{d-2}\) when \(d\geq 5\). It is notable here that there is a distinct difference in the boundedness properties of the spherical maximal function in the continuous and discrete settings. This motivates exploring more differences between the mapping properties of averaging operators in continuous and discrete settings.

 \subsection{\(L^{p}\) improving}   
    In the continuous setting, we say that  \(A_{\lambda}\) is \(L^{p}\) improving if for some \(q>p\), \(A_{\lambda}\) is bounded from \(L^{p}\) to \(L^{q}\). That is, there is some constant \(C\) so that
        \begin{equation*}
            \|A_{\lambda}(f)\|_{L^{q}}\leq C\|f\|_{L^{p}}.
        \end{equation*}
    Classic results from Littman \cite{L73} and Strichartz \cite{S70} state that \[\|A_{\lambda}(f)\|_{L^{d+1}}\leq C\|f\|_{L^{\frac{d+1}{d}}}\] for \(d\geq 2\). This result can be interpolated with the trivial \(L^{p}\to L^{p}\) bound to get a larger range of \(L^{p}\) improving.

    In the discrete setting, one could ask the same question, \say{When are there exponents \(1\leq p,q\leq\infty\) such that \[\|A_{\lambda}f\|_{\ell^{q}(\Z^{d})}\leq C\|f\|_{\ell^{p}(\Z^{d})}\] where \(C\) is a constant independent of \(\lambda\)?}
    However, the answer is trivial. Using the contraction inequality and nesting of \(\ell^{p}\) spaces, it is true whenever \(1\leq p\leq q\leq\infty\).

    Both Kevin Hughes \cite{H20} and Kessler and Lacey \cite{KL20} address the \(\ell^{p}\) improving properties of the discrete spherical averaging operator by asking the following question:
    \begin{question}[Hughes \cite{H20}, Kessler-Lacey \cite{KL20}]\label{sphavgquestion}
        For each \(1<p<2\), what is the best exponent \(\eta_{p}<0\) so that
        \[\|A_{\lambda}f\|_{\ell^{p'}(\Z^{d})}\leq C\lambda^{\eta_{p}}\|f\|_{\ell^{p}(\Z^{d})}\] where \(C\) is a constant independent of \(\lambda\in\N\) and \(f\in\ell^{p}(\Z^{d})\)?
    \end{question}
    This question asks specifically about \(p'\) where \(\frac{1}{p}+\frac{1}{p'}=1\), but we can ask a similar question about a range of \(q>p\).
    \begin{question}
        For which \(1<p,q<\infty\), is there an exponent \(\eta_{p,q}<0\) so that
        \[\|A_{\lambda}f\|_{\ell^{q}(\Z^{d})}\leq C\lambda^{\eta_{p,q}}\|f\|_{\ell^{p}(\Z^{d})}\] where \(C\) is a constant independent of \(\lambda\in\N\) and \(f\in\ell^{p}(\Z^{d})\)?
    \end{question}
    This formulation of the idea of \(\ell^{p}\) improving is that not only is the \(\ell^{q}\) norm of the spherical average of a function bounded by the \(\ell^{p}\) norm of the function, there is also an additional controlling factor depending on the radius of the spherical average. As that radius, \(\lambda\), increases, the control on the \(\ell^{q}\) norm of the spherical average will grow more and more restrictive.

    Both Hughes and Kessler-Lacey answered Question \ref{sphavgquestion} with the following theorem.
        \begin{theorem}[Hughes \cite{H20}, Kessler-Lacey \cite{KL20}]\label{hklthm}
        If \(d\geq 5\) and \(\frac{d+1}{d-1}< p \leq 2\), then there exists constants \(C_{p}\) depending on \(p\) such that for all \(\lambda\in\N\) , we have the \(\ell^{p}\)-improving inequality 
    \begin{equation}
        \|A_{\lambda}f\|_{\ell^{p'}(\Z^{d})}\leq C_{p}\lambda^{\frac{d}{2}(1-\frac{2}{p})}\|f\|_{\ell^{p}(\Z^{d})}.
    \end{equation}
    The implicit constants are independent of \(\lambda\).
    \end{theorem} 
    One can see that the exponent is sharp by using an example proposed by Hughes that takes \(f\) to be an indicator function for a ball of radius approximately \(\sqrt{\lambda}\).\\
    
    Similarly to the continuous setting, these results can be interpolated with trivial \(\ell^{p}\to\ell^{p}\) results to get that \begin{equation}\label{spheq}
            \|A_{\lambda}f\|_{q}\lesssim\lambda^{\frac{d}{2}(\frac{1}{q}-\frac{1}{p})}\|f\|_{p}
        \end{equation}
        when \(\frac{1}{p}>\frac{1}{q}>\frac{2}{d-1}(\frac{1}{p})\) and \(\frac{1}{p}<\frac{d-1}{d+1}\), or \(\frac{1}{p}>\frac{1}{q}>\frac{d-1}{2}(\frac{1}{p}-1)+1\) and \(\frac{1}{p}\geq\frac{d-1}{d+1}\).
    The region formed by these conditions is shown in Figure \ref{sphfig}.
    By using the same example as Hughes, this range is also sharp.
\begin{remark}
    We use the notation \(\|A_{\lambda}f\|_{q}\lesssim\lambda^{\frac{d}{2}(\frac{1}{q}-\frac{1}{p})}\|f\|_{p}\) to indicate that there is some constant \(C\) such that \(\|A_{\lambda}f\|_{q}\leq C\lambda^{\frac{d}{2}(\frac{1}{q}-\frac{1}{p})}\|f\|_{p}\).
\end{remark}
    \begin{figure}
\centering
\begin{tikzpicture}
    \begin{axis}[
        axis lines = middle,
        xlabel = \(1/p\),
        ylabel = \(1/q\),
        xmin = -0.3, xmax = 1.6,
        ymin = -0.2, ymax = 1.2,
        xtick = {0, 0.5,  1},
        ytick = {0,  0.5, 1},
        grid = both,
        width = 8cm,
        height = 8cm,
        axis equal image
    ]

    \addplot[thick, fill=blue!20] coordinates {
        (0, 0)
        (1, 1)
        (0.9, 0.1)
        (0, 0)
    };

    \addplot[dotted, thick] coordinates {
    (0.5,0.5) (0.9, 0.1)
    };

    \node at (axis cs: 0, 0) [below left] {\((0, 0)\)};
    \node at (axis cs: 1, 1) [above right] {\((1, 1)\)};
    \node at (axis cs: 0.9, 0.35) [below right] {\((\frac{d-1}{d+1}, \frac{2}{d+1})\)};

    \end{axis}
\end{tikzpicture}
\caption{Region where (\ref{spheq}) holds for the spherical averaging operator.}
\label{sphfig}
\end{figure}
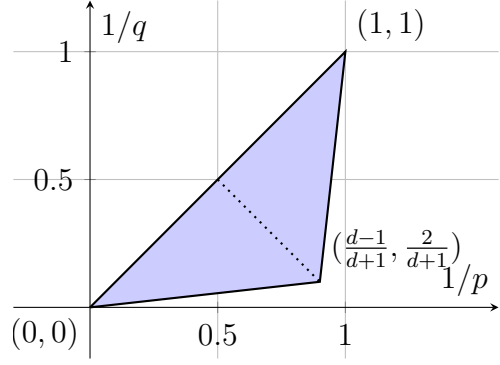

\begin{remark}    Because of the nesting of \(\ell^{p}\) spaces, Figure \ref{sphfig} indicates that there should be \(\ell^{p}\) improving results in the region below the triangle with a worse exponent on \(\lambda\). Indeed this is the case, although the best result is achieved through interpolation and Young's convolution inequality rather than the nesting of \(\ell^{p}\) spaces.  
    Using Young's convolution inequality we have,
        \begin{align*}
            \|A_{\lambda}f\|_{\infty}&= \|f\ast\sigma_{\lambda}\|_{\infty}
            \leq\|f\|_{1}\|\sigma_{\lambda}\|_{\infty}
            \lesssim\|f\|_{1}\lambda^{-\frac{d-2}{2}}.
        \end{align*}
         We interpolate between this bound where \((\frac{1}{p},\frac{1}{q})=(1,0)\) and some point \((\frac{1}{p},\frac{1}{q})\) that lies in the region of \(\ell^{p}\) improving for the spherical averaging operator in Figure \ref{sphfig}. For a point \(\left(\frac{1}{p_{\theta}},\frac{1}{q_{\theta}}\right)\) with \(\frac{1}{p_{\theta}}=1-\theta+\frac{\theta}{p}\) and \(\frac{1}{q_{\theta}}=\frac{\theta}{q}\) with \(0<\theta<1\), interpolation then gives that 
    \[\|A_{\lambda}\|_{q_{\theta}}\lesssim\lambda^{\frac{d}{2}(\frac{1}{q_{\theta}}-\frac{1}{p_{\theta}})+(1-\theta)}\|f\|_{p_{\theta}}.\]
        Thus we have \(\ell^{p}\) improving results for the spherical averaging operator if \((\frac{1}{p_{\theta}},\frac{1}{q_{\theta}})\) is outside of the region in Figure \ref{sphfig}  as long as \(0\leq \frac{1}{q_{\theta}}\leq\frac{1}{p_{\theta}}\leq1\), but the exponent, \(\frac{d}{2}(\frac{1}{q_{\theta}}-\frac{1}{p_{\theta}})+(1-\theta)\), is worse outside of the region in Figure \ref{sphfig}. This exponent can be shown to be sharp by taking \(f\) to be a ball of radius \(\lambda^{\frac{a}{2}}\) where \(a=1-(1-\theta)(\frac{2}{d})p_{\theta}\). 
    \end{remark}

\subsection{Multilinear operators and forms}
 The spherical averaging operator arises in several distance problems, such as the Falconer distance problem in the continuous setting, and the Erd\H{o}s-Falconer distance problem in finite fields. Many other distance graphs have been studied for Falconer type problems, which motivates looking at \(\ell^{p}\) improving results for the forms based on these graphs. Some examples are chains \cite{BIT16,OT22,pinnedtrees}, simplices \cite{GI12,EHI13,GGIP15,GILP15,GIT22,PRA23,GIT24,PRA25,IPPS22} and cycles \cite{GIP17,IMMM25} and even very recent results on general distance graphs \cite{BFOPRA26}. For a more comprehensive list of references on distance graphs see \cite{BFOPRA26}. \\

 Consider a graph, \(G=(\mathcal{V},\mathcal{E})\), with vertex set \(\mathcal{V}=\{v_{1},v_{2},\dots,v_{|\mathcal{V}|}\}\), and edge set \(\mathcal{E}\). We say that a set of points \(\{x_{1},x_{2},\dots,x_{|\mathcal{V}|}\}\) where \(x_{i}\in \Z^{d}\) for \(1\leq i\leq|\mathcal{V}|\) is in a distance configuration given by \(G\) and \(\lambda\in\N\) if for all \((v_{i},v_{j})\in\mathcal{E}\), we have that \(|x_{i}-x_{j}|=\lambda^{\frac{1}{2}}\).\\

\begin{definition}
   For a distance graph, \(G\), with \(k\) vertices, and \(\lambda\in\N\), let \(S_{G,\lambda}(x)\) be the set of all collections of points \(\{x_{2},...,x_{k}\}\) with \(x_{i}\in\Z^{d}\)  \(\forall\, 1\leq i\leq|\mathcal{V}|\) such that \(\{x,x_{2},...x_{k}\}\) is in a distance configuration given by \(G\) and \(\lambda\). Define \(N_{G}(\lambda)=|S_{G,\lambda}|\). Then, the form based on \(G\) and \(\lambda\) is defined as \[\Lambda_{G,\lambda}(f_{1},...,f_{k})=\frac{1}{N_{G}(\lambda)}\sum_{x_{1},...x_{k}\in\Z^{d}}1_{S_{G,\lambda}(x_{1})}(x_{2},...,x_{k})f_{1}(x_{1})...f_{k}(x_{k}).\]

We say that \(\Lambda_{G}\) is \(\ell^{p}\) improving for \(p_{1},\dots,p_{k}\) if 
\begin{equation}
    \Lambda_{G,\lambda}(f_{1},\dots,f_{k})\lesssim \lambda^{\eta}\|f_{1}\|_{p_{1}}\dots\|f_{k}\|_{p_{k}}
\end{equation}
for some \(\eta<0\) independent of \(\lambda\).\\
\end{definition}

\begin{remark}
    Due to \(\lambda\) being present throughout, we suppress \(\lambda\) from the notation and write \(\Lambda_{G,\lambda}\) as \(\Lambda_{G}\), \(N_{G}(\lambda)\) as \(N_{G}\), and \(1_{S_{G,\lambda}(x_{1})}(x_{2},x_{3},...,x_{k})\) as \(S_{G}(x_{1},x_{2},...,x_{k})\). We use the parameter \(\lambda\) for all distances which places our work in a single parameter setting. We leave the multi-parameter setting for future work. Additionally, we denote by \(S_{\lambda}^{d-k}\) a \(d-k\) dimensional sphere in \(\Z^{d}\). When \(k=1\), we write \(S_{\lambda}\) instead of \(S_{\lambda}^{d-1}\) for simplicity. For a discussion of assymptotics for \(N_{G}\), see Section \ref{size section}.
\end{remark}

The simplest example is a form based on \(P_{1}\), a graph with two vertices and one edge connecting them. Then, the form based on \(P_{1}\) is \[\Lambda_{P_{1}}(f_{1},f_{2})=\frac{1}{N_{P_{1}}}\sum_{x_{1},x_{2}\in\Z^{d}}f_{1}(x_{1})f_{2}(x_{2})S_{\lambda}(x_{1}-x_{2})\approx\frac{1}{|S_{\lambda}|}\sum_{x_{1},x_{2}\in\Z^{d}}f_{1}(x_{1})f_{2}(x_{2})S_{\lambda}(x_{1}-x_{2}),\]
where we use the notation \(N_{P_{1}}\approx |S_{\lambda}|\) to indicate that \(|S_{\lambda}|\lesssim N_{P_{1}}\lesssim |S_{\lambda}|.\) The \(\ell^{p}\) improving estimates for this form come as a direct consequence of Theorem \ref{hklthm}. We can write \(\Lambda_{P_{1}}\) as an inner product of a function and the spherical averaging operator, then apply H\"{o}lders inequality to get
\begin{align*}
    \Lambda_{P_{1}}(f_{1},f_{2})&\approx\langle f_{1},A_{\lambda f_{2}}\rangle\\
    &\leq\|f_{1}\|_{p_{1}}\|A_{\lambda}f_{2}\|_{(1-\frac{1}{p_{1}})^{-1}}\\
    &\lesssim \lambda^{\frac{d}{2}(1-\frac{1}{p_{1}}-\frac{1}{p_{2}})}\|f_{1}\|_{p_{1}}\|f_{2}\|_{p_{2}}
\end{align*}
provided that \((\frac{1}{p_{2}},1-\frac{1}{p_{1}})\) is in the region of \(\ell^{p}\) improving for the spherical averaging operator obtained through interpolating Theorem \ref{hklthm} with the trivial \(\ell^{p}\to\ell^{p}\) results. This is discussed more fully in Section \ref{1 chain section}.\\

Indeed, all of these forms can be written as the inner product of a function and an operator, and thus, finding \(\ell^{p}\) improving estimates for the form based on a graph gives information about \(\ell^{p}\) improving estimates for all the operators based on that graph.\\

For more complex examples, consider the forms based on the triangle and the 2-chain, shown in Figure \ref{toyshapes}. 

\begin{figure}
\centering


\begin{subfigure}{0.45\textwidth}
\centering
\begin{tikzpicture}
    \coordinate (A) at (0,0);
    \coordinate (B) at (2,0);
    \coordinate (C) at (1,1.732);
    \draw (A) -- (C) -- (B);
    \fill (A) circle (2pt);
    \fill (B) circle (2pt);
    \fill (C) circle (2pt);
\end{tikzpicture}
\caption{The 2-Chain, \(P_{2}\)}
\end{subfigure}
\hfill
\begin{subfigure}{0.45\textwidth}
\centering
\begin{tikzpicture}
    \coordinate (A) at (0,0);
    \coordinate (B) at (2,0);
    \coordinate (C) at (1,1.732);
    \draw (A) -- (B) -- (C) -- cycle;
    \fill (A) circle (2pt);
    \fill (B) circle (2pt);
    \fill (C) circle (2pt);
\end{tikzpicture}
\caption{The Triangle, \(K_{3}\)}
\end{subfigure}
\caption{ }
\label{toyshapes}
\end{figure}

For the triangle, the indicator function used is \(S_{K_{3}}(x_{1},x_{2},x_{3})=S_{\lambda}(x_{1}-x_{2})S_{\lambda}(x_{2}-x_{3})S_{\lambda}(x_{3}-x_{1})\), so \(N_{K_{3}}\approx|S_{\lambda}||S_{\lambda}^{d-2}|\). Then, we have that \[\Ld_{K_{3}}(f_{1},f_{2},f_{3})\approx\frac{1}{|S_{\lambda}||S_{\lambda}^{d-2}|}\sum_{x_{1},x_{2},x_{3}\in\Z^{d}}f_{1}(x_{1})f_{2}(x_{2})f_{3}(x_{3})S_{\lambda}(x_{1}-x_{2})S_{\lambda}(x_{2}-x_{3})S_{\lambda}(x_{3}-x_{1}).\]
Note that if we take the operator \[A_{tri}(f,g)(x)=\frac{1}{|S_{\lambda}||S_{\lambda}^{d-2}|}\displaystyle\sum_{y,z\in\Z^{d}}f(y)g(z)S_{\lambda}(x-y)S_{\lambda}(y-z)S_{\lambda}(z-x),\] then we can say that 
\[\Ld_{K_{3}}(f_{1},f_{2},f_{3})\approx\langle f_{1},A_{tri}(f_{2},f_{3})\rangle=\langle f_{2},A_{tri}(f_{1},f_{3})\rangle=\langle f_{3},A_{tri}(f_{1},f_{2})\rangle.\]
However, not all graphs have this symmetry throughout. For the 2-chain, the indicator function used is \(S_{P_{2}}=S_{\lambda}(x_{1}-x_{2})S_{\lambda}(x_{2}-x_{3})\), so \(N_{P_{2}}\approx|S_{\lambda}|^{2}\). Then, we have that \[\Ld_{P_{2}}(f_{1},f_{2},f_{3})\approx\frac{1}{|S_{\lambda}|^{2}}\sum_{x_{1},x_{2},x_{3}\in\Z^{d}}f_{1}(x_{1})f_{2}(x_{2})f_{3}(x_{3})S_{\lambda}(x_{1}-x_{2})S_{\lambda}(x_{2}-x_{3}).\]

Consider the operators \[A_{path_{1}}(f,g)(x)=\frac{1}{|S_{\lambda}|^{2}}\displaystyle\sum_{y,z\in\Z^{d}}f(y)g(z)S_{\lambda}(x-y)S_{\lambda}(y-z),\] and \[A_{path_{2}}(f,g)(x)=\frac{1}{|S_{\lambda}|^{2}}\displaystyle\sum_{y,z\in\Z^{d}}f(y)g(z)S_{\lambda}(x-y)S_{\lambda}(x-z).\] We can write \[\Ld_{P_{2}}(f_{1},f_{2},f_{3})\approx\langle f_{1}, A_{path_{1}}(f_{2},f_{3})\rangle=\langle f_{3}, A_{path_{1}}(f_{2},f_{1})\rangle=\langle f_{2},A_{path_{2}}(f_{1},f_{3})\rangle,\] so this form is associated with several multilinear averaging operators. 
\\

Analogous multilinear forms and operators have been studied in the continuous, finite field, and discrete settings. In \cite{IPS22}, Iosevich, Palsson and Sovine showed \(L^{p}\) improving estimates for simplex averaging operators. For example, they found that the triangle averaging operator is bounded in \(L^{\frac{d+1}{d}}\times L^{\frac{d+1}{d}}\) to \(L^{s}\) for \(s\in[\frac{d+1}{2d},1]\) and \(d\geq 2\). This was expanded to a whole class of graphs by Iosevich, Palsson, Wyman, and Zhai \cite{IPWZ25} who showed \(L^{p}\) improving properties for forms based on a regularly realizable graphs. The limitation in these cases is that very little is known about sharpness for the range or what even a conjecture for the range should be.\\

In \cite{BIKP25}, Bhowmik, Iosevich, Koh, and Pham obtained \(L^{p}\)-improving results for multilinear forms based on a collection of graphs with four points in finite fields. In this setting, the \(L^{p}\) improving properties need to be uniform in the size of the field, and therefore look much more similar to the continuous theory than the discrete theory. The authors of the paper asked the following question:
     \begin{question}[Bhowmik-Iosevich-Koh-Pham \cite{BIKP25}]\label{BIKPquestion}
        Suppose that \(G'\) is a subgraph of the graph \(G\) with \(n\) vertices in \(\mathbb{F}_{q}^{d}\). Let \(1\leq p_{i}\leq\infty\), \(1\leq i\leq n\). If the form based on \(G\), \(\Lambda_{G}\), is \(L^{p}\) improving for \(p_{1},...,p_{n}\), is the form \(\Lambda_{G'}\) also \(L^{p}\) improving for \(p_{1},...,p_{n}\)?
    \end{question}

Interestingly, they found that the answer is negative for \(G'=C_{4}\), the diamond shape, and \(G=C_{4+t}\), the diamond with a diagonal. They show that in \(\mathbb{F}_{q}^{2}\), for \((p_{1},p_{2},p_{3},p_{4})=(3/2,\infty,3/2,\infty)\) the diamond is not bounded, but the diamond with a diagonal is. These shapes are shown in Figure \ref{shapes}. \\

In the integer setting, both boundedness properties and \(\ell^{p}\) improving  have been studied for multilinear averaging operators. In \cite{AKP22}, Anderson, Kumchev, and Palsson found boundedness results for the maximal triangle averaging operator in \(\Z^{d}\). This was improved in higher dimensions by Cook, Lyall, and Maygar\cite{CLM21} who extended these results to simplex maximal operators in \(\Z^{d}\). The only \(\ell^{p}\) improving estimates for multilinear averaging operators in the discrete setting are due to Anderson, Kumchev, and Palsson \cite{AKP24} who found \(\ell^{p}\) improving estimates for simplex averaging operators well as \(\ell^{p}\) improving results for the bilinear spherical averaging operator. A comparison between their results for simplex averaging operators and ours can be found in Sections \ref{tri improvement} and \ref{tetr improvement}.

\subsection{Main Results}

Recall that in the integer setting, \(\ell^{p}\)-improving estimates have features different from the \(L^{p}\)-improving estimates in the continuous or finite field settings. Namely, in the integer setting there are two properties that must be considered, the best exponent on \(\lambda\), and the range of \(p_{1},...,p_{k}\) where this best exponent can be guaranteed. In the continuous and finite field settings, the properties that emerge for forms depend highly on the graphs that the forms are based on.

We undertake a systematic study of the \(\ell^{p}\) improving properties of forms based on all possible graphs with 2, 3, and 4 points, as well as paths of any length and \(k-\)simplices in \(\Z^{d}\). These graphs are shown in Figure \ref{shapes}.

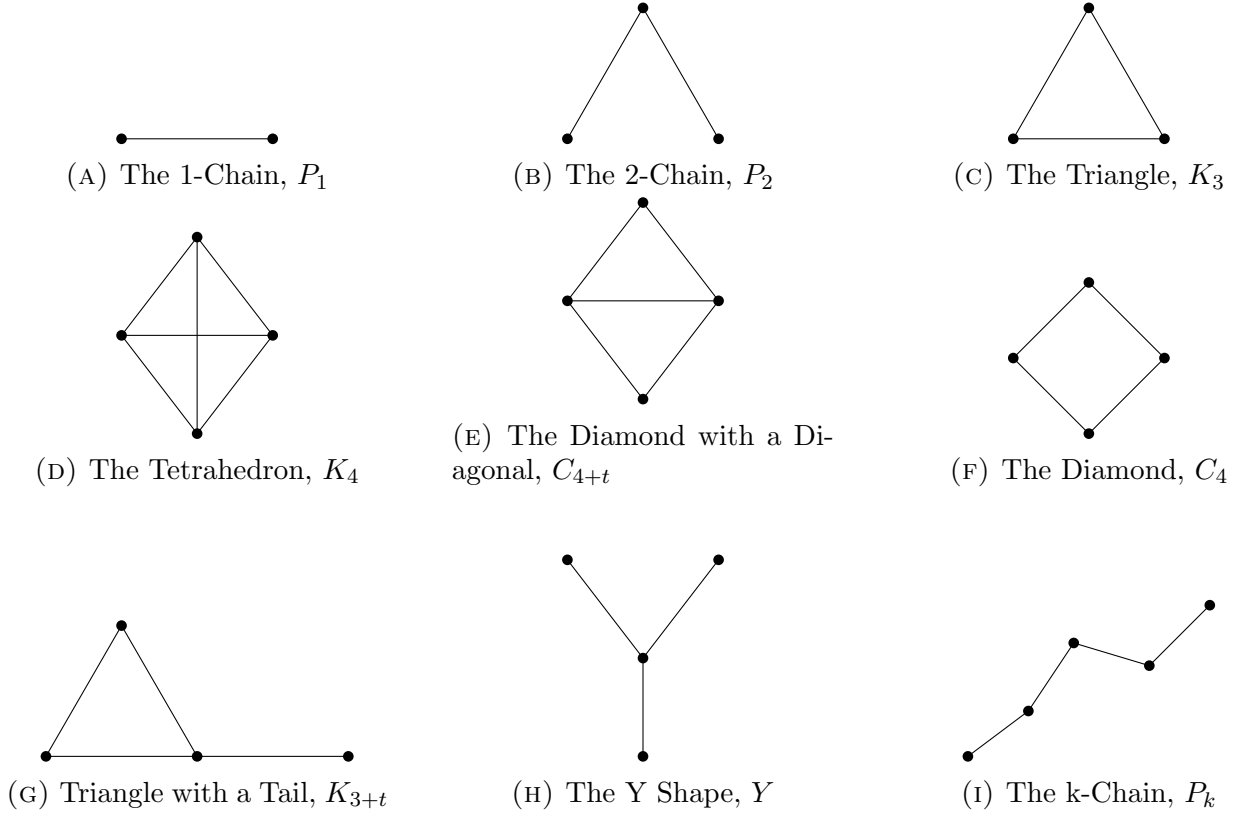
\begin{figure}

\centering

\begin{subfigure}{0.3\textwidth}
\centering
\begin{tikzpicture}
    \coordinate (A) at (0,0);
    \coordinate (B) at (2,0);
    \draw (A) -- (B);
    \fill (A) circle (2pt);
    \fill (B) circle (2pt);
\end{tikzpicture}
\caption{The 1-Chain, \(P_{1}\)}
\end{subfigure}
\hfill
\begin{subfigure}{0.3\textwidth}
\centering
\begin{tikzpicture}
    \coordinate (A) at (0,0);
    \coordinate (B) at (2,0);
    \coordinate (C) at (1,1.732);
    \draw (A) -- (C) -- (B);
    \fill (A) circle (2pt);
    \fill (B) circle (2pt);
    \fill (C) circle (2pt);
\end{tikzpicture}
\caption{The 2-Chain, \(P_{2}\)}
\end{subfigure}
\hfill
\begin{subfigure}{0.3\textwidth}
\centering
\begin{tikzpicture}
    \coordinate (A) at (0,0);
    \coordinate (B) at (2,0);
    \coordinate (C) at (1,1.732);
    \draw (A) -- (B) -- (C) -- cycle;
    \fill (A) circle (2pt);
    \fill (B) circle (2pt);
    \fill (C) circle (2pt);
\end{tikzpicture}
\caption{The Triangle, \(K_{3}\)}
\end{subfigure}
\vspace{1em}
\begin{subfigure}{0.3\textwidth}
\centering
\begin{tikzpicture}
    \coordinate (A) at (0,1.3);
    \coordinate (B) at (0,-1.3);
    \coordinate (C) at (-1,0);
    \coordinate (D) at (1,0);
    \draw (A) -- (C) -- (B) -- (D) -- cycle;
    \draw (C) -- (D);
     \draw (A) -- (B);
    \fill (A) circle (2pt);
    \fill (B) circle (2pt);
    \fill (C) circle (2pt);
    \fill (D) circle (2pt);
\end{tikzpicture}
\caption{The Tetrahedron, \(K_{4}\)}
\end{subfigure}
\hfill
\begin{subfigure}{0.3\textwidth}
\centering
\begin{tikzpicture}
    \coordinate (A) at (0,1.3);
    \coordinate (B) at (0,-1.3);
    \coordinate (C) at (-1,0);
    \coordinate (D) at (1,0);
    \draw (A) -- (C) -- (B) -- (D) -- cycle;
    \draw (C) -- (D);
    \fill (A) circle (2pt);
    \fill (B) circle (2pt);
    \fill (C) circle (2pt);
    \fill (D) circle (2pt);
\end{tikzpicture}
\caption{The Diamond with a Diagonal, \(C_{4+t}\)}
\end{subfigure}
\hfill
\begin{subfigure}{0.3\textwidth}
\centering
\begin{tikzpicture}
    \coordinate (top) at (0,1);
    \coordinate (bottom) at (0,-1);
    \coordinate (left) at (-1,0);
    \coordinate (right) at (1,0);
    \draw (top) -- (right) -- (bottom) -- (left) -- cycle;
    \fill (top) circle (2pt);
    \fill (bottom) circle (2pt);
    \fill (left) circle (2pt);
    \fill (right) circle (2pt);
\end{tikzpicture}
\caption{The Diamond, \(C_{4}\)}
\end{subfigure}

\vspace{1em}


\begin{subfigure}{0.3\textwidth}
\centering
\begin{tikzpicture}
    \coordinate (A) at (0,0);
    \coordinate (B) at (2,0);
    \coordinate (C) at (1,1.73);
    \coordinate (D) at (4,0);
    \draw (A) -- (B) -- (C) -- cycle;
    \draw (D) -- (B);
    \fill (A) circle (2pt);
    \fill (B) circle (2pt);
    \fill (C) circle (2pt);
    \fill (D) circle (2pt);
\end{tikzpicture}
\caption{Triangle with a Tail, \(K_{3+t}\)}
\end{subfigure}
\hfill
\begin{subfigure}{0.3\textwidth}
\centering
\begin{tikzpicture}
    \coordinate (center) at (0,0);
    \coordinate (left) at (-1,1.3);
    \coordinate (right) at (1,1.3);
    \coordinate (down) at (0,-1.3);
    \draw (center) -- (left);
    \draw (center) -- (right);
    \draw (center) -- (down);
    \fill (center) circle (2pt);
    \fill (left) circle (2pt);
    \fill (right) circle (2pt);
    \fill (down) circle (2pt);
\end{tikzpicture}
\caption{The Y Shape, \(Y\)}
\end{subfigure}
\hfill
\begin{subfigure}{0.3\textwidth}
\centering
\begin{tikzpicture}
  \def\len{1}
  \coordinate (P1) at (0, 0);
  \coordinate (P2) at (\len*0.8, \len*0.6);
  \coordinate (P3) at (\len*1.4, \len*1.5);
  \coordinate (P4) at (\len*2.4, \len*1.2);
  \coordinate (P5) at (\len*3.2, \len*2.0);
  \fill (P1) circle (2pt); 
  \fill (P2) circle (2pt); 
  \fill (P3) circle (2pt);
  \fill (P4) circle (2pt); 
  \fill (P5) circle (2pt); 
  \draw (P1) -- (P2) -- (P3) -- (P4) -- (P5);
\end{tikzpicture}
\caption{The k-Chain, \(P_{k}\)}
\end{subfigure}

\caption{The Graphs}
    \label{shapes}
\end{figure}


A surprising feature that emerges in the integer setting is that the best exponent on \(\lambda\) does not seem to depend on the edges of the graphs, rather, it only depends on the number of vertices of that graph.
    \begin{conjecture}\label{conjecture}
        
    For a form \(\Lambda_{G}(f_{1},...,f_{k})\) based on a graph, \(G\) with \(k\) points and for a high enough dimension \(d\),
    \begin{equation}\label{conjeqn}
        \Lambda_{G}(f_{1},...,f_{k})\lesssim\|f_{1}\|_{p_{1}}...\|f_{k}\|_{p_{k}}\lambda^{\frac{d}{2}\left(1-\sum_{i=1}^{k}\frac{1}{p_{i}}\right)}
    \end{equation}
    for a range of \(1<p_{1},...,p_{k}<\infty\).
    In each case, the exponent is sharp. 
   \end{conjecture}

We prove the following Theorem for the graphs shown in Figure \ref{shapes}. 

\begin{theorem}\label{MainTheorem}
    Conjecture \ref{conjecture} is true for the graphs \(K_{3}\), \(K_{4}\), \(C_{4+t}\), \(C_{4}\), \(K_{3+t}\), \(Y\), \(P_{k}\), and \(K_{k}\) for \(k\geq 1\).

\end{theorem}
\begin{remark}
    Theorem \ref{MainTheorem} is the accumulation of Theorem \ref{Lthm} for \(P_{1}\), Theorem \ref{2chain} for \(P_{2}\), Theorem \ref{trithm} for \(K_{3}\), Theorem \ref{tetrtheorem} for \(K_{4}\), Theorem \ref{Diamond+thm} for \(C_{4+t}\), Theorem \ref{Diamondthm} for \(C_{4}\), Theorem \ref{Tttheorem} for \(K_{3+t}\), Theorem \ref{Ytheorem} for Y, Theorem \ref{kchainthm} for a general \(P_{k}\), and Theorem \ref{simtheorem} for a general \(K_{k}\). The restrictions on \(p_{1},...,p_{k}\) and \(d\) vary based on each shape and can be found in each theorem.
\end{remark}
\begin{remark}
    In each of these theorems, the exponent is shown to be sharp by taking each of the functions to be equal to a ball of radius \(\sqrt{\lambda}\). While our theorems establish the sharpness of the exponent on \(\lambda\) for a range of \(p_{1}, ..., p_{k}\), we do not in general have sharpness in the range of \( p_{1}, ..., p_{k}\). We discuss such sharpness in the individual sections when relevant.\\
\end{remark}
\begin{remark}
    The only other multilinear \(\ell^{p}\) improving estimated that existed up until this point are by Anderson, Kumchev, and Palsson \cite{AKP24} for the triangle averaging operator and operators based on other \(k\)-simplices. In both of these cases, we have improvements on their results. For instance, in the case of the triangle, they found that for \(d\geq 7\) and \(\frac{d+1}{d-1}<p<\frac{2d}{d+2}\), the triangle averaging operator satisfies the bound \[\|A_{tri}(f,g)\|_{1}\lesssim\lambda^{1+\frac{d}{2}-\frac{d}{p}}\\f\|_{p}\|g\|_{p}.\] If we apply Theorem \ref{trithm} in this setting, letting \(p_{1}=p_{2}=p\) and \(p_{3}=\infty\), we get a decay term of \(\lambda^{\frac{d}{2}-\frac{d}{p}}\), which we can see is an improvement on the results in \cite{AKP24}. A similar process yields improvement in the case of the tetrahedron as well. The improvement for the tetrahedron is discussed more explicitly in Section \ref{tetr improvement}.
\end{remark}

In order to have good asymptotics for \(N_{G}\), we have to place some restrictions on the dimension, \(d\). In general, we require \(d\geq 5\), as the process of finding \(\ell^{p}\) improving estimates for these forms relies heavily on being able to approximate the number of integer points on a \(d-1\) dimensional sphere. If there is a triangle shape embedded into the graph, then we need to be able to estimate the points on a \(d-2\) dimensional sphere. As shown in \cite{AKP22}, requiring that \(d\geq 7\) will be sufficient. By similar methods, we can say that in the case of the tetrahedron where we need to estimate the number of points on a \(d-3\) dimensional circle, taking \(d\geq 9\) will be sufficient. In the case of a general \(k-\)simplex, a \(d-k\) dimensional circle is required, so we require that \(d>2k+1\).\\

Additionally, in the \(\Z^{d}\) setting, many \(\lambda\in\R^{d}\) will be such that no points in \(\Z^{d}\) lie on a sphere of radius \(\lambda\). For example, in the case of \(P_{2}\), we need \(\lambda\) to be such that \(\lambda\in\N\). As shown in \cite{AKP22} in the case of the triangle, we need not only that \(\lambda\in\N\), but also that \(\lambda\) is even. For a graph, \(G\), we will let \(\mathcal{R}_{G}\) be the set of all \(\lambda\) for which \(G\) is nondegenerate.
\\

We can also ask a question analogous to Question \ref{BIKPquestion}:
\begin{question}\label{mainquestion}
        Suppose that \(G'\) is a subgraph of the graph \(G\) with \(k\) vertices in \(\Z^{d}\). Let \(1\leq p_{i}\leq\infty\), \(1\leq i\leq k\). If the form based on \(G\), \(\Lambda_{G}\), is such that \[\Lambda_{G}\lesssim \lambda^{\frac{d}{2}(1-\sum_{i=1}^{k}\frac{1}{p_{i}})}\|f_{1}\|_{p_{1}}...\|f_{k}\|_{p_{k}}\] for \(p_{1},...,p_{k}\), is the form \(\Lambda_{G'}\) also such that \[\Lambda_{G}\lesssim \lambda^{\frac{d}{2}(1-\sum_{i=1}^{k}\frac{1}{p_{i}})}\|f_{1}\|_{p_{1}}...\|f_{k}\|_{p_{k}}\] for \(p_{1},...,p_{k}\)?
    \end{question}
In the finite field setting, it was shown through an example that the answer to this question is negative, but in the continuous setting, no such example has been found, despite a systematic study in \cite{IPWZ25}.
In the integer setting, as in the finite field case, we show that the answer is negative. There exists a graph \(G\), a subgraph \(G'\), and \(p_{1},\dots,p_{k}\) where Equation \ref{conjeqn} holds with \(p_{1},\dots,p_{k}\) for \(G\), but does not hold with \(p_{1},\dots,p_{k}\) for \(G'\). We verify this in Section \ref{diamondquestion} with the same graphs that Bhowmik, Iosevich, Koh, and Pham used to show that the answer to Question \ref{BIKPquestion} is negative. That is, taking \(G=C_{4+t}\) and \(G'=C_{4}\). We found two additional examples showing the the answer to Question \ref{mainquestion} is negative which are discussed in Sections \ref{diamond+question} and \ref{diamondquestion} and compiled in the following theorem.

\begin{theorem}\label{subgraph theorem}
    The graphs \(C_{4+t}\) and \(C_{4}\) are subgraphs of \(K_{4}\), and \(\Lambda_{K_{4}}\) is such that at the point \(\left(\frac{1}{p_{1}},\frac{1}{p_{2}},\frac{1}{p_{3}},\frac{1}{p_{4}}\right)=\left(0,\frac{d-1}{d+1}-\epsilon,0,\frac{d-1}{d+1}-\epsilon\right)\) for some \(\epsilon>0\) small, \ref{conjeqn} holds.
    However, \(\Lambda_{C_{4+t}}\) is such that as \(\lambda\to\infty\)
    \begin{equation}
        \Lambda_{C_{4+t}}(S_{\lambda},1_{\{0\}},S_{\lambda},1_{\{0\}})\gtrsim\lambda^{\frac{d}{2}\left(1-\frac{2(d-1)}{d+1}+2\epsilon\right)}\|S_{\lambda}\|_{(\frac{d-1}{d+1}-\epsilon)^{-1}}^{2}\|1_{\{0\}}\|_{\infty}^{2},
    \end{equation} 
    and \(\Lambda_{C_{4}}\) is such that as \(\lambda\to\infty\),
    \begin{equation}
        \Lambda_{C_{4}}(S_{\lambda},1_{\{0\}},S_{\lambda},1_{\{0\}})\gtrsim\lambda^{\frac{d}{2}\left(1-\frac{2(d-1)}{d+1}+2\epsilon\right)}\|S_{\lambda}\|_{(\frac{d-1}{d+1}-\epsilon)^{-1}}^{2}\|1_{\{0\}}\|_{\infty}^{2}.
    \end{equation}
\end{theorem}

In Section \ref{1 chain section} we discuss the properties of \(\Lambda_{P_{1}}\). In Section \ref{2 chain section} we discuss the properties of \(\Lambda_{P_{2}}\). In Section \ref{tri section} we discuss the properties of \(\Lambda_{K_{3}}\). In Section \ref{tetr section} we discuss the properties of \(\Lambda_{K_{4}}\). In Section \ref{Diamond+ section} we discuss the properties of \(\Lambda_{C_{4+t}}\) and address one example to show that the answer to Question \ref{mainquestion} is negative. In Section \ref{Diamond section} we discuss the properties of \(\Ld_{C_{4}}\) and address two more examples to show that the answer to Question \ref{mainquestion} is negative. In Section \ref{Tt section} we discuss the properties of \(\Ld_{K_{3+t}}\). In Section \ref{y section}, we discuss the properties of \(\Ld_{Y}\). In Section \ref{chain section} we discuss the properties of \(\Ld_{P_{k}}\) for \(k\geq3\).  Finally, in Section \ref{sim section} we discuss the properties of \(\Ld_{K_{k}}\) for \(k\geq5\).

\section{On \(N_{G}(\lambda)\)}\label{size section}
    We use the shorthand, \(|S_{\lambda}|\) to represent \(N_{\lambda}\), but we must note that this is not perfectly well defined, as there is some variance in the number of points we have on any particular sphere of a fixed radius. However, we know that the number of points on a sphere grows like \(\lambda^{\frac{d-2}{2}}\). We use similar shorthand for lower dimensional spheres.\\
    
    As with \(|S_{\lambda}|\), for a distance graph, \(G\), \(|S_{G}|\) is not necessarily well defined, in the sense that there is always some variance in the number of points in the configuration given by \(G\). However, we can similarly say that the number of points will grow like a power of \(\lambda\).\\

    For a simplex, \(K_{k}\), following \cite{CLM21}, we note by the work of Siegel \cite{S44}, Raghavan \cite{R59}, and Kitaoke \cite{K86} that \[N_{K_{k}}\approx\lambda^{\frac{d(k-1)-k(k-1)}{2}}\approx\prod_{i=1}^{k}|S_{\lambda}^{d-i}|.\]

    For a chain of length \(k\), we can approximate \(N_{P_{k}}\) by seeing that for some \(\lambda\), and \(x_{1}\in\Z^{d}\),
    \begin{align*}
        N_{P_{k}}&=\sum_{x_{2},...,x_{k+1}\in\Z^{d}}1_{S_{P_{k}}}(x_{1},x_{2},...,x_{k+1})\\
        &=\sum_{x_{2},...,x_{k+1}\in\Z^{d}}\prod_{i=1}^{k}S_{\lambda}(x_{i}-x_{i+1})\\
        &=\sum_{x_{2},...,x_{k}\in\Z^{d}}\prod_{i=1}^{k-1}S_{\lambda}(x_{i}-x_{i+1})\sum_{x_{k+1}\in\Z^{d}}S_{\lambda}(x_{k}-x_{k+1})\\
        &\approx |S_{\lambda}|\sum_{x_{2},...,x_{k}\in\Z^{d}}\prod_{i=1}^{k-1}S_{\lambda}(x_{i}-x_{i+1})\\
        &\quad\vdots\\
        &\approx |S_{\lambda}|^{k}\approx\lambda^{\frac{k(d-2)}{2}}.
    \end{align*}
    By the same process, we get that \(N_{C_{4+t}}\approx|S_{\lambda}||S_{\lambda}^{d-2}|^{2}\approx\lambda^{\frac{3d}{2}-5}\), \(N_{Y}\approx|S_{\lambda}|^{3}\approx\lambda^{\frac{3(d-2)}{2}}\), and \(N_{K_{3+t}}\approx|S_{K_{3}}||S_{\lambda}|\approx\lambda^{\frac{3d}{2}-4}\).\\

    The final graph to consider is \(C_{4}\). We know that 
    \begin{align*}
        N_{C_{4}}&=\sum_{x_{2},x_{3},x_{4}\in\Z^{d}}S_{\lambda}(x_{1}-x_{2})S_{\lambda}(x_{2}-x_{3})S_{\lambda}(x_{3}-x_{4})S_{\lambda}(x_{4}-x_{1})\\
        &=\sum_{x_{2},x_{4}\in\Z^{d}}S_{\lambda}(x_{1}-x_{2})S_{\lambda}(x_{1}-x_{4})\sum_{x_{3}\in\Z^{d}}S_{\lambda}(x_{3}-x_{4})S_{\lambda}(x_{2}-x_{3})\\
        &\lesssim|S_{\lambda}^{d-2}|\sum_{x_{2},x_{4}\in\Z^{d}}S_{\lambda}(x_{1}-x_{2})S_{\lambda}(x_{1}-x_{4})\\
        &=|S_{\lambda}^{d-2}|\sum_{x_{2}\in\Z^{d}}S_{\lambda}(x_{1}-x_{2})S_{\lambda}\sum_{x_{4}\in\Z^{d}}S_{\lambda}(x_{1}-x_{4})\\
        &\approx |S_{\lambda}^{d-2}||S_{\lambda}|^{2}\approx\lambda^{\frac{3d}{2}-4}.
    \end{align*}
    This is an overestimate, as it could be that \(x_{2}\) and \(x_{4}\) are \(2\lambda\) apart, or nearly that, in which case one might expect that \(\sum_{x_{3}\in\Z^{d}}S_{\lambda}(x_{3}-x_{4})S_{\lambda}(x_{2}-x_{3})\) is much less than \(|S_{\lambda}^{d-2}|\). \\

   To establish that \(N_{C_{4}}\gtrsim|S_{\lambda}|^{2}|S_{\lambda}^{d-2}|\), we consider a \(d-1\) dimensional sphere of radius \(\lambda\) centered at \(x_{1}\) in \(\Z^{d}\) where \(d\geq7\). Then, we know that \(x_{2}\) and \(x_{4}\) must lie on that sphere. Now, we restrict our choices for to only the points on the sphere whose last coordinates are between \(\frac{1}{2}\lambda\) and \(\frac{\sqrt{3}}{2}\lambda\). We will call this smaller region of the sphere \(S_{\lambda}'\). \(S_{\lambda}'\) is \(\frac{1}{6}\) of the whole sphere, and thus due to symmetry \(|S_{\lambda'}|\approx\frac{1}{6}|S_{\lambda}|\).\\
    Then,

    \begin{align*}
        N_{C_{4}}&=\sum_{x_{2},x_{3},x_{4}\in\Z^{d}}S_{\lambda}(x_{1}-x_{2})S_{\lambda}(x_{2}-x_{3})S_{\lambda}(x_{3}-x_{4})S_{\lambda}(x_{4}-x_{1})\\
        &=\sum_{x_{2},x_{4}\in\Z^{d}}S_{\lambda}(x_{1}-x_{2})S_{\lambda}(x_{1}-x_{4})\sum_{x_{3}\in\Z^{d}}S_{\lambda}(x_{3}-x_{4})S_{\lambda}(x_{2}-x_{3})\\
        &\geq \sum_{x_{2},x_{4}\in\Z^{d}}S_{\lambda}'(x_{1}-x_{2})S_{\lambda}'(x_{1}-x_{4})\sum_{x_{3}\in\Z^{d}}S_{\lambda}(x_{3}-x_{4})S_{\lambda}(x_{2}-x_{3})
    \end{align*}

    We know that \(\displaystyle \sum_{x_{3}\in\Z^{d}}S_{\lambda}(x_{3}-x_{4})S_{\lambda}(x_{2}-x_{3})\) is a \(d-2\) dimensional sphere of some radius \(\lambda'\leq\lambda\). We know that \(\lambda'\) is smallest when \(x_{2}\) and \(x_{4}\) are the farthest apart. Given that \(x_{2},x_{4}\in S_{\lambda}'\), the largest possible distance between them is \(\sqrt{3}\lambda\), in which case we have that \(\lambda'=\frac{1}{2}\). Thus, 
    \begin{align*}
        N_{C_{4}}&\geq \sum_{x_{2},x_{4}\in\Z^{d}}S_{\lambda}'(x_{1}-x_{2})S_{\lambda}'(x_{1}-x_{4})\sum_{x_{3}\in\Z^{d}}S_{\lambda}(x_{3}-x_{4})S_{\lambda}(x_{2}-x_{3})\\
        &\gtrsim|S_{\frac{1}{2}\lambda}^{d-2}|\sum_{x_{2},x_{4}\in\Z^{d}}S_{\lambda}'(x_{1}-x_{2})S_{\lambda}'(x_{1}-x_{4})\\
        &\approx |S_{\frac{1}{2}\lambda}^{d-2}||S_{\lambda}'|^{2}\\
        &\approx |S_{\lambda}^{d-2}||S_{\lambda}|^{2}.
    \end{align*}
    Thus, we have that \(|S_{\lambda}^{d-2}||S_{\lambda}|^{2}\lesssim N_{C_{4}}\lesssim |S_{\lambda}^{d-2}||S_{\lambda}|^{2}\).


\section{The 1-Chain, \(P_{1}\)}\label{1 chain section}

Define \[\Ld_{P_{1}}(f_{1},f_{2})=\frac{1}{N_{P_{1}}}\sum_{x_{1},x_{2}\in\Z^{d}}f_{1}(x_{1})f_{2}(x_{2})S_{\lambda}(x_{1}-x_{2}).\]

\begin{theorem}\label{Lthm}
    If \(d\geq 5\) and \(1<p_{1},p_{2}<\infty\) with \(\frac{1}{p_{1}}+\frac{1}{p_{2}}>1\), then there exist constants \(C_{p_{1},p_{2}}\) such that for all \(\lambda\in\mathcal{R}_{P_{1}}\) we have the \(\ell^{p}\) improving inequality
    \begin{equation}\label{Lequation}\Ld_{P_{1}}(f_{1},f_{2})\leq C_{p_{1},p_{2}}\lambda^{\frac{d}{2}(1-\frac{1}{p_{1}}-\frac{1}{p_{2}})}\|f_{1}\|_{p_{1}}\|f_{2}\|_{p_{2}}\end{equation}
    provided \((\frac{1}{p_{1}},\frac{1}{p_{2}})\) is in the convex hull of the points:
    \begin{itemize}
        \item \((1,0)\)
        \item \((0,1)\)
        \item \((\frac{d-1}{d+1},\frac{d-1}{d+1})\)
    \end{itemize}
\end{theorem}

\begin{remark}
    The region in Theorem \ref{Lthm} is the set of all points \(\left(\frac{1}{p_{1}},\frac{1}{p_{2}}\right)\) such that the following hold:
        \begin{enumerate}
        \item \(\frac{1}{p_{1}}+\frac{2}{d-1}\left(\frac{1}{p_{2}}\right)<1\) when \(\frac{1}{p_{1}}<\frac{d-1}{d+1}\).
        \item \((\frac{2}{d-1})\frac{1}{p_{1}}+\frac{1}{p_{2}}<1\) when  \(\frac{1}{p_{1}}\geq\frac{d-1}{d+1}\).
    \end{enumerate}
    This region is shown in Figure \ref{Lfig}.\\
    
    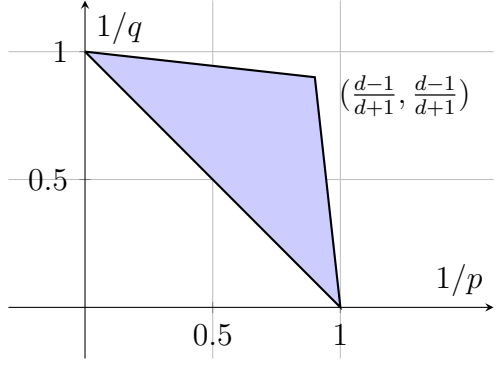
\begin{figure}
\centering
\begin{tikzpicture}
    \begin{axis}[
        axis lines = middle,
        xlabel = \(1/p\),
        ylabel = \(1/q\),
        xmin = -0.3, xmax = 1.6,
        ymin = -0.2, ymax = 1.2,
        xtick = {0, 0.5,  1},
        ytick = {0,  0.5, 1},
        grid = both,
        width = 8cm,
        height = 8cm,
        axis equal image
    ]

    \addplot[thick, fill=blue!20] coordinates {
        (0, 1)
        (1, 0)
        (0.9, 0.9)
        (0, 1)
    };


    \node at (axis cs: 0.95, 0.95) [below right] {\((\frac{d-1}{d+1}, \frac{d-1}{d+1})\)};

    \end{axis}
\end{tikzpicture}
\caption{Region of \(\ell^{p}\) improving for \(\Ld_{P_{1}}\).}
\label{Lfig}
\end{figure}
\end{remark}

\begin{proof}[Proof of Theorem \ref{Lthm}]
This theorem is a direct consequence of Theorem \ref{hklthm}.\\ Let \(q_{2}\) be such that \(\frac{1}{q_{2}}=1-\frac{1}{p_{1}}\), and \((\frac{1}{p_{2}},\frac{1}{q_{2}})\) is in the region of \(\ell^{p}\) improving for the spherical averaging operator. Then,
\begin{align*}
    \Ld_{P_{1}}(f_{1},f_{2})&\approx\langle f_{1},A_{\lambda}f_{2}\rangle\\
    &\leq\|f_{1}\|_{p_{1}}\|A_{\lambda}f_{2}\|_{q_{2}} \\
    &\lesssim\lambda^{\frac{d}{2}(\frac{1}{q_{2}}-\frac{1}{p_{1}})}\|f_{1}\|_{p_{1}}\|f_{2}\|_{p_{2}}\\
    &=\lambda^{\frac{d}{2}(1-\frac{1}{p_{1}}-\frac{1}{p_{2}})}\|f_{1}\|_{p_{1}}\|f_{2}\|_{p_{2}}.
\end{align*}
Thus, we get that (\ref{Lequation}) holds as long as \((\frac{1}{p_{2}},1-\frac{1}{p_{1}})\) lies within the region of \(\ell^{p}\) improving for the spherical averaging operator in Figure \ref{sphfig}, which happens when the following conditions are satisfied:
\begin{enumerate}
    \item \(\frac{1}{p_{1}}+\frac{1}{p_{2}}>1\)
    \item \(\frac{2}{d-1}(\frac{1}{p_{1}})+\frac{1}{p_{2}}<1\) for \(\frac{1}{p_{1}}\geq\frac{d-1}{d+1}\)
    \item \(\frac{2}{d-1}(\frac{1}{p_{2}})+\frac{1}{p_{1}}<1\) for  \(\frac{1}{p_{1}}<\frac{d-1}{d+1}\).
\end{enumerate}
This is equivalent to saying that the point \((\frac{1}{p_{1}},\frac{1}{p_{2}})\) is in the convex hull of the points
\begin{itemize}
        \item \((1,0)\)
        \item \((0,1)\)
        \item \((\frac{d-1}{d+1},\frac{d-1}{d+1})\).
    \end{itemize}
\end{proof}

\subsection{Exponent sharpness}
To see that the exponent of \(\frac{d}{2}(1-\frac{1}{p_{1}}-\frac{1}{p_{2}})\) on \(\lambda\) is the best we can guarantee, take \(f_{1}=f_{2}=B_{\lambda}\) where \(B_{\lambda}\) is an indicator function for a ball of radius \(\lambda^{\frac{1}{2}}\). Note that as discussed in \cite{H20}, \(A_{\lambda}B_{\lambda}(x)\approx B_{\lambda}(x)\). Then,
\[\|f_{1}\|_{p_{1}}\|f_{2}\|_{p_{2}}\approx\lambda^{\frac{d}{2}(\frac{1}{p_{1}}+\frac{1}{p_{2}})},\]
and 
\begin{align*}
    \Ld_{P_{1}}(f_{1},f_{2})&\approx\frac{1}{|S_{\lambda}|}\sum_{x_{1},x_{2}\in\Z^{d}}B_{\lambda}(x_{1})B_{\lambda}(x_{2})S_{\lambda}(x_{1}-x_{2})\\
    &=\frac{1}{|S_{\lambda}|}\sum_{x_{1}\in\Z^{d}}B_{\lambda}(x_{1})\sum_{x_{2}\in\Z^{d}}B_{\lambda}(x_{2})S_{\lambda}(x_{1}-x_{2})\\
    &=\sum_{x_{1}\in\Z^{d}}B_{\lambda}(x_{1})A_{\lambda}B_{\lambda}(x_{1})\\
    &\approx\sum_{x_{1}\in\Z^{d}}B_{\lambda}(x_{1})\\
    &\approx \lambda^{\frac{d}{2}}\approx\lambda^{\frac{d}{2}(1-\frac{1}{p_{1}}-\frac{1}{p_{2}})}\|f_{1}\|_{p_{1}}\|f_{2}\|_{p_{2}}.
\end{align*}

Thus, \(\frac{d}{2}(1-\frac{1}{p_{1}}-\frac{1}{p_{2}})\) is the best exponent on \(\lambda\) that we can guarantee.


\section{The 2-Chain, \(P_{2}\)} \label{2 chain section}
Define
\[\Ld_{P_{2}}(f_{1},f_{2},f_{3})=\frac{1}{N_{P_{2}}}\sum_{x_{1},x_{2},x_{3}\in\Z^{d}}S_{\lambda}(x_{1}-x_{2})S_{\lambda}(x_{2}-x_{3})f_{1}(x_{1})f_{2}(x_{2})f_{3}(x_{3}).\]

\begin{theorem}\label{2chain}
    If \(d\geq 5\) and \(1<p_{1},p_{2},p_{3}<\infty\) with \(\frac{1}{p_{1}}+\frac{1}{p_{2}}+\frac{1}{p_{3}}>1\), then there exist constants \(C_{p_{1},p_{2},p_{3}}\) such that for all \(\lambda\in\mathcal{R}_{P_{2}}\) we have the \(\ell^{p}\) improving inequality
    \begin{equation}\label{2chainineq}\Ld_{P_{2}}(f_{1},f_{2},f_{3})\leq C_{p_{1},p_{2},p_{3}}\lambda^{\frac{d}{2}(1-\frac{1}{p_{1}}-\frac{1}{p_{2}}-\frac{1}{p_{3}})}\|f_{1}\|_{p_{1}}\|f_{2}\|_{p_{2}}\|f_{3}\|_{p_{3}}\end{equation}
    provided \(\left(\frac{1}{p_{1}},\frac{1}{p_{2}},\frac{1}{p_{3}}\right)\) is in the convex hull of the points:
\begin{multicols}{3}
    \begin{itemize}
        \item \((1,0,0)\)
        \item \((0,1,0)\)
        \item \((0,0,1)\)
        \columnbreak
        \item \((\frac{d-1}{d+1},\frac{d-1}{d+1},0)\)
        \item \((0,\frac{d-1}{d+1},\frac{d-1}{d+1})\)
        \item \((\frac{d-1}{d+1},\frac{d-3}{d+1},\frac{d-1}{d+1})\)
        \columnbreak
        \item \((\frac{d-1}{d+1},0,\frac{d^{2}-5}{d^{2}-1})\)
        \item \((\frac{d^{2}-5}{d^{2}-1},0,\frac{d-1}{d+1})\)
    \end{itemize}
\end{multicols}
\end{theorem}

\begin{remark}\label{2chainrmk}
    The region in Theorem \ref{2chain} is the set of all points \(\left(\frac{1}{p_{1}},\frac{1}{p_{2}},\frac{1}{p_{3}}\right)\) such that the following hold:
        \begin{enumerate}
        \item \(\frac{1}{p_{2}}+\frac{2}{d-1}\left(\frac{1}{p_{1}}+\frac{1}{p_{3}}\right)<1\) when \(\frac{1}{p_{1}},\frac{1}{p_{3}}\leq\frac{d-1}{d+1}\)
        \item \((\frac{2}{d-1})\frac{1}{p_{3}}+\frac{1}{p_{2}}+(\frac{d-1}{2})\frac{1}{p_{1}}<\frac{d-1}{2}\) when \(\frac{1}{p_{3}}\leq\frac{d-1}{d+1}<\frac{1}{p_{1}}\)
        \item \((\frac{2}{d-1})\frac{1}{p_{1}}+\frac{1}{p_{2}}+(\frac{d-1}{2})\frac{1}{p_{3}}<\frac{d-1}{2}\) when \(\frac{1}{p_{1}}\leq\frac{d-1}{d+1}<\frac{1}{p_{3}}\)
        \item \(\frac{1}{p_{1}}+\frac{1}{p_{3}}+(\frac{2}{d-1})\frac{1}{p_{2}}<2-\frac{2}{d-1}\) when \(\frac{1}{p_{1}},\frac{1}{p_{3}}>\frac{d-1}{d+1}\).
    \end{enumerate}
    This region is shown in Figure \ref{fig:chainbounds}.\\
    
    \begin{figure}[h]
\centering
    \begin{tikzpicture}[tdplot_main_coords,scale=4]

\coordinate (A) at (0,1,0);
\coordinate (B) at (1,0,0);
\coordinate (C) at (0,0,1);
\coordinate (D) at (0,10/11,10/11);
\coordinate (E) at (10/11,10/11,0);
\coordinate (G) at (0.91,0.82,0.91);
\coordinate (H) at (0.91,0.82,0.91);
\coordinate (I) at (10/11,0,0.99);
\coordinate (J) at (0.99,0,10/11);

\draw[->] (0,0,0) -- (1.2,0,0) node[below left]{$\frac{1}{p_{1}}$};
\draw[->] (0,0,0) -- (0,1.2,0) node[left]{$\frac{1}{p_{2}}$};
\draw[->] (0,0,0) -- (0,0,1.2) node[above]{$\frac{1}{p_{3}}$};


    \filldraw[fill=gray!40,opacity=0.7] (A) -- (B) -- (C) -- cycle;
    \filldraw[fill=gray!40,opacity=0.7] (I)-- (J) -- (B) -- (C) -- cycle;
    \filldraw[fill=gray!40,opacity=0.7] (A) -- (C) -- (D) -- cycle;
    \filldraw[fill=gray!40,opacity=0.7] (A) -- (B) -- (E) -- cycle;
    \filldraw[fill=gray!40,opacity=0.7] (I) -- (J) -- (H) -- cycle;
    \filldraw[fill=gray!40,opacity=0.7] (I) -- (H) -- (D) -- (C) -- cycle;
    \filldraw[fill=gray!40,opacity=0.7] (J) -- (H) -- (E) -- (B) -- cycle;
    \filldraw[fill=gray!40,opacity=0.7] (A) -- (E) -- (H) -- (D) -- cycle;
    \foreach \p in {A,B,C,D,E,G,H,I,J}
       \fill (\p) circle[radius=0.5pt];

   
\end{tikzpicture}
\caption{Region of \(\ell^{p}\) improvement for \(\Ld_{P_{2}}\).}
\label{fig:chainbounds}
\end{figure}
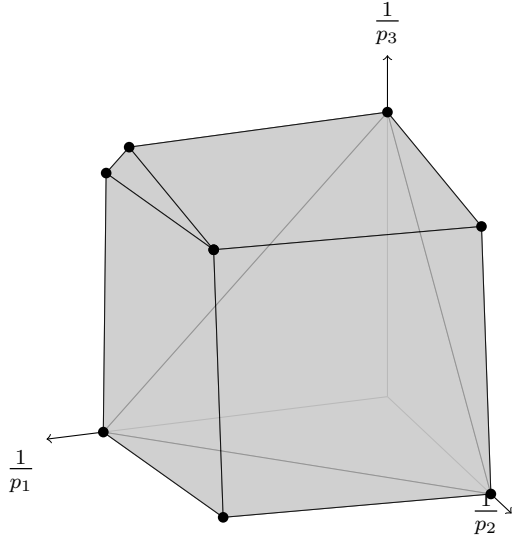
\end{remark}

\begin{proof}[Proof of Theorem \ref{2chain}]
    Note that \[\Ld_{P_{2}}(f_{1},f_{2},f_{3})\approx\frac{1}{|S_{\lambda}|^{2}}\sum_{x_{1},x_{2},x_{3}\in\Z^{d}}S_{\lambda}(x_{1}-x_{2})S_{\lambda}(x_{2}-x_{3})f_{1}(x_{1})f_{2}(x_{2})f_{3}(x_{3})=\langle f_{2},A_{\lambda}(f_{1})\cdot A_{\lambda}(f_{3})\rangle.\] Assume that \(q_{1},q_{3}\) are such that \(\frac{1}{q_{1}}+\frac{1}{q_{3}}+\frac{1}{p_{2}}=1\) and (\ref{spheq}) holds for \((\frac{1}{p_{1}},\frac{1}{q_{1}})\) and \((\frac{1}{p_{3}},\frac{1}{q_{3}})\). Through an application of Cauchy-Schwarz, we get that
    \begin{align*}
        \Ld_{P_{2}}(f_{1},f_{2},f_{3})&\approx\langle f_{2},A_{\lambda}(f_{1})\cdot A_{\lambda}(f_{3})\rangle\\
        & \leq \|f_{2}\|_{p_{2}}\|A_{\lambda}(f_{1})\|_{q_{1}}\|A_{\lambda}(f_{3})\|_{q_{3}}\\
        &\lesssim \|f_{1}\|_{p_{1}}\|f_{2}\|_{p_{2}}\|f_{3}\|_{p_{3}}\lambda^{\frac{d}{2}\left(1-\frac{1}{p_{1}}-\frac{1}{p_{2}}-\frac{1}{p_{3}}\right)}
    \end{align*}
    Thus, (\ref{2chainineq}) holds as long as such \(q_{1}\) and \(q_{3}\) exist.\\
    To find if there exist such \(q_{1}\) and \(q_{3}\), we consider four cases, depending on \(\frac{1}{p_{1}}\) and \(\frac{1}{p_{3}}\):
    \begin{itemize}
        \item Case 1: \(\frac{1}{p_{1}},\frac{1}{p_{3}}\leq\frac{d-1}{d+1}\).
        \item Case 2: \(\frac{1}{p_{1}}\leq\frac{d-1}{d+1}<\frac{1}{p_{3}}\).
        \item Case 3: \(\frac{1}{p_{3}}\leq\frac{d-1}{d+1}<\frac{1}{p_{1}}\).
        \item Case 4: \(\frac{1}{p_{1}},\frac{1}{p_{3}}>\frac{d-1}{d+1}\).
    \end{itemize}

    We provide a detailed proof for Case 1. The proofs for the rest of the cases are similar.\\

    Case 1: \(\frac{1}{p_{1}},\frac{1}{p_{3}}\leq\frac{d-1}{d+1}\).\\
    Let us assume that we are given \(\frac{1}{p_{1}},\frac{1}{p_{2}},\frac{1}{p_{3}}\) with \(\frac{1}{p_{1}},\frac{1}{p_{3}}<\frac{d-1}{d+1}\). We want to choose \(\frac{1}{q_{1}}, \frac{1}{q_{3}}\) such that \(\frac{1}{q_{1}}+\frac{1}{q_{3}}+\frac{1}{p_{2}}=1\), and \(\left(\frac{1}{p_{1}},\frac{1}{q_{1}}\right)\) and \(\left(\frac{1}{p_{3}},\frac{1}{q_{3}}\right)\) are such that (\ref{spheq}) holds. This is guaranteed if \(\frac{1}{p_{1}}>\frac{1}{q_{1}}>(\frac{2}{d-1})\frac{1}{p_{1}}\) and \(\frac{1}{p_{3}}>\frac{1}{q_{3}}>(\frac{2}{d-1})\frac{1}{p_{3}}\). Rewriting \(\frac{1}{q_{3}}=1-\frac{1}{p_{2}}-\frac{1}{q_{1}}\) and rearranging the inequalities gives the following four conditions that \(\frac{1}{q_{1}}\) needs to satisfy:\\
    \begin{itemize}
        \item \(\frac{1}{q_{1}}\leq\frac{1}{p_{1}}\)
        \item \(\frac{1}{q_{1}}>(\frac{2}{d-1})\frac{1}{p_{1}}\)
        \item \(\frac{1}{q_{1}}\geq1-\frac{1}{p_{3}}-\frac{1}{p_{2}}\)
        \item \(\frac{1}{q_{1}}<1-\frac{1}{p_{2}}-(\frac{2}{d-1})\frac{1}{p_{3}}\).
    \end{itemize}

    Since \(\frac{1}{p_{1}}>(\frac{2}{d-1})\frac{1}{p_{1}}\) and \(\frac{1}{p_{3}}>(\frac{2}{d-1})\frac{1}{p_{3}}\), we have that \(1-\frac{1}{p_{2}}-(\frac{2}{d-1})\frac{1}{p_{3}}>1-\frac{1}{p_{3}}-\frac{1}{p_{2}}\). We assume that \(\frac{1}{p_{1}}+\frac{1}{p_{2}}+\frac{1}{p_{3}}>1\), so we have that \(\frac{1}{p_{1}}>1-\frac{1}{p_{2}}-\frac{1}{p_{3}}\). If we add the extra condition that \(\frac{1}{p_{2}}+\frac{2}{d-1}\left(\frac{1}{p_{1}}+\frac{1}{p_{3}}\right)<1\), then \((\frac{2}{d-1})\frac{1}{p_{1}}<1-\frac{1}{p_{2}}-(\frac{2}{d-1})\frac{1}{p_{3}}\). This guarantees that a \(\frac{1}{q_{1}}\) exists that satisfies the desired inequalities.\\

    Thus, in Case 1, if \(\frac{1}{p_{2}}+\frac{2}{d-1}\left(\frac{1}{p_{1}}+\frac{1}{p_{3}}\right)<1\) then (\ref{2chainineq}) holds. \\

    By doing the same process in Cases 2-4, we get that (\ref{2chainineq}) holds when \(\left(\frac{1}{p_{1}},\frac{1}{p_{2}},\frac{1}{p_{3}}\right)\) is in the region bounded by the planes listed in Remark \ref{2chainrmk}, which is the same region as the convex hull of the points listed in Theorem \ref{2chain}.

\end{proof} 

\subsection{Exponent sharpness} 
    To see that this exponent is sharp, take \(f_{1}=f_{2}=f_{3}=B_{\lambda}\) where \(B_{\lambda}\) is an indicator function for a ball of radius \(\lambda^{\frac{1}{2}}\). Note that \(A_{\lambda}B_{\lambda}\approx B_{\lambda}\).\\
    Then, \[\|f_{1}\|_{p_{1}}\|f_{2}\|_{p_{2}}\|f_{3}\|_{p_{3}}\approx\lambda^{\frac{d}{2}(\frac{1}{p_{1}}+\frac{1}{p_{2}}+\frac{1}{p_{3}})},\]
    and
    \begin{align*}
        \Ld_{P_{2}}(f_{1},f_{2},f_{3})&\approx\frac{1}{|S_{\lambda}|^{2}}\sum_{x_{1},x_{2},x_{3}\in\Z^{d}}B_{\lambda}(x_{1})B_{\lambda}(x_{2})B_{\lambda}(x_{3})S_{\lambda}(x_{1}-x_{2})S_{\lambda}(x_{2}-x_{3})\\
        &=\sum_{x_{2}\in\Z^{d}}B_{\lambda}(x_{2})A_{\lambda}B_{\lambda}(x_{2})\\
        &\approx \sum_{x_{2}\in\Z^{d}}B_{\lambda}(x_{2})\\
        &\approx \lambda^{\frac{d}{2}}\\
        &=\lambda^{\frac{d}{2}(\frac{1}{p_{1}}+\frac{1}{p_{2}}+\frac{1}{p_{3}})+\frac{d}{2}(1-\frac{1}{p_{1}}-\frac{1}{p_{2}}-\frac{1}{p_{3}})}\\
        &\approx\|f_{1}\|_{p_{1}}\|f_{2}\|_{p_{2}}\|f_{3}\|_{p_{3}}\lambda^{\frac{d}{2}(1-\frac{1}{p_{1}}-\frac{1}{p_{2}}-\frac{1}{p_{3}})}
    \end{align*}

    Thus, we can see that the exponent, \(\frac{d}{2}(1-\frac{1}{p_{1}}-\frac{1}{p_{2}}-\frac{1}{p_{3}})\), on \(\lambda\) is the best that we can guarantee.\\

\subsection{Necessary conditions.}

    By using test functions, we can find some necessary conditions for (\ref{2chainineq}) to hold. Take \(f_{1}(x)=f_{3}(x)=S_{\lambda}(x)\), and \(f_{2}(x)=1_{\{0\}}(x)\). Then
    \[\|f_{1}|\|_{p_{1}}\|f_{2}\|_{p_{2}}\|f_{3}\|_{p_{3}}\approx\lambda^{\frac{d-2}{2}(\frac{1}{p_{1}}+\frac{1}{p_{3}})},\]
    and 
    \begin{align*}
        \Ld_{P_{2}}(f_{1},f_{2},f_{3})&\approx\frac{1}{|S_{\lambda}|^{2}}\sum_{x_{1},x_{2},x_{3}\in\Z^{d}}1_{\{0\}}(x_{2})S_{\lambda}(x_{1})S_{\lambda}(x_{3})S_{\lambda}(x_{1}-x_{2})S_{\lambda}(x_{2}-x_{3})\\
        &=\frac{1}{|S_{\lambda}|^{2}}\sum_{x_{1},x_{3}\in\Z^{d}}S_{\lambda}(x_{1})S_{\lambda}(x_{3})\\
        &\approx 1.
    \end{align*}
    Thus, (\ref{2chainineq}) holds if \[0\leq \frac{d}{2}\left(1-\frac{1}{p_{1}}-\frac{1}{p_{2}}-\frac{1}{p_{3}}\right)+\frac{d-2}{2}\left(\frac{1}{p_{1}}+\frac{1}{p_{3}}\right),\] which is equivalent to saying that
    \[\left(\frac{1}{p_{1}}+\frac{1}{p_{3}}\right)\frac{2}{d}+\frac{1}{p_{2}}\leq 1.\]
    While not being an exact match, this is very similar to the equation for the plane that forms the boundary of the region in Theorem \ref{2chain} when \(\frac{1}{p_{3}},\frac{1}{p_{1}}\leq\frac{d-1}{d+1}\), which requires that \[\left(\frac{1}{p_{1}}+\frac{1}{p_{3}}\right)\frac{2}{d-1}+\frac{1}{p_{2}}\leq 1.\]
    These planes approach each other asymptotically as \(d\to\infty\).\\

    Similarly, take \(f_{1}(x)=f_{3}(x)=1_{\{0\}}(x)\) and \(f_{2}(x)=S_{\lambda}(x)\). Then, 
    \[\|f_{1}\|_{p_{1}}\|f_{2}\|_{p_{2}}\|f_{3}\|_{p_{3}}\approx\lambda^{\frac{d-2}{2}(\frac{1}{p_{2}})},\]
    and 
    \begin{align*}
        \Ld_{P_{2}}(f_{1},f_{2},f_{3})&\approx\frac{1}{|S_{\lambda}|^{2}}\sum_{x_{1},x_{2},x_{3}\in\Z^{d}}1_{\{0\}}(x_{1})1_{\{0\}}(x_{3})S_{\lambda}(x_{2})S_{\lambda}(x_{1}-x_{2})S_{\lambda}(x_{2}-x_{3})\\
        &=\frac{1}{|S_{\lambda}|^{2}}\sum_{x_{2}\in\Z^{d}}S_{\lambda}(x_{2})\\
        &\approx \frac{1}{|S_{\lambda}|}\approx\lambda^{-\frac{d-2}{2}}.
    \end{align*}
    In this case, (\ref{2chainineq}) holds if 
    \[-\frac{d-2}{2}<\frac{d}{2}\left(1-\frac{1}{p_{1}}-\frac{1}{p_{2}}-\frac{1}{p_{3}}\right)+\frac{d-2}{2}\left(\frac{1}{p_{2}}\right),\]
    or equivalently,
    \[\frac{1}{p_{2}}(\frac{2}{d})+\frac{1}{p_{1}}+\frac{1}{p_{3}}<2-\frac{2}{d}.\]
    This is again very similar to the equation for one of the planes that forms the boundary of the region in Theorem \ref{2chain}. Namely, when \(\frac{1}{p_{1}},\frac{1}{p_{3}}>\frac{d-1}{d+1}\), the plane bounding the region is \[\frac{1}{p_{1}}+\frac{1}{p_{3}}+\frac{2}{d-1}\left(\frac{1}{p_{2}}\right)<2-\frac{2}{d-1}.\]
    While again not being equal, the two planes approach each other asymptotically as \(d\to\infty\).


\section{The Triangle, \(K_{3}\)} \label{tri section}

For \(d\geq7\), define \[\Ld_{K_{3}}(f_{1},f_{2},f_{3})=\frac{1}{N_{K_{3}}}\sum_{x_{1},x_{2},x_{3}}f_{1}(x_{1})f_{2}(x_{2})f_{3}(x_{3})S_{\lambda}(x_{1}-x_{2})S_{\lambda}(x_{2}-x_{3})S_{\lambda}(x_{3}-x_{1})\]

\begin{theorem}\label{trithm}
    If \(d\geq 7\) and \(1<p_{1},p_{2},p_{3}\) with \(\frac{1}{p_{1}}+\frac{1}{p_{2}}+\frac{1}{p_{3}}>1\), then there exist constants \(C_{p_{1},p_{2},p_{3}}\) such that for all \(\lambda\in\mathcal{R}_{K_{3}}\) we have the \(\ell^{p}\) improving inequality
    \begin{equation}\label{tri ineq}\Ld_{K_{3}}(f_{1},f_{2},f_{3})\leq C_{p_{1},p_{2},p_{3}}\lambda^{\frac{d}{2}(1-\frac{1}{p_{1}}-\frac{1}{p_{2}}-\frac{1}{p_{3}})}\|f_{1}\|_{p_{1}}\|f_{2}\|_{p_{2}}\|f_{3}\|_{p_{3}}\end{equation}
    provided \(\left(\frac{1}{p_{1}},\frac{1}{p_{2}},\frac{1}{p_{3}}\right)\) is in the convex hull of the points:
\begin{multicols}{3}
    \begin{itemize}
        \item \((1,0,0)\)
        \item \((0,1,0)\)
        \columnbreak
        \item \((0,0,1)\)
        \item \((\frac{d-1}{d+1},\frac{d-1}{d+1},0)\)
        \columnbreak
        \item \((\frac{d-1}{d+1},0,\frac{d-1}{d+1})\)
        \item \((0,\frac{d-1}{d+1},\frac{d-1}{d+1})\)
    \end{itemize}
\end{multicols}
\end{theorem}

\begin{remark}\label{Trirmk}
    The region in Theorem \ref{trithm} is the set of all points \(\left(\frac{1}{p_{1}},\frac{1}{p_{2}},\frac{1}{p_{3}}\right)\) such that the following hold:
    \begin{enumerate}
    \item \(\frac{1}{p_{1}}+\frac{1}{p_{2}}+\frac{1}{p_{3}}>1\)
    \item \(\left(\frac{1}{p_{1}}+\frac{1}{p_{2}}\right)\frac{2}{d-1}+\frac{1}{p_{3}}<1\) when \(\frac{1}{p_{1}}+\frac{1}{p_{2}}<\frac{d-1}{d+1}\)
    \item \(\left(\frac{1}{p_{1}}+\frac{1}{p_{3}}\right)\frac{2}{d-1}+\frac{1}{p_{2}}<1\) when \(\frac{1}{p_{2}}>\frac{d-1}{d+1}\)
    \item \(\left(\frac{1}{p_{3}}+\frac{1}{p_{2}}\right)\frac{2}{d-1}+\frac{1}{p_{1}}<1\) when \(\frac{1}{p_{1}}>\frac{d-1}{d+1}\)
    \item \(\frac{1}{p_{1}}+\frac{1}{p_{2}}+\frac{1}{p_{3}}<\frac{2(d-1)}{d+1}\) when \(\frac{1}{p_{1}},\frac{1}{p_{2}}<\frac{d-1}{d+1},\) but \(\frac{1}{p_{1}}+\frac{1}{p_{2}}>\frac{d-1}{d+1}\).
    \end{enumerate}
    This is the region shown in Figure \ref{fig:tribounds}.
    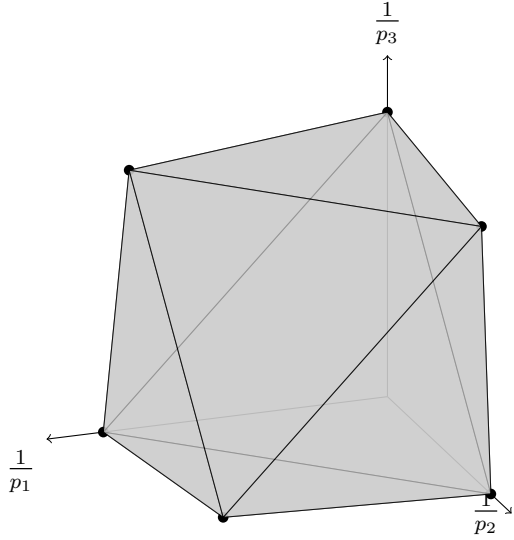
\begin{figure}[h]
\centering
\begin{tikzpicture}[tdplot_main_coords,scale=4]
\coordinate (A) at (0,1,0);
\coordinate (B) at (1,0,0);
\coordinate (C) at (0,0,1);
\coordinate (D) at (0,10/11,10/11);
\coordinate (E) at (10/11,10/11,0);
\coordinate (F) at (10/11,0,10/11);

\draw[->] (0,0,0) -- (1.2,0,0) node[below left]{$\frac{1}{p_{1}}$};
\draw[->] (0,0,0) -- (0,1.2,0) node[left]{$\frac{1}{p_{2}}$};
\draw[->] (0,0,0) -- (0,0,1.2) node[above]{$\frac{1}{p_{3}}$};

\foreach \p in {A,B,C,D,E,F}
   \fill (\p) circle[radius=0.5pt];
    \filldraw[fill=gray!40,opacity=0.7] (A) -- (B) -- (C) -- cycle;
    \filldraw[fill=gray!40,opacity=0.7] (F) -- (B) -- (C) -- cycle;
    \filldraw[fill=gray!40,opacity=0.7] (A) -- (C) -- (D) -- cycle;
    \filldraw[fill=gray!40,opacity=0.7] (A) -- (B) -- (E) -- cycle;
    \filldraw[fill=gray!40,opacity=0.7] (C) -- (D) -- (F) -- cycle;
    \filldraw[fill=gray!40,opacity=0.7] (B) -- (E) -- (F) -- cycle;
    \filldraw[fill=gray!40,opacity=0.7] (A) -- (E) -- (D) -- cycle;
    \filldraw[fill=gray!40,opacity=0.7] (D) -- (E) -- (F) -- cycle;

\end{tikzpicture}
\caption{Region of \(\ell^{p}\) improvement for \(\Ld_{K_{3}}\).}
\label{fig:tribounds}
\end{figure}
\end{remark}

\begin{proof}[Proof of Theorem \ref{trithm}]

Note that

\begin{align*}
    \Ld_{K_{3}}(f_{1},f_{2},f_{3})\approx&\frac{1}{|S_{\lambda}||S_{\lambda}^{d-2}|}\sum_{x_{1},x_{2},x_{3}\in\Z^{d}}f_{1}(x_{1})f_{2}(x_{2})f_{3}(x_{3})S_{\lambda}(x_{1}-x_{2})S_{\lambda}(x_{2}-x_{3})S_{\lambda}(x_{3}-x_{1})\\
    =&\frac{1}{|S_{\lambda}||S_{\lambda}^{d-2}|}\sum_{x_{1},x_{2}\in\Z^{d}}f_{1}(x_{1})f_{2}(x_{2})S_{\lambda}(x_{1}-x_{2})\sum_{x_{3}\in\Z^{d}}f_{3}(x_{3})S_{\lambda}(x_{2}-x_{3})S_{\lambda}(x_{3}-x_{1})\\
    \leq& \frac{1}{|S_{\lambda}||S_{\lambda}^{d-2}|}\left(\sum_{x_{1},x_{2}\in\Z^{d}}f_{1}(x_{1})f_{2}(x_{2})S_{\lambda}(x_{1}-x_{2})\right)\\
    &\,\,\,\,\,\,\,\,\,\,\,\,\,\,\cdot\left(\max_{|x_{1}-x_{2}|=\lambda}\sum_{x_{3}\in\Z^{d}}f_{3}(x_{3})S_{\lambda}(x_{2}-x_{3})S_{\lambda}(x_{3}-x_{1})\right)\\
    \lesssim& \Ld_{P_{1}}(f_{1},f_{2})\|f_{3}\|_{\infty}\left(\max_{|x_{1}-x_{2}|=\lambda}\frac{1}{|S_{\lambda}^{d-2}|}\sum_{x_{3}\in\Z^{d}}S_{\lambda}(x_{2}-x_{3})S_{\lambda}(x_{3}-x_{1})\right)\\
    \lesssim &\Ld_{P_{1}}(f_{1},f_{2})\|f_{3}\|_{\infty}
\end{align*}

Thus, we have that \[\Ld_{K_{3}}(f_{1},f_{2},f_{3})\lesssim \|f_{1}\|_{p_{1}}\|f_{2}\|_{p_{2}}\|f_{3}\|_{\infty}\lambda^{\frac{d}{2}(1-\frac{1}{p_{1}}-\frac{1}{p_{2}})}\] when \(\left(\frac{1}{p_{1}},\frac{1}{p_{2}}\right)\) satisfies the conditions in Theorem \ref{Lthm} for \(\ell^{p}\) improving with the best exponent for \(\Ld_{P_{1}}(f_{1},f_{2})\). Using these conditions, the symmetry of the triangle, and interpolation, we get that \[\Ld_{K_{3}}(f_{1},f_{2},f_{3})\lesssim \|f_{1}\|_{p_{1}}\|f_{2}\|_{p_{2}}\|f_{3}\|_{p_{3}}\lambda^{\frac{d}{2}(1-\frac{1}{p_{1}}-\frac{1}{p_{2}}-\frac{1}{p_{3}})}\] holds when \((\frac{1}{p_{1}},\frac{1}{p_{2}},\frac{1}{p_{3}})\) is in the convex hull of the points 
\begin{multicols}{3}
    \begin{itemize}
        \item \((1,0,0)\)
        \item \((0,1,0)\)
        \columnbreak
        \item \((0,0,1)\)
        \item \((\frac{d-1}{d+1},\frac{d-1}{d+1},0)\)
        \columnbreak
        \item \((\frac{d-1}{d+1},0,\frac{d-1}{d+1})\)
        \item \((0,\frac{d-1}{d+1},\frac{d-1}{d+1})\).
    \end{itemize}
\end{multicols}

\end{proof}

\subsection{Exponent Sharpness}

Let \(B_{\lambda}(x)\) be an indicator for a ball of radius \(\lambda^{1/2}\). Given generic choices of \(y,z\in\Z\), and given that \(d\geq7\), then \( \displaystyle\sum_{x\in\Z}B_{\lambda}(x)S_{\lambda}(x-y)S_{\lambda}(x-z)\) counts the number of points \(x\in\Z\) where \(|x|<\lambda^{1/2}\), \(|x-y|=\lambda^{1/2}\), and \(|x-z|=\lambda^{1/2}\). This will be \(0\) if \(y>2\lambda^{1/2}\) or \(z>2\lambda^{1/2}\), and approximately \(|S_{\lambda}^{d-2}|\) otherwise. Thus, \[ \sum_{x\in\Z}B_{\lambda}(x)S_{\lambda}(x-y)S_{\lambda}(x-z)\approx |S_{\lambda}^{d-2}|B_{2\lambda}(y)B_{2\lambda}(z)\approx|S_{\lambda}^{d-2}|B_{\lambda}(y)B_{\lambda}(z).\]

To show the sharpness of the exponent on \(\lambda\), let \(f_{1}(x)=f_{2}(x)=f_{3}(x)=B_{\lambda}(x)\). Then \(\|f_{1}\|_{p_{1}}\|f_{2}\|_{p_{2}}\|f_{3}\|_{p_{3}}\approx \lambda^{\frac{d}{2}(\frac{1}{p_{1}}+\frac{1}{p_{2}}+\frac{1}{p_{3}})}\).\\
Now,
\begin{align*}
    \Ld_{K_{3}}(f_{1},f_{2},f_{3})&\approx\frac{1}{|S_{\lambda}||S_{\lambda}^{d-2}|}\sum_{x_{1},x_{2}, x_{3}\in\Z^{d}}B_{\lambda}(x_{1})B_{\lambda}(x_{2})B_{\lambda}(x_{3})S_{\lambda}(x_{1}-x_{2})S_{\lambda}(x_{2}-x_{3})S_{\lambda}(x_{3}-x_{1})\\
    &=\frac{1}{|S_{\lambda}||S_{\lambda}^{d-2}|}\sum_{x_{1},x_{2}\in\Z^{d}}B_{\lambda}(x_{1})B_{\lambda}(x_{2})S_{\lambda}(x_{1}-x_{2})\sum_{x_{3}\in\Z^{d}}B_{\lambda}(x_{3})S_{\lambda}(x_{2}-x_{3})S_{\lambda}(x_{3}-x_{1})\\
    &\approx \frac{1}{|S_{\lambda}|}\sum_{x_{1},x_{2}\in\Z^{d}}B_{\lambda}(x_{1})B_{\lambda}(x_{2})S_{\lambda}(x_{1}-x_{2})\\
    &=\frac{1}{|S_{\lambda}|}\sum_{x_{1}\in\Z^{d}}B_{R}(x_{1})\sum_{x_{2}\in\Z^{d}}B_{\lambda}(x_{2})S_{\lambda}(x_{1}-x_{2})\\
    &=\sum_{x_{1}\in\Z^{d}}B_{\lambda}(x_{1})A_{\lambda}(B_{\lambda})(x_{1})\\
    &\approx \sum_{x_{1}\in\Z^{d}}B_{\lambda}(x_{1})\\
    &\approx \lambda^{\frac{d}{2}}\\
    &=\lambda^{\frac{d}{2}(1-\frac{1}{p_{1}}-\frac{1}{p_{2}}-\frac{1}{p_{3}})}\lambda^{\frac{d}{2}(\frac{1}{p_{1}}+\frac{1}{p_{2}}+\frac{1}{p_{3}})}\\
    &=\lambda^{\frac{d}{2}(1-\frac{1}{p_{1}}-\frac{1}{p_{2}}-\frac{1}{p_{3}})}\|f_{1}\|_{p_{1}}\|f_{2}\|_{p_{2}}\|f_{3}\|_{p_{3}}
\end{align*}
Thus, \(\frac{d}{2}(1-\frac{1}{p_{1}}-\frac{1}{p_{2}}-\frac{1}{p_{3}})\) is the best exponent we can expect for \(\lambda\).
\\

\subsection{Necessary conditions}

To determine the necessary conditions to get an exponent of \(\frac{d}{2}(\frac{1}{p_{1}}-\frac{1}{p_{2}}-\frac{1}{p_{3}})\) on \(\lambda\), we take \(f_{1}(x)=1_{\{0\}}(x)\) and \(f_{2}(x)=f_{3}(x)=S_{\lambda}(x)\). Then
\[\|f_{1}\|_{p_{1}}\|f_{2}\|_{p_{2}}\|f_{3}\|_{p_{3}}\approx\lambda^{\frac{d-2}{2}(\frac{1}{p_{2}}+\frac{1}{p_{3}})},\]
and
\begin{align*}
    \Ld_{K_{3}}(f_{1},f_{2},f_{3})&\approx\frac{1}{|S_{\lambda}||S_{\lambda}^{d-2}|}\sum_{x_{1},x_{2},x_{3}\in\Z^{d}}1_{\{0\}}(x_{1})S_{\lambda}(x_{2})S_{\lambda}(x_{3})S_{\lambda}(x_{1}-x_{2})S_{\lambda}(x_{2}-x_{3})S_{\lambda}(x_{3}-x_{1})\\
    &=\frac{1}{|S_{\lambda}||S_{\lambda}^{d-2}|}\sum_{x_{2},x_{3}\in\Z^{d}}S_{\lambda}(x_{2})S_{\lambda}(x_{3})S_{\lambda}(x_{1}-x_{2})\\
    &\approx\frac{|S_{\lambda}|}{|S_{\lambda}^{d-2}|}\approx\lambda^{-1}.
\end{align*}
Thus, (\ref{tri ineq}) holds if 
\[-1\leq\frac{d}{2}\left(1-\frac{1}{p_{1}}-\frac{1}{p_{2}}-\frac{1}{p_{3}}\right),\]
or equivalently, if 
\[\frac{1}{p_{1}}+\frac{2}{d}\left(\frac{1}{p_{2}}+\frac{1}{p_{3}}\right)\leq 1+\frac{2}{d}.\]
This is similar to the equation for one of the planes the forms the boundary of the region in Theorem \ref{trithm}. When \(\frac{1}{p_{1}}>\frac{d-1}{d+1}\), the plane bounding the region is \[\frac{1}{p_{1}}+\frac{2}{d-1}\left(\frac{1}{p_{2}}+\frac{1}{p_{3}}\right)<1.\]
While not being equal, the two planes approach each other asymptotically as \(d\to\infty\). \\
By switching the roles of \(f_{1}\), \(f_{2}\), and \(f_{3}\), we can also get the necessary conditions 
\[\frac{1}{p_{2}}+\frac{2}{d}\left(\frac{1}{p_{1}}+\frac{1}{p_{3}}\right)\leq 1+\frac{2}{d},\]
and \[\frac{1}{p_{3}}+\frac{2}{d}\left(\frac{1}{p_{2}}+\frac{1}{p_{1}}\right)\leq 1+\frac{2}{d},\]
which, as \(d\to\infty\), asymptotically approach the sufficient conditions 
\[\frac{1}{p_{2}}+\frac{2}{d-1}\left(\frac{1}{p_{1}}+\frac{1}{p_{3}}\right)<1\] when \(\frac{1}{p_{2}}>\frac{d-1}{d+1}\) and \[\frac{1}{p_{3}}+\frac{2}{d-1}\left(\frac{1}{p_{2}}+\frac{1}{p_{1}}\right)<1\] when \(\frac{1}{p_{1}}+\frac{1}{p_{2}}<\frac{d-1}{d+1}\) respectively.\\

\subsection{Improvement on previous results}\label{tri improvement}

   For the triangle averaging operator, 
   \[A_{tri}(f,g)(x)=\frac{1}{\lambda^{d-3}}\sum_{y,z\in\Z^{d}}f(y)g(z)S_{\lambda}(x-y)S_{\lambda}(y-z)S_{\lambda}(z-x),\]
   Anderson, Kumchev, and Palsson \cite{AKP24} found that for \(d\geq 7\) and \(\frac{d+1}{d-1}<p<\frac{2d}{d+2}\), the triangle averaging operator satisfies the bound \[\|A_{tri}(f,g)\|_{1}\lesssim\lambda^{1+\frac{d}{2}-\frac{d}{p}}\\f\|_{p}\|g\|_{p}.\] If we apply Theorem \ref{trithm} in this setting, letting \(p_{1}=p_{2}=p\) and \(p_{3}=\infty\), we get a decay term of \(\lambda^{\frac{d}{2}-\frac{d}{p}}\), which improves on the results in \cite{AKP24}. Additionally, we have a much wider range for the choices of \(p_{1},p_{2},p_{3}\). 

\section{The Tetrahedron, \(K_{4}\)} \label{tetr section}

For \(d\geq 9\), we define

    \[\Ld_{K_{4}}(f_{1},f_{2},f_{3},f_{4})=\frac{1}{N_{K_{4}}}\sum_{x_{1},...x_{4}\in\Z^{d}}\prod_{i=1}^{4}f_{i}(x_{i})S_{K_{4}}(x_{1},x_{2},x_{3},x_{4})\]
Where \[S_{K_{4}}(x_{1},x_{2},x_{3},x_{4})=S_{\lambda}(x_{1}-x_{2})S_{\lambda}(x_{2}-x_{3})S_{\lambda}(x_{3}-x_{4})S_{\lambda}(x_{4}-x_{1})S_{\lambda}(x_{1}-x_{3})S_{\lambda}(x_{4}-x_{2}).\]
We have the following theorem:
\begin{theorem}\label{tetrtheorem}
    If \(d\geq 9\) and \(1<p_{1},p_{2},p_{3},p_{4}<\infty\) with \(\frac{1}{p_{1}}+\frac{1}{p_{2}}+\frac{1}{p_{3}}+\frac{1}{p_{4}}>1\), then there exist constants \(C_{p_{1},p_{2},p_{3},p_{4}}\) such that for all \(\lambda\in\mathcal{R}_{K_{4}}\) we have the \(\ell^{p}\) improving inequality
    \begin{equation}\label{Tetreq}\Ld_{K_{4}}(f_{1},f_{2},f_{3},f_{4})\leq C_{p_{1},p_{2},p_{3},p_{4}}\lambda^{\frac{d}{2}(1-\frac{1}{p_{1}}-\frac{1}{p_{2}}-\frac{1}{p_{3}}-\frac{1}{p_{4}})}\|f_{1}\|_{p_{1}}\|f_{2}\|_{p_{2}}\|f_{3}\|_{p_{3}}\|f_{4}\|_{p_{4}}\end{equation}
    provided \(\left(\frac{1}{p_{1}},\frac{1}{p_{2}},\frac{1}{p_{3}},\frac{1}{p_{4}}\right)\) is in the convex hull of the points:
\begin{multicols}{3}
    \begin{itemize}
            \item \((1,0,0,0)\)
            \item \((0,1,0,0)\)
            \item \((0,0,1,0)\)
            \item \((0,0,0,1)\)
            \columnbreak
            \item \((0, \frac{d-1}{d+1}, 0, \frac{d-1}{d+1})\)
            \item \((\frac{d-1}{d+1},\frac{d-1}{d+1},0,0)\)
            \item \((0,\frac{d-1}{d+1},\frac{d-1}{d+1},0)\)
            \columnbreak
            \item \((0,0,\frac{d-1}{d+1},\frac{d-1}{d+1})\)
            \item \((\frac{d-1}{d+1},0,0,\frac{d-1}{d+1})\)
            \item \((\frac{d-1}{d+1},0,\frac{d-1}{d+1},0)\).
        \end{itemize}
\end{multicols}
\end{theorem}
\begin{proof}[Proof of Theorem \ref{tetrtheorem}]

We have that 
\begin{align*}
    \Ld_{K_{4}}(f_{1},f_{2},f_{3},f_{4})&\lesssim \frac{\|f_{4}\|_{\infty}}{|S_{\lambda}||S_{\lambda}^{d-2}||S_{\lambda}^{d-3}|}\sum_{x_{1},x_{2},x_{3}\in\Z^{d}}\prod_{i=1}^{3}f_{i}(x_{i})S_{\lambda}(x_{1}-x_{2})S_{\lambda}(x_{2}-x_{3})S_{\lambda}(x_{1}-x_{3})\\
    &\,\,\,\,\,\,\,\,\,\,\,\,\,\,\,\,\cdot \sum_{x_{4}\in\Z^{d}}S_{\lambda}(x_{3}-x_{4})S_{\lambda}(x_{4}-x_{1})S_{\lambda}(x_{4}-x_{2})\\
    &\lesssim \frac{\|f_{4}\|_{\infty}}{|S_{\lambda}||S_{\lambda}^{d-2}|}\sum_{x_{1},x_{2},x_{3}\in\Z^{d}}\prod_{i=1}^{3}f_{i}(x_{i})S_{\lambda}(x_{1}-x_{2})S_{\lambda}(x_{2}-x_{3})S_{\lambda}(x_{1}-x_{3})\\
    &\approx\|f_{4}\|_{\infty} \Ld_{K_{3}}(f_{1},f_{2},f_{3})
\end{align*}

By symmetry, we also get that \(\Ld_{K_{4}}(f_{1},f_{2},f_{3},f_{4})\lesssim \|f_{1}\|_{\infty} \Ld_{K_{3}}(f_{2},f_{3},f_{4})\), \(\Ld_{K_{4}}(f_{1},f_{2},f_{3},f_{4})\lesssim \|f_{2}\|_{\infty}  \Ld_{K_{3}}(f_{3},f_{4},f_{1})\), and \( \Ld_{K_{4}}(f_{1},f_{2},f_{3},f_{4})\lesssim \|f_{3}\|_{\infty}  \Ld_{K_{3}}(f_{4},f_{1},f_{2})\). Using this symmetry, interpolation, and Theorem \ref{trithm}, we get that \(\Ld_{K_{4}}(f_{1},f_{2},f_{3},f_{4})\lesssim \lambda^{\frac{d}{2}(1-\frac{1}{p_{1}}-\frac{1}{p_{2}}-\frac{1}{p_{3}}-\frac{1}{p_{4}})}\|f_{1}\|_{p_{1}}\|f_{2}\|_{p_{2}}\|f_{3}\|_{p_{3}}\|f_{4}\|_{p_{4}}\) when \((\frac{1}{p_{1}},\frac{1}{p_{2}},\frac{1}{p_{3}},\frac{1}{p_{4}})\) is in the convex hull of the points: 
    \begin{multicols}{3}
        \begin{itemize}
            \item \((1,0,0,0)\)
            \item \((0,1,0,0)\)
            \item \((0,0,1,0)\)
            \item \((0,0,0,1)\)
            \columnbreak
            \item \((0, \frac{d-1}{d+1}, 0, \frac{d-1}{d+1})\)
            \item \((\frac{d-1}{d+1},\frac{d-1}{d+1},0,0)\)
            \item \((0,\frac{d-1}{d+1},\frac{d-1}{d+1},0)\)
            \columnbreak
            \item \((0,0,\frac{d-1}{d+1},\frac{d-1}{d+1})\)
            \item \((\frac{d-1}{d+1},0,0,\frac{d-1}{d+1})\)
            \item \((\frac{d-1}{d+1},0,\frac{d-1}{d+1},0)\).
        \end{itemize}
    \end{multicols}
\end{proof}

\subsection{Exponent sharpness}

To show the sharpness of the exponent on \(\lambda\), take \(f_{1}(x)=f_{2}(x)=f_{3}(x)=f_{4}(x)=B_{\lambda}(x)\). Then \[\|f_{1}\|_{p_{1}}\|f_{2}\|_{p_{2}}\|f_{3}\|_{p_{3}}\|f_{4}\|_{p_{4}}\approx\lambda^{\frac{d}{2}(\frac{1}{p_{1}}+\frac{1}{p_{2}}+\frac{1}{p_{3}}+\frac{1}{p_{4}})}.\]
Additionally,
    \[\Ld_{K_{4}}(f_{1},f_{2},f_{3},f_{4})\approx\frac{1}{|S_{\lambda}||S_{\lambda}^{d-2}||S_{\lambda}^{d-3}|}\sum_{x_{1},...x_{4}\in\Z^{d}}\prod_{i=1}^{4}B_{\lambda}(x_{i})S_{K}(x_{1},x_{2},x_{3},x_{4}).\]
Through the same reasoning as we used for \(\Ld_{K_{3}}\), we have that \[\sum_{x_{1},x_{2},x_{3}\in\Z^{d}}B_{\lambda}(x_{4})S_{\lambda}(x_{1}-x_{4})S_{\lambda}(x_{2}-x_{4})S_{\lambda}(x_{3}-x_{4})\approx |S_{\lambda}^{d-3}|B_{\lambda}(x_{1})B_{\lambda}(x_{2})B_{\lambda}(x_{3}).\]
Thus, 
\begin{align*}
    \Ld_{K_{4}}(f_{1},f_{2},f_{3},f_{4})\approx&\frac{1}{|S_{\lambda}||S_{\lambda}^{d-2}||S_{\lambda}^{d-3}|}\sum_{x_{1},x_{2},x_{3}\in\Z^{d}}\prod_{i=1}^{3}B_{\lambda}(x_{i})S_{\lambda}(x_{1}-x_{2})S_{\lambda}(x_{2}-x_{3})S_{\lambda}(x_{1}-x_{3})\\
    &\,\,\,\,\,\,\,\,\,\,\,\,\,\cdot\sum_{x_{4}\in\Z^{d}}B_{\lambda}(x_{4})S_{\lambda}(x_{1}-x_{4})S_{\lambda}(x_{2}-x_{4})S_{\lambda}(x_{3}-x_{4})\\
    \approx &\frac{1}{|S_{\lambda}||S_{\lambda}^{d-2}|}\sum_{x_{1},x_{2},x_{3}\in\Z^{d}}\prod_{i=1}^{3}B_{\lambda}(x_{i})S_{\lambda}(x_{1}-x_{2})S_{\lambda}(x_{2}-x_{3})S_{\lambda}(x_{1}-x_{3})\\
    =&\frac{1}{|S_{\lambda}||S_{\lambda}^{d-2}|}\sum_{x_{1},x_{2}\in\Z^{d}}B_{\lambda}(x_{1})B_{\lambda}(x_{2})S_{\lambda}(x_{1}-x_{2})\sum_{x_{3}\in\Z^{d}}B_{\lambda}(x_{3})S_{\lambda}(x_{1}-x_{3})S_{\lambda}(x_{2}-x_{3})\\
    \approx& \frac{1}{|S_{\lambda}|}\sum_{x_{1},x_{2}\in\Z^{d}}B_{\lambda}(x_{1})B_{\lambda}(x_{2})S_{\lambda}(x_{1}-x_{2})\\
    =& \sum_{x_{1}\in\Z^{d}}B_{\lambda}(x_{1})A_{\lambda}B_{\lambda}(x_{2})\approx  \sum_{x_{1}\in\Z^{d}}B_{\lambda}(x_{1})
    \\\approx&\lambda^{\frac{d}{2}}\approx\lambda^{\frac{d}{2}(1-\frac{1}{p_{1}}-\frac{1}{p_{2}}-\frac{1}{p_{3}}-\frac{1}{p_{4}})}\|f_{1}\|_{p_{1}}\|f_{2}\|_{p_{2}}\|f_{3}\|_{p_{3}}\|f_{4}\|_{p_{4}}
\end{align*}

Thus, the exponent on \(\lambda\) of \(\frac{d}{2}(1-\frac{1}{p_{1}}-\frac{1}{p_{2}}-\frac{1}{p_{3}}-\frac{1}{p_{4}})\) is the best that we can guarantee.

\subsection{Improvement on previous results}\label{tetr improvement}
For the averaging operator based on the tetrahedron, 
\[A_{K_{4}}(f,g,h)(x)=\frac{1}{\lambda^{\frac{3d-12}{2}}}\sum_{x_{1},x_{2},x_{3}\in\Z^{d}}f(x_{1})g(x_{2})h(x_{3})S_{K_{4}}(x,x_{1},x_{2},x_{3}),\]
Anderson, Kumchev, and Palsson \cite{AKP24} gave two propositions. First, they have that if \(d\geq 9\), \(2<p_{1}\leq p_{2}\leq p_{3}\), and \(\frac{1}{p_{1}}+\frac{1}{p_{2}}+\frac{1}{p_{3}}=1\), then 
\[\|A_{K_{4}}(f_{1},f_{2},f_{3})\|_{1}\lesssim\lambda^{\frac{6-dj}{2}+d(\sum_{i\leq j}\frac{1}{p_{i}})}\prod_{i=1}^{j}\|f_{i}\|_{p_{i}'}\prod_{i=j+1}^{3}\|f_{i}\|_{p_{i}}\]
where \(j\geq1\) is the largest index for which \(p_{i}<\frac{1}{2}(d+1)\). \\

If \(j=1\), this bound is 
\[\|A_{K_{4}}(f_{1},f_{2},f_{3})\|_{1}\lesssim\lambda^{\frac{6-d}{2}+d(\frac{1}{p_{1}})}\|f_{1}\|_{p_{1}'}\|f_{2}\|_{p_{2}}\|f_{3}\|_{p_{3}}=\lambda^{3+\frac{d}{2}(1-\frac{1}{p_{1}'}+\frac{1}{p_{2}}+\frac{1}{p_{3}})}\|f_{1}\|_{p_{1}'}\|f_{2}\|_{p_{2}}\|f_{3}\|_{p_{3}}.\]
If we apply Theorem \ref{tetrtheorem} in this setting, letting our \(p_{1}\) be their \(p_{1}'\), our \(p_{2}, p_{3}\) be the same as theirs, and \(p_{4}=\infty\), we get a decay term of \(\lambda^{\frac{d}{2}(1-\frac{1}{p_{1}'}-\frac{1}{p_{2}}-\frac{1}{p_{3}})}\), which we can see is strictly better than the improvement in \cite{AKP24}. In the cases of \(j=2\) and \(j=3\), the results are the same.\\

In the next proposition, Anderson, Kumchev, and Palsson have that if \(d\geq 9\), \(\frac{d+1}{d-1}<p_{1}<2<p_{2}\leq p_{3}\), and \(\frac{1}{p_{1}}+\frac{1}{p_{2}}+\frac{1}{p_{3}}=1\), then 
\[\|A_{K_{4}}(f_{1},f_{2},f_{3})\|_{1}\lesssim\lambda^{\frac{6-dj}{2}+\frac{d}{p_{1}'}+d\sum_{1<i\leq j}\frac{1}{p_{i}}}\|f_{1}\|_{p_{1}}\|f_{2}\|_{p_{2}'}\|f_{3}\|_{p_{3}'},\] where \(j\geq1\) is the largest integer such that \(p_{i}<\frac{1}{2}(d+1)\).\\
If \(j=1\), this bound is 
\[\|A_{K_{4}}(f_{1},f_{2},f_{3})\|_{1}\lesssim\lambda^{\frac{6-d}{2}+\frac{d}{p_{1}'}}\|f_{1}\|_{p_{1}}\|f_{2}\|_{p_{2}'}\|f_{3}\|_{p_{3}'}=\lambda^{3+\frac{d}{2}+\frac{d}{2}(1-\frac{1}{p_{1}}-\frac{1}{p_{2}'}-\frac{1}{p_{3}'})}\|f_{1}\|_{p_{1}}\|f_{2}\|_{p_{2}'}\|f_{3}\|_{p_{3}'}.\]

Then, when \(j=2\), the bound is
\[\|A_{K_{4}}(f_{1},f_{2},f_{3})\|_{1}\lesssim\lambda^{\frac{6-2d}{2}+\frac{d}{p_{1}'}+\frac{d}{p_{2}}}\|f_{1}\|_{p_{1}}\|f_{2}\|_{p_{2}'}\|f_{3}\|_{p_{3}'}=\lambda^{3+\frac{d}{2}+\frac{d}{2}(1-\frac{1}{p_{1}}-\frac{1}{p_{2}'}-\frac{1}{p_{3}'})}\|f_{1}\|_{p_{1}}\|f_{2}\|_{p_{2}'}\|f_{3}\|_{p_{3}'},\]
and when \(j=3\) the bound is
\[\|A_{K_{4}}(f_{1},f_{2},f_{3})\|_{1}\lesssim\lambda^{\frac{6-3d}{2}+\frac{d}{p_{1}'}+\frac{d}{p_{2}}+\frac{d}{p_{3}}}\|f_{1}\|_{p_{1}}\|f_{2}\|_{p_{2}'}\|f_{3}\|_{p_{3}'}=\lambda^{3+\frac{3d}{2}+\frac{d}{2}(1-\frac{1}{p_{1}}-\frac{1}{p_{2}'}-\frac{1}{p_{3}'})}\|f_{1}\|_{p_{1}}\|f_{2}\|_{p_{2}'}\|f_{3}\|_{p_{3}'}.\]

If we apply Theorem \ref{tetrtheorem} in this setting, letting our \(p_{1}\) be the same as theirs, our \(p_{2}\) and \( p_{3}\) be their \(p_{2}'\) and \(p_{3}'\), and \(p_{4}=\infty\), we get a decay term of \(\lambda^{\frac{d}{2}(1-\frac{1}{p_{1}}-\frac{1}{p_{2}'}-\frac{1}{p_{3}'})}\), which we can see is strictly better than the improvement in \cite{AKP24}. \\

Additionally, compared to both propositions we have a wider range for the choices of \(p_{1},p_{2},p_{3}\).


\section{The Diamond with a Diagonal, \(C_{4+t}\)} \label{Diamond+ section}

We define
\[\Ld_{C_{4+t}}(f_{1},f_{2},f_{3},f_{4})=\frac{1}{N_{C_{4+t}}}\sum_{x_{1},x_{2},x_{3},x_{4}\in\Z^{d}}\left(\prod_{i=1}^{4}f_{i}(x_{i})\right)S_{C_{4+t}}(x_{1},x_{2},x_{3},x_{4}),\]

where \[S_{C_{4+t}}(x_{1},x_{2},x_{3},x_{4})=S_{\lambda}(x_{1}-x_{2})S_{\lambda}(x_{2}-x_{3})S_{\lambda}(x_{3}-x_{4})S_{\lambda}(x_{4}-x_{1})S_{\lambda}(x_{3}-x_{1}).\]

\begin{theorem}\label{Diamond+thm}
    Suppose that \(1<p_{1},p_{2},p_{3},p_{4}\).
   If \(d\geq 7\) and \(1<p_{1},p_{2},p_{3},p_{4}\) with \(\frac{1}{p_{1}}+\frac{1}{p_{2}}+\frac{1}{p_{3}}+\frac{1}{p_{4}}>1\), then there exist constants \(C_{p_{1},p_{2},p_{3},p_{4}}\) such that for all \(\lambda\in\mathcal{R}_{C_{4+t}}\) we have the \(\ell^{p}\) improving inequality
    \begin{equation}\label{Diamond+eq}\Ld_{C_{4+t}}(f_{1},f_{2},f_{3},f_{4})\leq C_{p_{1},p_{2},p_{3}}\lambda^{\frac{d}{2}(1-\frac{1}{p_{1}}-\frac{1}{p_{2}}-\frac{1}{p_{3}}-\frac{1}{p_{4}})}\|f_{1}\|_{p_{1}}\|f_{2}\|_{p_{2}}\|f_{3}\|_{p_{3}}\|f_{4}\|_{p_{4}}\end{equation}
    provided that the point \((\frac{1}{p_{1}},\frac{1}{p_{2}},\frac{1}{p_{3}},\frac{1}{p_{4}})\) is in the convex hull of the following points: 
    \begin{multicols}{3}
        \begin{itemize}
            \item \((1,0,0,0)\)
            \item \((0,1,0,0)\)
            \item \((0,0,1,0)\)
            \columnbreak
            \item \((0,0,0,1)\)
            \item \((\frac{d-1}{d+1},\frac{d-1}{d+1},0,0)\)
            \item \((0,\frac{d-1}{d+1},\frac{d-1}{d+1},0)\)
            \columnbreak
            \item \((0,0,\frac{d-1}{d+1},\frac{d-1}{d+1})\)
            \item \((\frac{d-1}{d+1},0,0,\frac{d-1}{d+1})\)
            \item \((\frac{d-1}{d+1},0,\frac{d-1}{d+1},0)\)
        \end{itemize}
    \end{multicols}
\end{theorem}

\begin{proof}[Proof of Theorem \ref{Diamond+thm}]
To start, we will rearrange \(\Ld_{C_{4+t}}\).
\begin{align*}
    \Ld_{C_{4+t}}(f_{1},f_{2},f_{3},f_{4})&\approx\frac{1}{|S_{\lambda}||S_{\lambda}^{d-2}|^2}\sum_{x_{1},x_{2},x_{3},x_{4}\in\Z^{d}}\left(\prod_{i=1}^{4}f_{i}(x_{i})\right)S_{\lozenge+t}(x_{1},x_{2},x_{3},x_{4})\\
    &=\frac{1}{|S_{\lambda}||S_{\lambda}^{d-2}|^{2}}\sum_{x_{1},x_{2},x_{3}\in\Z^{d}}\left(\prod_{i=1}^{3}f_{i}(x_{i})\right)S_{\lambda}(x_{1}-x_{2})S_{\lambda}(x_{2}-x_{3})S_{\lambda}(x_{3}-x_{1})\\
    & \,\,\,\,\,\,\,\,\,\,\,\,\,\cdot\left(\sum_{x_{4}\in\Z^{d}}f_{4}(x_{4})S_{\lambda}(x_{3}-x_{4})S_{\lambda}(x_{4}-x_{1})\right)\\
    &\leq \frac{\|f_{4}\|_{\infty}}{|S_{\lambda}||S_{\lambda}^{d-2}|^{2}}\sum_{x_{1},x_{2},x_{3}\in\Z^{d}}\left(\prod_{i=1}^{3}f_{i}(x_{i})\right)S_{\lambda}(x_{1}-x_{2})S_{\lambda}(x_{2}-x_{3})S_{\lambda}(x_{3}-x_{1})\\
    & \,\,\,\,\,\,\,\,\,\,\,\,\,\cdot\left(\sum_{x_{4}\in\Z^{d}}S_{\lambda}(x_{3}-x_{4})S_{\lambda}(x_{4}-x_{1})\right)\\
    &\lesssim\frac{\|f_{4}\|_{\infty}}{|S_{\lambda}||S_{\lambda}^{d-2}|}\sum_{x_{1},x_{2},x_{3}\in\Z^{d}}\left(\prod_{i=1}^{3}f_{i}(x_{i})\right)S_{\lambda}(x_{1}-x_{2})S_{\lambda}(x_{2}-x_{3})S_{\lambda}(x_{3}-x_{1})\\
    &\approx\|f_{4}\|_{\infty}\Lambda_{K_{3}}(f_{1},f_{2},f_{3})
\end{align*}

By the same method, 
\[\Ld_{C_{4+t}}(f_{1},f_{2},f_{3},f_{4})\lesssim \|f_{2}\|_{\infty}\Ld_{K_{3}}(f_{1},f_{4},f_{3}).\]

Using this symmetry, Theorem \ref{trithm}, and interpolation we get that
\[\Ld_{C_{4+t}}(f_{1},f_{2},f_{3},f_{4})\lesssim\lambda^{\frac{d}{2}(1-\frac{1}{p_{1}}-\frac{1}{p_{2}}-\frac{1}{p_{3}}-\frac{1}{p_{4}})}\|f_{1}\|_{p_{1}}\|f_{2}\|_{p_{2}}\|f_{3}\|_{p_{3}}\|f_{4}\|_{p_{4}}\]
when \(\left(\frac{1}{p_{1}},\frac{1}{p_{2}},\frac{1}{p_{3}},\frac{1}{p_{4}}\right)\) is in the convex hull of points listed in Theorem \ref{Diamond+thm}.

\end{proof}

\subsection{Exponent Sharpness} 
To show that \(\frac{d}{2}(1-\frac{1}{p_{1}}-\frac{1}{p_{2}}-\frac{1}{p_{3}}-\frac{1}{p_{4}})\) is the best exponent on \(\lambda\) that we can expect, we let \(f_{1}(x)=f_{2}(x)=f_{3}(x)=f_{4}(x)=B_{\lambda}(x)\). Then \[\|f_{1}\|_{p_{1}}\|f_{2}\|_{p_{2}}\|f_{3}\|_{p_{3}}\|f_{4}\|_{p_{4}}\approx\lambda^{\frac{d}{2}(\frac{1}{p_{1}}+\frac{1}{p_{2}}+\frac{1}{p_{3}}+\frac{1}{p_{4}})},\]
and
\begin{align*}
    \Ld_{C_{4+t}}&(f_{1},f_{2},f_{3},f_{4})\\
    &\approx\frac{1}{|S_{\lambda}||S_{\lambda}^{d-2}|^{2}}\sum_{x_{1},\dots,x_{4}\in\Z^{d}}\left(\prod_{i=1}^{4}B_{\lambda}(x_{i})\right)S_{C_{4+t}}(x_{1},x_{2},x_{3},x_{4})\\
    &=\frac{1}{|S_{\lambda}||S_{\lambda}^{d-2}|}\sum_{x_{1},x_{2},x_{3}\in\Z^{d}}\left(\prod_{i=1}^{3}B_{\lambda}(x_{i})\right)S_{\lambda}(x_{1}-x_{2})S_{\lambda}(x_{2}-x_{3})S_{\lambda}(x_{3}-x_{1})\\
    &\,\,\,\,\,\,\,\,\,\,\,\,\,\,\,\,\,\,\cdot\left(\frac{1}{|S_{\lambda}^{d-2}|}\sum_{x_{4}\in\Z^{d}}B_{\lambda}(x_{4})S_{\lambda}(x_{3}-x_{4})S_{\lambda}(x_{4}-x_{3})\right)\\
    &\approx \frac{1}{|S_{\lambda}||S_{\lambda}^{d-2}|}\sum_{x_{1},x_{2}\in\Z^{d}}B_{\lambda}(x_{1})B_{\lambda}(x_{2})S_{\lambda}(x_{1}-x_{2})B_{\lambda}(x_{3})S_{\lambda}(x_{2}-x_{3})S_{\lambda}(x_{3}-x_{1})\\
    &= \frac{1}{|S_{\lambda}|}\sum_{x_{1},x_{2}\in\Z^{d}}B_{\lambda}(x_{1})B_{\lambda}(x_{2})S_{\lambda}(x_{1}-x_{2})\frac{1}{|S_{\lambda}^{d-2}|}\sum_{x_{3}\in\Z^{d}}B_{\lambda}(x_{3})S_{\lambda}(x_{2}-x_{3})S_{\lambda}(x_{3}-x_{1})\\
    &\approx\frac{1}{|S_{\lambda}|}\sum_{x_{1},x_{2}\in\Z^{d}}B_{\lambda}(x_{1})B_{\lambda}(x_{2})S_{\lambda}(x_{1}-x_{2})\\
    &=\sum_{x_{1}\in\Z^{d}}B_{\lambda}(x_{1})\frac{1}{|S_{\lambda}|}\sum_{x_{2}\in\Z^{d}}B_{\lambda}(x_{2})S_{\lambda}(x_{1}-x_{2})\\
    &\approx \sum_{x_{1}\in\Z^{d}}B_{\lambda}(x_{1})\\
    &\approx \lambda^{\frac{d}{2}}\approx \lambda^{\frac{d}{2}(1-\frac{1}{p_{1}}-\frac{1}{p_{2}}-\frac{1}{p_{3}}-\frac{1}{p_{4}})}\|f_{1}\|_{p_{1}}\|f_{2}\|_{p_{2}}\|f_{3}\|_{p_{3}}\|f_{4}\|_{p_{4}}
\end{align*}
Thus, the exponent on \(\lambda\) is sharp.\\

\subsection{Addressing Question \ref{mainquestion}}\label{diamond+question}

Take \(f_{1}=f_{3}=S_{\lambda}\) and \(f_{2}=f_{4}=1_{\{0\}}\). Then \[\|f_{1}\|_{p_{1}}\|f_{2}\|_{p_{2}}\|f_{3}\|_{p_{3}}\|f_{4}\|_{p_{4}}\approx\lambda^{\frac{d-2}{2}(\frac{1}{p_{1}}+\frac{1}{p_{3}})},\]

and so 
\begin{align*}
    \Ld_{C_{4+t}}(f_{1},f_{2},f_{3},f_{4})&\approx\frac{1}{|S_{\lambda}||S_{\lambda}^{d-2}|^{2}}\sum_{x_{1},...,x_{4}\in\Z^{d}} S_{\lambda}(x_{1})S_{\lambda}(x_{3})1_{\{0\}}(x_{2})1_{\{0\}}(x_{4})S_{\lozenge+t}(x_{1},x_{2},x_{3},x_{4})\\
    &=\frac{1}{|S_{\lambda}||S_{\lambda}^{d-2}|^{2}}\sum_{x_{1},x_{3}\in\Z^{d}} S_{\lambda}(x_{1})S_{\lambda}(x_{3})S_{\lambda}(x_{1}-x_{3})\\
    &=\frac{1}{|S_{\lambda}||S_{\lambda}^{d-2}|^{2}}\sum_{x_{1}\in\Z^{d}} S_{\lambda}(x_{1})\sum_{x_{3}\in\Z^{d}}S_{\lambda}(x_{3})S_{\lambda}(x_{1}-x_{3})\\
    &\approx \frac{|S_{\lambda}||S_{\lambda}^{d-2}|}{|S_{\lambda}||S_{\lambda}^{d-2}|^{2}} \approx \lambda^{-\frac{d-4}{2}}
\end{align*}

Thus, in order for the inequality in Theorem \ref{Diamond+thm} to hold, it must be that \[\lambda^{-\frac{d-4}{2}}\lesssim \lambda^{\frac{d}{2}(1-\frac{1}{p_{1}}-\frac{1}{p_{2}}-\frac{1}{p_{3}}-\frac{1}{p_{4}})}) \lambda^{\frac{d-2}{2}(\frac{1}{p_{1}}+\frac{1}{p_{3}})}.\]

Take the point where \(\frac{1}{p_{1}}=\frac{1}{p_{3}}=0\) and \(\frac{1}{p_{2}}=\frac{1}{p_{4}}=\frac{d-1}{d+1}\). Note that this is one of the points that makes up the convex hull of the range where Theorem \ref{tetrtheorem} holds for \(\Ld_{K_{4}}\). However, at this point, 
\[\frac{d}{2}(1-\frac{1}{p_{1}}-\frac{1}{p_{2}}-\frac{1}{p_{3}}-\frac{1}{p_{4}})+\frac{d-2}{2}(\frac{1}{p_{1}}+\frac{1}{p_{3}})=\frac{d}{2}(1-\frac{2(d-1)}{d+1})=2(\frac{d}{d+1})-\frac{d}{2}<-\frac{d-4}{2}.\]

Thus, there is some \(\epsilon>0\) small such that when \(\left(\frac{1}{p_{1}},\frac{1}{p_{2}},\frac{1}{p_{3}},\frac{1}{p_{4}}\right)=\left(0,\frac{d-1}{d+1}-\epsilon,0,\frac{d-1}{d+1}-\epsilon\right)\), we have that the \(\ell^{p}\) improving inequality holds with the best exponent for \(\Ld_{K_{4}}\), but does not hold for \(\Ld_{C_{4+t}}\). So the answer to Question \ref{mainquestion} is negative. Note that this example is one that was not addressed by Bhowmik, Iosevich, Koh, and Pham.


\section{The Diamond, \(C_{4}\)} \label{Diamond section}

Define 
\[\Ld_{C_{4}}(f_{1},f_{2},f_{3},f_{4})=\frac{1}{N_{C_{4}}}\sum_{x_{1},x_{2},x_{3},x_{4}\in\Z^{d}}\left(\prod_{i=1}^{4}f_{i}(x_{i})\right)S_{C_{4}}(x_{1},x_{2},x_{3},x_{4})\]

where \[S_{C_{4}}(x_{1},x_{2},x_{3},x_{4})=S_{\lambda}(x_{1}-x_{2})S_{\lambda}(x_{2}-x_{3})S_{\lambda}(x_{3}-x_{4})S_{\lambda}(x_{4}-x_{1})\]

\begin{theorem}\label{Diamondthm}
    Suppose that \(1<p_{1},p_{2},p_{3},p_{4}\). If \(d\geq 5\) and \(1<p_{1},p_{2},p_{3},p_{4}\) with \(\frac{1}{p_{1}}+\frac{1}{p_{2}}+\frac{1}{p_{3}}+\frac{1}{p_{4}}>1\), then there exist constants \(C_{p_{1},p_{2},p_{3},p_{4}}\) such that for all \(\lambda\in\mathcal{R}_{C_{4}}\) we have the \(\ell^{p}\) improving inequality
    \begin{equation}\label{Diamondeq}\Ld_{C_{4}}(f_{1},f_{2},f_{3},f_{4})\leq C_{p_{1},p_{2},p_{3}}\lambda^{\frac{d}{2}(1-\frac{1}{p_{1}}-\frac{1}{p_{2}}-\frac{1}{p_{3}}-\frac{1}{p_{4}})}\|f_{1}\|_{p_{1}}\|f_{2}\|_{p_{2}}\|f_{3}\|_{p_{3}}\|f_{4}\|_{p_{4}}\end{equation}
    provided that the point \((\frac{1}{p_{1}},\frac{1}{p_{2}},\frac{1}{p_{3}},\frac{1}{p_{4}})\) is in the convex hull of the following points: 
    \begin{multicols}{3}
        \begin{itemize}
            \item \((1,0,0,0)\)
            \item \((0,1,0,0)\)
            \item \((0,0,1,0)\)
            \columnbreak
            \item \((0,0,0,1)\)
            \item \((\frac{d-1}{d+1},\frac{d-1}{d+1},0,0)\)
            \item \((0,\frac{d-1}{d+1},\frac{d-1}{d+1},0)\)
            \columnbreak
            \item \((0,0,\frac{d-1}{d+1},\frac{d-1}{d+1})\)
            \item \((\frac{d-1}{d+1},0,0,\frac{d-1}{d+1})\)
        \end{itemize}
    \end{multicols}
     
\end{theorem}

\begin{proof}[Proof of Theorem \ref{Diamondthm}]
We will start by setting up the case where \(p_{4}=\infty\), then use the symmetry of the shape to get the full range in Theorem \ref{Diamondthm}.
\begin{align*}
    \Ld_{C_{4}}(f_{1},f_{2},f_{3},f_{4})&\approx\frac{1}{|S_{\lambda}|^{2}|S_{\lambda}^{d-2}|}\sum_{x_{1},x_{2},x_{3},x_{4}\in\Z^{d}}\prod_{i=1}^{4}f_{i}(x_{i})S_{\lambda}(x_{1}-x_{2})S_{\lambda}(x_{2}-x_{3})S_{\lambda}(x_{3}-x_{4})S_{\lambda}(x_{4}-x_{1})\\
    &=\frac{1}{|S_{\lambda}|^{2}|S_{\lambda}^{d-2}|}\sum_{x_{1},x_{2},x_{3}\in\Z^{d}}\prod_{i=1}^{3}f_{i}(x_{i})S_{\lambda}(x_{1}-x_{2})S_{\lambda}(x_{2}-x_{3})\\&\quad\quad\quad\cdot\sum_{x_{4}\in\Z^{d}}f_{4}(x_{4})S_{\lambda}(x_{3}-x_{4})S_{\lambda}(x_{4}-x_{1})\\
    &\leq \frac{\|f_{4}\|_{\infty}}{|S_{\lambda}|^{2}|S_{\lambda}^{d-2}|}\sum_{x_{1},x_{2},x_{3}\in\Z^{d}}\prod_{i=1}^{3}f_{i}(x_{i})S_{\lambda}(x_{1}-x_{2})S_{\lambda}(x_{2}-x_{3})\sum_{x_{4}\in\Z^{d}}S_{\lambda}(x_{3}-x_{4})S_{\lambda}(x_{4}-x_{1})
    \end{align*}
    If \(x_{1}\neq x_{3}\), then \(\displaystyle\sum_{x_{4}\in\Z^{d}}S_{\lambda}(x_{3}-x_{4})S_{\lambda}(x_{4}-x_{1}) \approx |S_{\lambda}^{d-2}|.\) However, if \(x_{1}=x_{3}\), then \\ \(\displaystyle\sum_{x_{4}\in\Z^{d}}S_{\lambda}(x_{3}-x_{4})S_{\lambda}(x_{4}-x_{1}) \approx |S_{\lambda}|\). To account for this difference in constants, we must split the sum into two parts.
    
    \begin{align*}
    \Ld_{C_{4}}(f_{1},f_{2},f_{3},f_{4})&\lesssim\frac{\|f_{4}\|_{\infty}}{|S_{\lambda}|^{2}}\sum_{\substack{x_{1},x_{2},x_{3}\in\Z^{d}\\x_{1}\neq x_{3}}}\prod_{i=1}^{3}f_{i}(x_{i})S_{\lambda}(x_{1}-x_{2})S_{\lambda}(x_{2}-x_{3})\\
    &\,\,\,\,\,\,\,\,\,\,+\frac{\|f_{4}\|_{\infty}}{|S_{\lambda}||S_{\lambda}^{d-2}|}\sum_{\substack{x_{1},x_{2},x_{3}\in\Z^{d}\\x_{1}=x_{3}}}\prod_{i=1}^{3}f_{i}(x_{i})S_{\lambda}(x_{1}-x_{2})S_{\lambda}(x_{2}-x_{3})\\
   &\approx\|f_{4}\|_{\infty} \left(\Ld_{P_{2}}(f_{1},f_{2},f_{3})+\frac{1}{|S_{\lambda}^{d-2}|}\Ld_{P_{1}}(f_{2},f_{3}\cdot f_{1})\right).
\end{align*}

If \(\frac{1}{p_{1}},\frac{1}{p_{2}},\frac{1}{p_{3}}\) are such that the  conditions in Theorem \ref{2chain} for \(\Ld_{P_{2}}\) are satisfied by \(\left(\frac{1}{p_{1}},\frac{1}{p_{2}},\frac{1}{p_{3}}\right)\), and \(\left(\frac{1}{p_{2}},\frac{1}{p_{1}}+\frac{1}{p_{3}}\right)\) is in the region where the \(\ell^{p}\) improving inequality in Theorem \ref{Lthm} holds for \(\Ld_{P_{1}}\), then

\begin{align*}
    \Ld_{C_{4}}(f_{1},f_{2},f_{3},f_{4})&\lesssim\|f_{4}\|_{\infty} \left(\Ld_{P_{2}}(f_{1},f_{2},f_{3})+\frac{1}{|S_{\lambda}^{d-2}|}\Ld_{P_{1}}(f_{2},f_{3}\cdot f_{1})\right)\\
    &\lesssim \|f_{4}\|_{\infty}\|f_{1}\|_{p_{1}}\|f_{2}\|_{p_{2}}\|f_{3}\|_{p_{3}}\lambda^{\frac{d}{2}(1-\frac{1}{p_{1}}-\frac{1}{p_{2}}-\frac{1}{p_{3}})}\left(1+\frac{1}{|S_{\lambda}^{d-2}|}\right)\\
    &\lesssim\|f_{4}\|_{\infty}\|f_{1}\|_{p_{1}}\|f_{2}\|_{p_{2}}\|f_{3}\|_{p_{3}}\lambda^{\frac{d}{2}(1-\frac{1}{p_{1}}-\frac{1}{p_{2}}-\frac{1}{p_{3}})}
\end{align*}
as
\[\Ld_{P_{2}}(f_{1},f_{2},f_{3})\lesssim \|f_{1}\|_{p_{1}}\|f_{2}\|_{p_{2}}\|f_{3}\|_{p_{3}}\lambda^{\frac{d}{2}(1-\frac{1}{p_{1}}-\frac{1}{p_{2}}-\frac{1}{p_{3}})}\]
and 
\begin{align*}
    \Ld_{P_{1}}(f_{2},f_{1}\cdot f_{3})
    &\lesssim\|f_{2}\|_{p_{2}}\|f_{1}f_{3}\|_{\left(\frac{1}{p_{1}}+\frac{1}{p_{3}}\right)^{-1}}\lambda^{\frac{d}{2}(1-\frac{1}{p_{2}}-\frac{1}{p_{3}}-\frac{1}{p_{1}})}\\
    &\leq \|f_{2}\|_{p_{2}}\|f_{1}\|_{p_{1}}\|f_{3}\|_{p_{3}}\lambda^{\frac{d}{2}(1-\frac{1}{p_{2}}-\frac{1}{p_{3}}-\frac{1}{p_{1}})}.
\end{align*}

We need to know for what \(\frac{1}{p_{1}},\frac{1}{p_{2}},\frac{1}{p_{3}}\) we get that the conditions in Theorem \ref{2chain} for \(\Ld_{P_{2}}\) are satisfied for \(\left(\frac{1}{p_{1}},\frac{1}{p_{2}},\frac{1}{p_{3}}\right)\), and the conditions in Theorem \ref{Lthm} for \(\Ld_{P_{1}}\) are satisfied with \(\left(\frac{1}{p_{2}},\frac{1}{p_{1}}+\frac{1}{p_{3}}\right)\).\\

We have that the conditions in Theorem \ref{2chain} are satisfied when \((\frac{1}{p_{1}},\frac{1}{p_{2}},\frac{1}{p_{3}})\) is in the convex hull of the points: 
\begin{multicols}{3}
    \begin{itemize}
        \item \((1,0,0)\)
        \item \((0,1,0)\)
        \item \((0,0,1)\)
        \columnbreak
        \item \((\frac{d-1}{d+1},\frac{d-1}{d+1},0)\)
        \item \((0,\frac{d-1}{d+1},\frac{d-1}{d+1})\)
        \item \((\frac{d-1}{d+1},\frac{d-3}{d+1},\frac{d-1}{d+1})\)
        \columnbreak
        \item \((\frac{d-1}{d+1},0,\frac{d^{2}-5}{d^{2}-1})\)
        \item \(\frac{d^{2}-5}{d^{2}-1}, 0, \frac{d-1}{d+1})\)
    \end{itemize}
\end{multicols}

The conditions in Theorem \ref{Lthm} will be satisfied when, additionally, we have that \((\frac{1}{p_{1}},\frac{1}{p_{2}},\frac{1}{p_{3}})\) is in the convex hull of the points 
\begin{multicols}{3}
    \begin{itemize}
        \item \((1,0,0)\)
        \item \((0,1,0)\)
        \columnbreak
        \item \((0,0,1)\)
        \item \((\frac{d-1}{d+1},\frac{d-1}{d+1},0)\)
        \columnbreak
        \item \((0,\frac{d-1}{d+1},\frac{d-1}{d+1})\)
    \end{itemize}
\end{multicols}

The conditions in Theorem \ref{Lthm} are more restrictive than the ones in Theorem \ref{2chain}. Thus, we can say that 
\[C_{4}(f_{1},f_{2},f_{3},f_{4})\lesssim\lambda^{\frac{d}{2}(1-\frac{1}{p_{1}}-\frac{1}{p_{2}}-\frac{1}{p_{3}})}\|f_{1}\|_{p_{1}}\|f_{2}\|_{p_{2}}\|f_{3}\|_{p_{3}}\|f_{4}\|_{\infty}\]
 when \((\frac{1}{p_{1}},\frac{1}{p_{2}},\frac{1}{p_{3}})\) is in the convex hull of the points:
 \begin{multicols}{3}
    \begin{itemize}
        \item \((1,0,0)\)
        \item \((0,1,0)\)
        \columnbreak
        \item \((0,0,1)\)
        \item \((\frac{d-1}{d+1},\frac{d-1}{d+1},0)\)
        \columnbreak
        \item \((0,\frac{d-1}{d+1},\frac{d-1}{d+1})\)
    \end{itemize}
\end{multicols}

Using the symmetry of the diamond, we can say additionally that \[\Ld_{C_{4}}(f_{1},f_{2},f_{3},f_{4})\lesssim\lambda^{\frac{d}{2}(1-\frac{1}{p_{1}}-\frac{1}{p_{2}}-\frac{1}{p_{4}})}\|f_{1}\|_{p_{1}}\|f_{2}\|_{p_{2}}\|f_{3}\|_{\infty}\|f_{4}\|_{p_{4}}\] when \((\frac{1}{p_{1}},\frac{1}{p_{2}},\frac{1}{p_{4}})\) is in the convex hull of the points listed above, that \[\Ld_{C_{4}}(f_{1},f_{2},f_{3},f_{4})\lesssim\lambda^{\frac{d}{2}(1-\frac{1}{p_{1}}-\frac{1}{p_{3}}-\frac{1}{p_{4}})}\|f_{1}\|_{p_{1}}\|f_{2}\|_{\infty}\|f_{3}\|_{p_{3}}\|f_{4}\|_{p_{4}}\] when \((\frac{1}{p_{1}},\frac{1}{p_{3}},\frac{1}{p_{4}})\) is in the convex hull of the points listed above, and that \[\Ld_{C_{4}}(f_{1},f_{2},f_{3},f_{4})\lesssim\lambda^{\frac{d}{2}(1-\frac{1}{p_{2}}-\frac{1}{p_{3}}-\frac{1}{p_{4}})}\|f_{1}\|_{\infty}\|f_{2}\|_{p_{2}}\|f_{3}\|_{p_{3}}\|f_{4}\|_{p_{4}}\] when \((\frac{1}{p_{2}},\frac{1}{p_{3}},\frac{1}{p_{4}})\) when is in the convex hull of the points listed above. \\ 
Interpolating between these results gives Theorem \ref{Diamondthm}. 
\end{proof}

\subsection{Exponent Sharpness}
To establish that \(\frac{d}{2}\left(1-\frac{1}{p_{1}}-\frac{1}{p_{2}}-\frac{1}{p_{3}}-\frac{1}{p_{4}}\right)\) is the best exponent that we can guarantee, take \(f_{1}(x)=f_{2}(x)=f_{3}(x)=f_{4}(x)=B_{\lambda}(x)\). Then \[\|f_{1}\|_{p_{1}}\|f_{2}\|_{p_{2}}\|f_{3}\|_{p_{3}}\|f_{4}\|_{p_{4}}\approx\lambda^{\frac{d}{2}(\frac{1}{p_{1}}+\frac{1}{p_{2}}+\frac{1}{p_{3}}+\frac{1}{p_{4}})},\]
and
\begin{align*}
    \Ld_{C_{4}}(f_{1},f_{2},f_{3},f_{4})&\approx\frac{1}{|S_{\lambda}|^{2}|S_{\lambda}^{d-2}|}\sum_{x_{1},x_{2},x_{3},x_{4}\in\Z^{d}}\left(\prod_{i=1}^{4}B_{\lambda}(x_{i})\right)S_{C_{4}}(x_{1},x_{2},x_{3},x_{4})\\
    &=\frac{1}{|S_{\lambda}|^{2}}\sum_{x_{1},x_{2},x_{3}\in\Z^{d}}\left(\prod_{i=1}^{3}B_{\lambda}(x_{i})\right)S_{\lambda}(x_{1}-x_{2})S_{\lambda}(x_{2}-x_{3})\frac{1}{|S_{\lambda}^{d-2}|}\sum_{x_{4}\in\Z^{d}}B_{\lambda}(x_{4})S_{\lambda}(x_{3}-x_{4})S_{\lambda}(x_{4}-x_{3})\\
    &\approx \frac{1}{|S_{\lambda}|^{2}}\sum_{x_{1},x_{2}\in\Z^{d}}B_{\lambda}(x_{1})B_{\lambda}(x_{2})S_{\lambda}(x_{1}-x_{2})B_{\lambda}(x_{3})S_{\lambda}(x_{2}-x_{3})\\
    &= \frac{1}{|S_{\lambda}|}\sum_{x_{1},x_{2}\in\Z^{d}}B_{\lambda}(x_{1})B_{\lambda}(x_{2})S_{\lambda}(x_{1}-x_{2})\frac{1}{|S_{\lambda}|}\sum_{x_{3}\in\Z^{d}}B_{\lambda}(x_{3})S_{\lambda}(x_{2}-x_{3})\\
    &\approx\frac{1}{|S_{\lambda}|}\sum_{x_{1},x_{2}\in\Z^{d}}B_{\lambda}(x_{1})B_{\lambda}(x_{2})S_{\lambda}(x_{1}-x_{2})\\
    &=\sum_{x_{1}\in\Z^{d}}B_{\lambda}(x_{1})\frac{1}{|S_{\lambda}|}\sum_{x_{2}\in\Z^{d}}B_{\lambda}(x_{2})S_{\lambda}(x_{1}-x_{2})\\
    &\approx \sum_{x_{1}\in\Z^{d}}B_{\lambda}(x_{1})\\
    &\approx \lambda^{\frac{d}{2}}\approx \lambda^{\frac{d}{2}(1-\frac{1}{p_{1}}-\frac{1}{p_{2}}-\frac{1}{p_{3}}-\frac{1}{p_{4}})}\|f_{1}\|_{p_{1}}\|f_{2}\|_{p_{2}}\|f_{3}\|_{p_{3}}\|f_{4}\|_{p_{4}}
\end{align*}

Thus, \(\frac{d}{2}\left(1-\frac{1}{p_{1}}-\frac{1}{p_{2}}-\frac{1}{p_{3}}-\frac{1}{p_{4}}\right)\) is the best exponent that we can expect for \(\Lambda_{C_{4}}\).

\subsection{Addressing Question \ref{mainquestion}}\label{diamondquestion}

As in Section \ref{diamond+question}, take \(f_{1}=f_{3}=S_{\lambda}\) and \(f_{2}=f_{4}=1_{\{0\}}\). Then \[\|f_{1}\|_{p_{1}}\|f_{2}\|_{p_{2}}\|f_{3}\|_{p_{3}}\|f_{4}\|_{p_{4}}\approx\lambda^{\frac{d-2}{2}(\frac{1}{p_{1}}+\frac{1}{p_{3}})},\]

and 
\begin{align*}
    \Ld_{C_{4}}(f_{1},f_{2},f_{3},f_{4})&\approx\frac{1}{|S_{\lambda}|^{2}|S_{\lambda}^{d-2}|}\sum_{x_{1},...,x_{4}\in\Z^{d}}1_{\{0\}}(x_{2})1_{\{0\}}(x_{4}) S_{\lambda}(x_{1})S_{\lambda}(x_{3})S_{C_{4}}(x_{1},x_{2},x_{3},x_{4})\\
    &=\frac{1}{|S_{\lambda}|^{2}|S_{\lambda}^{d-2}|}\sum_{x_{1},x_{3}\in\Z^{d}} S_{\lambda}(x_{1})S_{\lambda}(x_{3})\\
    &\approx\frac{|S_{\lambda}|^{2}}{|S_{\lambda}|^{2}|S_{\lambda}^{d-2}|}
     \approx \lambda^{-\frac{d-4}{2}}.
\end{align*}

Thus, in order to get \(\ell^{p}\) improving results with the best exponent on \(\lambda\), it must be that \[\lambda^{-\frac{d-4}{2}}\lesssim \lambda^{\frac{d}{2}(1-\frac{1}{p_{1}}-\frac{1}{p_{2}}-\frac{1}{p_{3}}-\frac{1}{p_{4}})}) \lambda^{\frac{d-2}{2}(\frac{1}{p_{2}}+\frac{1}{p_{4}})}.\]

Take the point where \(\frac{1}{p_{1}}=\frac{1}{p_{3}}=0\) and \(\frac{1}{p_{2}}=\frac{1}{p_{4}}=\frac{d-1}{d+1}\). Note that this is one of the points that makes up the convex hull of the range where Theorem \ref{tetrtheorem} holds for \(\Ld_{K_{4}}\). However, at this point, 
\[\frac{d}{2}(1-\frac{1}{p_{1}}-\frac{1}{p_{2}}-\frac{1}{p_{3}}-\frac{1}{p_{4}})+\frac{d-2}{2}(\frac{1}{p_{1}}+\frac{1}{p_{3}})=\frac{d}{2}(1-\frac{2(d-1)}{d+1})=2(\frac{d}{d+1})-\frac{d}{2}<-\frac{d-4}{2}.\]

Thus, there is some small \(\epsilon>0\) such that at the point \(\left(\frac{1}{p_{1}},\frac{1}{p_{2}},\frac{1}{p_{3}},\frac{1}{p_{4}}\right)=\left(0,\frac{d-1}{d+1}-\epsilon,0, \frac{d-1}{d+1}-\epsilon,\right)\), we have that the \(\ell^{p}\) improving inequality holds with the best exponent for \(\Ld_{K_{4}}\), but we have shown with this example that it does not hold for \(\Ld_{C_{4}}\). This shows that the answer to Question \ref{mainquestion} is negative.\\

We can also use the example from \cite{BIKP25} to show that the answer to question \ref{mainquestion} is negative. That is, where \(G=C_{4+t}\) and \(G'=C_{4}\), and we let \(f_{2}=f_{4}=S_{\lambda}\) and \(f_{1}=f_{3}=1_{\{0\}}\). Then \[\|f_{1}\|_{p_{1}}\|f_{2}\|_{p_{2}}\|f_{3}\|_{p_{3}}\|f_{4}\|_{p_{4}}\approx\lambda^{\frac{d-2}{2}(\frac{1}{p_{2}}+\frac{1}{p_{4}})},\]

and 
\begin{align*}
    \Ld_{C_{4}}(f_{1},f_{2},f_{3},f_{4})&\approx\frac{1}{|S_{\lambda}|^{2}|S_{\lambda}^{d-2}|}\sum_{x_{1},...,x_{4}\in\Z^{d}}1_{\{0\}}(x_{1})1_{\{0\}}(x_{3}) S_{\lambda}(x_{2})S_{\lambda}(x_{4})S_{C_{4}}(x_{1},x_{2},x_{3},x_{4})\\
    &=\frac{1}{|S_{\lambda}|^{2}|S_{\lambda}^{d-2}|}\sum_{x_{2},x_{4}\in\Z^{d}} S_{\lambda}(x_{2})S_{\lambda}(x_{4})\\
    &\approx\frac{|S_{\lambda}|^{2}}{|S_{\lambda}|^{2}|S_{\lambda}^{d-2}|}
     \approx \lambda^{-\frac{d-4}{2}}
\end{align*}

Thus, in order for the inequality in Theorem \ref{Diamond+thm} to hold, it must be that \[\lambda^{-\frac{d-4}{2}}\lesssim \lambda^{\frac{d}{2}(1-\frac{1}{p_{1}}-\frac{1}{p_{2}}-\frac{1}{p_{3}}-\frac{1}{p_{4}})}) \lambda^{\frac{d-2}{2}(\frac{1}{p_{2}}+\frac{1}{p_{4}})}.\]

Take the point where \(\frac{1}{p_{2}}=\frac{1}{p_{4}}=0\) and \(\frac{1}{p_{1}}=\frac{1}{p_{3}}=\frac{d-1}{d+1}\). Note that this is one of the points that makes up the convex hull of the range where where Theorem \ref{Diamond+thm} holds for \(\Ld_{C_{4+t}}\). However, at this point, 
\[\frac{d}{2}(1-\frac{1}{p_{1}}-\frac{1}{p_{2}}-\frac{1}{p_{3}}-\frac{1}{p_{4}})+\frac{d-2}{2}(\frac{1}{p_{2}}+\frac{1}{p_{4}})=\frac{d}{2}(1-\frac{2(d-1)}{d+1})=2(\frac{d}{d+1})-\frac{d}{2}<-\frac{d-4}{2}.\]

Thus, there is some small \(\epsilon>0\) such that at the point \(\left(\frac{1}{p_{1}},\frac{1}{p_{2}},\frac{1}{p_{3}},\frac{1}{p_{4}}\right)=\left(\frac{d-1}{d+1}-\epsilon,0,\frac{d-1}{d+1}\epsilon,0\right)\), we have that the \(\ell^{p}\) improving inequality holds with the best exponent for \(\Ld_{C_{4+t}}\), but does not hold for \(\Ld_{C_{4}}\). This is another example to show that the answer to Question \ref{mainquestion} is negative.


\section{The Triangle with a Tail, \(K_{3+t}\)} \label{Tt section}

For \(d\geq 7\), define \(\Ld_{K_{3+t}}\) to be the form \[\Ld_{K_{3+t}}(f_{1},f_{2},f_{3},f_{4})=\frac{1}{N_{K_{3}}}\sum_{x_{1},\dots,x_{4}\in\Z^{d}}\left(\prod_{i=1}^{4}f_{i}(x_{i})\right)S_{\lambda}(x_{1}-x_{2})S_{\lambda}(x_{2}-x_{3})S_{\lambda}(x_{3}-x_{1})S_{\lambda}(x_{3}-x_{4}).\]

\begin{theorem}\label{Tttheorem}
    If \(d\geq 7\) and \(1<p_{1},p_{2},p_{3},p_{4}\) with \(\frac{1}{p_{1}}+\frac{1}{p_{2}}+\frac{1}{p_{3}}+\frac{1}{p_{4}}>1\), then there exist constants \(C_{p_{1},p_{2},p_{3},p_{4}}\) such that for all \(\lambda\in\mathcal{R}_{K_{3+t}}\) we have the \(\ell^{p}\) improving inequality
    \begin{equation}\label{Tteq}\Ld_{K_{3+t}}(f_{1},f_{2},f_{3},f_{4})\leq C_{p_{1},p_{2},p_{3},p_{4}}\lambda^{\frac{d}{2}(1-\frac{1}{p_{1}}-\frac{1}{p_{2}}-\frac{1}{p_{3}}-\frac{1}{p_{4}})}\|f_{1}\|_{p_{1}}\|f_{2}\|_{p_{2}}\|f_{3}\|_{p_{3}}\|f_{4}\|_{p_{4}}\end{equation}
    provided \(\left(\frac{1}{p_{1}},\frac{1}{p_{2}},\frac{1}{p_{3}},\frac{1}{p_{4}}\right)\) is in the convex hull of the points:
\begin{multicols}{4}
    \begin{itemize}
            \item \((1,0,0,0)\)
            \item \((0,1,0,0)\)
            \item \((0,0,1,0)\)
            \item \((0,0,0,1)\)
            \columnbreak
            \item \((0,\frac{d-1}{d+1},0,\frac{d-1}{d+1})\)
            \item \((\frac{d-1}{d+1},\frac{d-1}{d+1},0,0)\)
            \item \((0,0,\frac{d-1}{d+1},\frac{d-1}{d+1})\)
            \item \((0,\frac{d-1}{d+1},\frac{d-1}{d+1},0)\)
            \columnbreak
            \item \((0,\frac{d-1}{d+1},\frac{d-3}{d+1},\frac{d-1}{d+1})\)
            \item \((\frac{d-1}{d+1},0,\frac{d-3}{d+1},\frac{d-1}{d+1})\)
            \item \((\frac{d-3}{d+1},\frac{d-1}{d+1},0,\frac{d-1}{d+1})\)
            \item \((\frac{d-1}{d+1},\frac{d-3}{d+1},0,\frac{d-1}{d+1})\)
            \columnbreak
            \item \((0,\frac{d-1}{d+1},0,\frac{d^{2}-5}{d^{2}-1})\)
            \item \((0,\frac{d^{2}-5}{d^{2}-1},0,\frac{d-1}{d+1})\)
            \item \((\frac{d-1}{d+1},0,0,\frac{d^{2}-5}{d^{2}-1})\)
            \item \((\frac{d^{2}-5}{d^{2}-1},0,0,\frac{d-1}{d+1})\)
        \end{itemize}
\end{multicols}
\end{theorem}

\begin{remark}\label{Ttrmk}
    The region in Theorem \ref{Tttheorem} is the set of all points \(\left(\frac{1}{p_{1}},\frac{1}{p_{2}},\frac{1}{p_{3}},\frac{1}{p_{4}}\right)\) such that the following hold:
    \begin{itemize}
    \item \(\frac{1}{p_{1}}+\frac{1}{p_{2}}+\frac{1}{p_{3}}+\frac{1}{p_{4}}>1\)
    \item \(\left(\frac{1}{p_{1}}+\frac{1}{p_{2}}+\frac{1}{p_{4}}\right)\frac{2}{d-1}+\frac{1}{p_{3}}<1\) when \(\frac{1}{p_{1}}+\frac{1}{p_{2}}<\frac{d-1}{d+1}\) and \(\frac{1}{p_{4}}<\frac{d-1}{d+1}\)
    \item \(\left(\frac{1}{p_{1}}+\frac{1}{p_{3}}\right)\frac{2}{d-1}+\frac{1}{p_{2}}+\frac{d-1}{2}\left(\frac{1}{p_{4}}\right)<1\) when \(\frac{1}{p_{2}}>\frac{d-1}{d+1}\) and \(\frac{1}{p_{4}}<\frac{d-1}{d+1}\)
    \item \(\left(\frac{1}{p_{3}}+\frac{1}{p_{2}}\right)\frac{2}{d-1}+\frac{1}{p_{1}}+\frac{d-1}{2}\left(\frac{1}{p_{4}}\right)<1\) when \(\frac{1}{p_{1}}>\frac{d-1}{d+1}\) and \(\frac{1}{p_{4}}<\frac{d-1}{d+1}\)
    \item \(\frac{1}{p_{1}}+\frac{1}{p_{2}}+\frac{1}{p_{3}}+\frac{2}{d-1}\left(\frac{1}{p_{4}}\right)<\frac{2(d-1)}{d+1}\) when \(\frac{1}{p_{1}},\frac{1}{p_{2}}<\frac{d-1}{d+1},\) but \(\frac{1}{p_{1}}+\frac{1}{p_{2}}>\frac{d-1}{d+1}\), and \(\frac{1}{p_{4}}<\frac{d-1}{d+1}\).
    \item \(\frac{1}{p_{1}}+\frac{1}{p_{2}}+\frac{1}{p_{3}}+\left(\frac{1}{p_{4}}\right)\frac{2}{d-1}<\frac{d-1}{2}\) when \(\frac{1}{p_{1}}+\frac{1}{p_{2}}<\frac{d-1}{d+1}\), and \(\frac{1}{p_{4}}\geq\frac{d-1}{d+1}\)
    \item \(\left(\frac{1}{p_{1}}+\frac{1}{p_{3}}\right)\frac{2}{d-1}+\frac{1}{p_{2}}+\frac{1}{p_{4}}<2-\frac{2}{d-1}\) when \(\frac{1}{p_{2}}>\frac{d-1}{d+1}\) and \(\frac{1}{p_{4}}\geq\frac{d-1}{d+1}\)
    \item \(\left(\frac{1}{p_{2}}+\frac{1}{p_{3}}\right)\frac{2}{d-1}+\frac{1}{p_{1}}+\frac{1}{p_{4}}<2-\frac{2}{d-1}\) when \(\frac{1}{p_{1}}>\frac{d-1}{d+1}\) and \(\frac{1}{p_{4}}\geq\frac{d-1}{d+1}\)
    \item \(\frac{1}{p_{1}}+\frac{1}{p_{2}}+\frac{1}{p_{3}}+\frac{d-1}{2}\left(\frac{1}{p_{4}}\right)<\frac{2(d-1)}{d+1}+\frac{d-3}{2}\) when \(\frac{1}{p_{1}},\frac{1}{p_{2}}<\frac{d-1}{d+1},\) but \(\frac{1}{p_{1}}+\frac{1}{p_{2}}>\frac{d-1}{d+1}\),  and \(\frac{1}{p_{4}}\geq\frac{d-1}{d+1}\).
\end{itemize}
Note that if we consider the cases where \(\frac{1}{p_{1}}=0\) or \(\frac{1}{p_{2}}=0\), then these conditions become the equations in Remark \ref{2chainrmk} that bound the region of \(\ell^{p}\) improvement for \(\Lambda_{P_{2}}\). This makes sense as ignoring the points \(x_{1}\) or \(x_{2}\) gives a 2-chain. Similarly, taking \(\frac{1}{p_{4}}=0\) gives the conditions in Remark \ref{Trirmk} for the region of \(\ell^{p}\) improvement for \(\Lambda_{K_{3}}\). Again, this makes sense as ignoring the point \(x_{4}\) leaves a triangle shape.
\end{remark}

\begin{proof}[Proof of Theorem \ref{Tttheorem}]
We can rewrite \(\Lambda_{K_{3+t}}\) as
\begin{align*}
    \Ld_{K_{3+t}}(f_{1},f_{2},f_{3},f_{4})&\approx\frac{1}{|S_{\lambda}|^{2}|S_{\lambda}^{d-2}|}\sum_{x_{1},x_{2},x_{3},x_{4}\in\Z^{d}}\left(\prod_{i=1}^{4}f_{i}(x_{i})\right)S_{\lambda}(x_{1}-x_{2})S_{\lambda}(x_{2}-x_{3})S_{\lambda}(x_{3}-x_{1})S_{\lambda}(x_{3}-x_{4})\\
    &=\frac{1}{|S_{\lambda}|^{2}|S_{\lambda}^{d-2}|}\sum_{x_{1},x_{2},x_{3}\in\Z^{d}}\left(\prod_{i=1}^{3}f_{i}(x_{i})\right)S_{\lambda}(x_{1}-x_{2})S_{\lambda}(x_{2}-x_{3})S_{\lambda}(x_{3}-x_{1})\\
    &\quad\quad\quad\cdot\sum_{x_{4}\in\Z^{d}}f_{4}(x_{4})S_{\lambda}(x_{3}-x_{4})\\
    &=\frac{1}{|S_{\lambda}||S_{\lambda}^{d-2}|}\sum_{x_{1},x_{2},x_{3}\in\Z_{d}}\left(\prod_{i=1}^{3}f_{i}(x_{i})\right)A_{\lambda}f_{4}(x_{3})S_{\lambda}(x_{1}-x_{2})S_{\lambda}(x_{2}-x_{3})S_{\lambda}(x_{3}-x_{1})\\
    &\approx\Ld_{K_{3}}(f_{1},f_{2},f_{3}\cdot A_{\lambda}f_{4})
\end{align*}
Assume that \(\frac{1}{p_{1}},\frac{1}{p_{2}},\frac{1}{p_{3}},\frac{1}{p_{4}}\) are such that there exists a \(\frac{1}{q_{4}}\) where \(\left(\frac{1}{p_{4}},\frac{1}{q_{4}}\right)\) is in the region of \(\ell^{p}\) improvement for \(A_{\lambda}\), and \(\left(\frac{1}{p_{1}},\frac{1}{p_{2}},\frac{1}{p_{3}}+\frac{1}{q_{4}}\right)\) satisfies the conditions in Theorem \ref{trithm} to get \(\ell^{p}\) improvement for \(\Ld_{K_{3}}\). If such a \(\frac{1}{q_{4}}\) exists, then we can say
\begin{align*}
    \Ld_{K_{3+t}}(f_{1},f_{2},f_{3},f_{4})&\approx\Ld_{K_{3}}(f_{1},f_{2},f_{3}\cdot A_{\lambda}f_{4})\\
    &\lesssim \lambda^{\frac{d}{2}(1-\frac{1}{p_{1}}-\frac{1}{p_{2}}-\frac{1}{p_{3}}-\frac{1}{q_{4}})}\|f_{1}\|_{p_{1}}\|f_{2}\|_{p_{2}}\|f_{3}\|_{p_{3}}\|A_{\lambda}f_{4}\|_{q_{4}}\\
    &\lesssim  \lambda^{\frac{d}{2}(1-\frac{1}{p_{1}}-\frac{1}{p_{2}}-\frac{1}{p_{3}}-\frac{1}{p_{4}})}\|f_{1}\|_{p_{1}}\|f_{2}\|_{p_{2}}\|f_{3}\|_{p_{3}}\|f_{4}\|_{p_{4}}.
\end{align*}

If \(\frac{1}{p_{4}}<\frac{d-1}{d+1}\), we will need \(\frac{1}{q_{4}}\) to be such that \(\frac{1}{p_{4}}>\frac{1}{q_{4}}>\frac{2}{d-1}(\frac{1}{p_{4}})\) so that \((\frac{1}{p_{4}},\frac{1}{q_{4}}\) is in the region of \(\ell^{p}\) improving for \(A_{\lambda}\), and we will need that the following equations hold so that \(\left(\frac{1}{p_{1}},\frac{1}{p_{2}},\frac{1}{p_{3}}+\frac{1}{q_{4}}\right)\) satisfies the conditions in Theorem \ref{trithm} to get \(\ell^{p}\) improvement for \(\Ld_{K_{3}}\).
\begin{itemize}
    \item \(\frac{1}{p_{1}}+\frac{1}{p_{2}}+\frac{1}{p_{3}}+\frac{1}{q_{4}}>1\)
    \item \(\left(\frac{1}{p_{1}}+\frac{1}{p_{2}}\right)\frac{2}{d-1}+\frac{1}{p_{3}}+\frac{1}{q_{4}}<1\) when \(\frac{1}{p_{1}}+\frac{1}{p_{2}}<\frac{d-1}{d+1}\)
    \item \(\left(\frac{1}{p_{1}}+\frac{1}{p_{3}}+\frac{1}{q_{4}}\right)\frac{2}{d-1}+\frac{1}{p_{2}}<1\) when \(\frac{1}{p_{2}}>\frac{d-1}{d+1}\)
    \item \(\left(\frac{1}{p_{3}}+\frac{1}{q_{4}}+\frac{1}{p_{2}}\right)\frac{2}{d-1}+\frac{1}{p_{1}}<1\) when \(\frac{1}{p_{1}}>\frac{d-1}{d+1}\)
    \item \(\frac{1}{p_{1}}+\frac{1}{p_{2}}+\frac{1}{p_{3}}+\frac{1}{q_{4}}<\frac{2(d-1)}{d+1}\) when \(\frac{1}{p_{1}},\frac{1}{p_{2}}<\frac{d-1}{d+1},\) but \(\frac{1}{p_{1}}+\frac{1}{p_{2}}>\frac{d-1}{d+1}\).
\end{itemize}
Putting these conditions together, in the case where \(\frac{1}{p_{4}}<\frac{d-1}{d+1}\), we get \(\ell^{p}\) improvement with an exponent of \(\lambda^{\frac{d}{2}(1-\frac{1}{p_{1}}-\frac{1}{p_{2}}-\frac{1}{p_{3}}-\frac{1}{p_{4}})}\) as long as the following conditions hold:

\begin{itemize}
    \item \(\frac{1}{p_{1}}+\frac{1}{p_{2}}+\frac{1}{p_{3}}+\frac{1}{p_{4}}>1\)
    \item \(\left(\frac{1}{p_{1}}+\frac{1}{p_{2}}+\frac{1}{p_{4}}\right)\frac{2}{d-1}+\frac{1}{p_{3}}<1\) when \(\frac{1}{p_{1}}+\frac{1}{p_{2}}<\frac{d-1}{d+1}\)
    \item \(\left(\frac{1}{p_{1}}+\frac{1}{p_{3}}\right)\frac{2}{d-1}+\frac{1}{p_{2}}+\frac{d-1}{2}\left(\frac{1}{p_{4}}\right)<1\) when \(\frac{1}{p_{2}}>\frac{d-1}{d+1}\)
    \item \(\left(\frac{1}{p_{3}}+\frac{1}{p_{2}}\right)\frac{2}{d-1}+\frac{1}{p_{1}}+\frac{d-1}{2}\left(\frac{1}{p_{4}}\right)<1\) when \(\frac{1}{p_{1}}>\frac{d-1}{d+1}\)
    \item \(\frac{1}{p_{1}}+\frac{1}{p_{2}}+\frac{1}{p_{3}}+\frac{2}{d-1}\left(\frac{1}{p_{4}}\right)<\frac{2(d-1)}{d+1}\) when \(\frac{1}{p_{1}},\frac{1}{p_{2}}<\frac{d-1}{d+1},\) but \(\frac{1}{p_{1}}+\frac{1}{p_{2}}>\frac{d-1}{d+1}\).
\end{itemize}

When, \(\frac{1}{p_{4}}\geq\frac{d-1}{d+1}\), we need that \(\frac{1}{p_{4}}>\frac{1}{q_{4}}>\frac{d-1}{2}(\frac{1}{p_{4}}-1)+1\). Again, putting this together with the conditions for \(\left(\frac{1}{p_{1}},\frac{1}{p_{2}},\frac{1}{p_{3}}+\frac{1}{q_{4}}\right)\) satisfying the conditions in Theorem \ref{trithm} to get \(\ell^{p}\) improvement for \(\Ld_{K_{3}}\), we get the following additional conditions for \(\ell^{p}\) improvement for \(\Lambda_{K_{3+t}}\):
\begin{itemize}
    \item \(\frac{1}{p_{1}}+\frac{1}{p_{2}}+\frac{1}{p_{3}}+\frac{1}{p_{4}}>1\)
    \item \(\frac{1}{p_{1}}+\frac{1}{p_{2}}+\frac{1}{p_{3}}+\left(\frac{1}{p_{4}}\right)\frac{2}{d-1}<\frac{d-1}{2}\) when \(\frac{1}{p_{1}}+\frac{1}{p_{2}}<\frac{d-1}{d+1}\)
    \item \(\left(\frac{1}{p_{1}}+\frac{1}{p_{3}}\right)\frac{2}{d-1}+\frac{1}{p_{2}}+\frac{1}{p_{4}}<2-\frac{2}{d-1}\) when \(\frac{1}{p_{2}}>\frac{d-1}{d+1}\)
    \item \(\left(\frac{1}{p_{2}}+\frac{1}{p_{3}}\right)\frac{2}{d-1}+\frac{1}{p_{1}}+\frac{1}{p_{4}}<2-\frac{2}{d-1}\) when \(\frac{1}{p_{1}}>\frac{d-1}{d+1}\)
    \item \(\frac{1}{p_{1}}+\frac{1}{p_{2}}+\frac{1}{p_{3}}+\frac{d-1}{2}\left(\frac{1}{p_{4}}\right)<\frac{2(d-1)}{d+1}+\frac{d-3}{2}\) when \(\frac{1}{p_{1}},\frac{1}{p_{2}}<\frac{d-1}{d+1},\) but \(\frac{1}{p_{1}}+\frac{1}{p_{2}}>\frac{d-1}{d+1}\).
\end{itemize}

Together, the cases where \(\frac{1}{p_{4}}<\frac{d-1}{d+1}\) and \(\frac{1}{p_{4}}\geq\frac{d-1}{d+1}\), we get the conditions in Remark \ref{Ttrmk} which bound the region formed by the convex hull of the points listed in Theorem \ref{Tttheorem}.
\end{proof}

\subsection{Exponent sharpness}
We now establish that \(\frac{d}{2}\left(1-\frac{1}{p_{1}}-\frac{1}{p_{2}}-\frac{1}{p_{3}}-\frac{1}{p_{4}}\right)\) is the smallest exponent that we can guarantee. Take \(f_{1}(x)=f_{2}(x)=f_{3}(x)=f_{4}(x)=B_{\lambda}(x)\). Then \[\|f_{1}\|_{p_{1}}\|f_{2}\|_{p_{2}}\|f_{3}\|_{p_{3}}\|f_{4}\|_{p_{4}}\approx\lambda^{\frac{d}{2}(\frac{1}{p_{1}}+\frac{1}{p_{2}}+\frac{1}{p_{3}}+\frac{1}{p_{4}})},\]
and
\begin{align*}
    \Ld_{K_{3+t}}&(f_{1},f_{2},f_{3},f_{4})\\
    &\approx\frac{1}{|S_{\lambda}|^{2}|S_{\lambda}^{d-2}|}\sum_{x_{1},x_{2},x_{3},x_{4}\in\Z^{d}}\left(\prod_{i=1}^{4}B_{\lambda}(x_{i})\right)S_{\lambda}(x_{1}-x_{2})S_{\lambda}(x_{2}-x_{3})S_{\lambda}(x_{3}-x_{1})S_{\lambda}(x_{3}-x_{4})\\
    &=\frac{1}{|S_{\lambda}||S_{\lambda}^{d-2}|}\sum_{x_{1},x_{2},x_{3}\in\Z^{d}}\left(\prod_{i=1}^{3}B_{\lambda}(x_{i})\right)S_{\lambda}(x_{1}-x_{2})S_{\lambda}(x_{2}-x_{3})S_{\lambda}(x_{3}-x_{1})A_{\lambda}B_{\lambda}(x_{3})\\
    &\approx \frac{1}{|S_{\lambda}||S_{\lambda}^{d-2}|}\sum_{x_{1},x_{2},x_{3}\in\Z^{d}}\left(\prod_{i=1}^{3}B_{\lambda}(x_{i})\right)S_{\lambda}(x_{1}-x_{2})S_{\lambda}(x_{2}-x_{3})S_{\lambda}(x_{3}-x_{1})\\
    &=\frac{1}{|S_{\lambda}||S_{\lambda}^{d-2}|}\sum_{x_{1},x_{3}\in\Z^{d}}B_{\lambda}(x_{1})B_{\lambda}(x_{3})S_{\lambda}(x_{3}-x_{3})\sum_{x_{2}\in\Z^{d}}B_{\lambda}(x_{2})(S_{\lambda}(x_{1}-x_{2})S_{\lambda}(x_{2}-x_{3})\\
    &\approx \frac{1}{|S_{\lambda}|}\sum_{x_{1},x_{3}\in\Z^{d}}B_{\lambda}(x_{1})B_{\lambda}(x_{3})S_{\lambda}(x_{1}-x_{3})\\
    &=\sum_{x_{1}\in\Z^{d}}B_{\lambda}(x_{1})A_{\lambda}B_{\lambda}(x_{1})\approx \sum_{x_{1}\in\Z^{d}}B_{\lambda}(x_{1})\\
    &\approx \lambda^{\frac{d}{2}}\approx \lambda^{\frac{d}{2}(1-\frac{1}{p_{1}}-\frac{1}{p_{2}}-\frac{1}{p_{3}}-\frac{1}{p_{4}})}\|f_{1}\|_{p_{1}}\|f_{2}\|_{p_{2}}\|f_{3}\|_{p_{3}}\|f_{4}\|_{p_{4}}.
\end{align*}

Thus, the exponent on \(\lambda\) is sharp.


\section{The Y-shape, \(Y\)}\label{y section}

Define
\[\Ld_{Y}(f_{1},f_{2},f_{3},f_{4})=\frac{1}{N_{Y}}\sum_{x_{1},\dots,x_{4}\in\Z^{d}}\left(\prod_{i=1}^{4}f_{i}(x_{1})\right)S_{\lambda}(x_{3}-x_{1})S_{\lambda}(x_{3}-x_{2})S_{\lambda}(x_{3}-x_{4}).\]

\begin{theorem}\label{Ytheorem}
    If \(d\geq 5\) and \(1<p_{1},p_{2},p_{3},p_{4}\) with \(\frac{1}{p_{1}}+\frac{1}{p_{2}}+\frac{1}{p_{3}}+\frac{1}{p_{4}}>1\), then there exist constants \(C_{p_{1},p_{2},p_{3},p_{4}}\) such that for all \(\lambda\in\mathcal{R}_{Y}\) we have the \(\ell^{p}\) improving inequality
    \begin{equation}\label{Yeq}\Ld_{Y}(f_{1},f_{2},f_{3},f_{4})\leq C_{p_{1},p_{2},p_{3},p_{4}}\lambda^{\frac{d}{2}(1-\frac{1}{p_{1}}-\frac{1}{p_{2}}-\frac{1}{p_{3}}-\frac{1}{p_{4}})}\|f_{1}\|_{p_{1}}\|f_{2}\|_{p_{2}}\|f_{3}\|_{p_{3}}\|f_{4}\|_{p_{4}}\end{equation}
    provided \(\left(\frac{1}{p_{1}},\frac{1}{p_{2}},\frac{1}{p_{3}},\frac{1}{p_{4}}\right)\) is in the convex hull of the points:
\begin{multicols}{4}
    \begin{itemize}
            \item \((1,0,0,0)\)
            \item \((0,1,0,0)\)
            \item \((0,0,1,0)\)
            \item \((0,0,0,1)\)
            \columnbreak
            \item \((0,\frac{d-1}{d+1},\frac{d-1}{d+1},0)\)
            \item \((0,0,\frac{d-1}{d+1},\frac{d-1}{d+1})\)
            \item \((\frac{d-1}{d+1},0,\frac{d-1}{d+1},0)\)
            \item \((\frac{d-1}{d+1},\frac{d-1}{d+1},\frac{d-5}{d+1},\frac{d-1}{d+1})\)
           \columnbreak
            \item \((\frac{d-1}{d+1}, \frac{d^{2}-5}{d^{2}-1},0,0)\)
            \item \((\frac{d^{2}-5}{d^{2}-1},\frac{d-1}{d+1},0,0)\)
            \item \((0,\frac{d-1}{d+1}, 0, \frac{d^{2}-5}{d^{2}-1})\)
             \item \((0,\frac{d^{2}-5}{d^{2}-1},0,\frac{d-1}{d+1})\)
             \columnbreak
             \item \((\frac{d-1}{d+1},0,0, \frac{d^{2}-5}{d^{2}-1})\)
            \item \((\frac{d^{2}-5}{d^{2}-1},0,0,\frac{d-1}{d+1})\)
            \item \((\frac{d-1}{d+1},\frac{d-1}{d+1},\frac{d-3}{d+1},0)\)
            \item \((0,\frac{d-1}{d+1},\frac{d-3}{d+1},\frac{d-1}{d+1})\)
            \item \((\frac{d-1}{d+1},0,\frac{d-3}{d+1},\frac{d-1}{d+1})\)
        \end{itemize}
\end{multicols}
\end{theorem}

We note that 
\begin{align*}
    \Ld_{Y}(f_{1},f_{2},f_{3},f_{4})&\approx\frac{1}{|S_{\lambda}|^{3}}\sum_{x_{1},\dots,x_{4}\in\Z^{d}}\left(\prod_{i=1}^{4}f_{i}(x_{1})\right)S_{\lambda}(x_{3}-x_{1})S_{\lambda}(x_{3}-x_{2})S_{\lambda}(x_{3}-x_{4})\\
    &=\frac{1}{|S_{\lambda}|^{3}}\sum_{x_{1},x_{2}\in\Z^{d},x_{3}}\left(\prod_{i=1}^{3}f_{i}(x_{1})\right)S_{\lambda}(x_{3}-x_{1})S_{\lambda}(x_{3}-x_{2})\sum_{x_{4}}S_{\lambda}(x_{3}-x_{4})\\
    &=\frac{1}{|S_{\lambda}|^{2}}\sum_{x_{1},\dots,x_{3}}\left(\prod_{i=1}^{3}f_{i}(x_{1})\right)A_{\lambda}f_{4}(x_{3})S_{\lambda}(x_{3}-x_{1})S_{\lambda}(x_{3}-x_{2})\\
    &=\frac{1}{|S_{\lambda}|^{2}}\sum_{x_{1},x_{3}}f_{1}(x_{1})f_{3}(x_{3})A_{\lambda}f_{4}(x_{3})S_{\lambda}(x_{3}-x_{1})\sum_{x_{2}}f_{2}(x_{2})S_{\lambda}(x_{3}-x_{2})\\
    &=\frac{1}{|S_{\lambda}|}\sum_{x_{3}}f_{3}(x_{3})A_{\lambda}f_{4}(x_{3})A_{\lambda}f_{2}(x_{3})\sum_{x_{1}}f_{1}(x_{1})S_{\lambda}(x_{3}-x_{1})\\
    &=\frac{1}{|S_{\lambda}|}\sum_{x_{3}}f_{3}(x_{3})A_{\lambda}f_{4}(x_{3})A_{\lambda}f_{2}(x_{3})A_{\lambda}f_{1}(x_{3})\\
    &=\langle f_{3},A_{\lambda}f_{4}\cdot A_{\lambda}f_{2}\cdot A_{\lambda}f_{1}\rangle
\end{align*}
Assume there exist \(\frac{1}{q_{1}},\frac{1}{q_{2}},\frac{1}{q_{4}}\) such that \(\frac{1}{q_{1}}+\frac{1}{q_{2}}+\frac{1}{p_{3}}+\frac{1}{q_{4}}=1\) and the points \(\left(\frac{1}{p_{i}},\frac{1}{q_{i}}\right)\), (\(i=1,2,4\)), are in the region of \(\ell^{p}\) improving with the exponent \(\frac{d}{2}(\frac{1}{q_{i}}-\frac{1}{p_{i}})\) on \(\lambda\). Then,

\begin{align*}
    \Ld_{Y}(f_{1},f_{2},f_{3},f_{4})&\approx\langle f_{3},A_{\lambda}f_{4}\cdot A_{\lambda}f_{2}\cdot A_{\lambda}f_{1}\rangle\\
    &\leq \|f_{3}\|_{p_{3}}\|A_{\lambda}f_{1}\|_{q_{1}}\|A_{\lambda}f_{2}\|_{q_{2}}\|A_{\lambda}f_{4}\|_{q_{4}}\\
    &\lesssim \lambda^{\frac{d}{2}(\frac{1}{q_{1}}-\frac{1}{p_{1}}+\frac{1}{q_{2}}-\frac{1}{p_{2}}+\frac{1}{q_{4}}-\frac{1}{p_{4}})}\|f_{1}\|_{p_{1}}\|f_{2}\|_{p_{2}}\|f_{3}\|_{p_{3}}\|f_{4}\|_{p_{4}}\\
    &=\lambda^{\frac{d}{2}(1-\frac{1}{p_{1}}-\frac{1}{p_{2}}-\frac{1}{p_{3}}-\frac{1}{p_{4}})}\|f_{1}\|_{p_{1}}\|f_{2}\|_{p_{2}}\|f_{3}\|_{p_{3}}\|f_{4}\|_{p_{4}}.
\end{align*}

Finding when such \(\frac{1}{q_{1}},\frac{1}{q_{2}},\frac{1}{q_{4}}\) exist breaks down into 8 cases. Following the methods outlined in the case of \(P_{2}\), we find that such \(\frac{1}{q_{1}},\frac{1}{q_{2}},\frac{1}{q_{4}}\) exist when \(\left(\frac{1}{p_{1}},\frac{1}{p_{2}},\frac{1}{p_{3}},\frac{1}{p_{4}}\right)\) is in the convex hull of the following 17 points:

\begin{multicols}{4}
    \begin{itemize}
            \item \((1,0,0,0)\)
            \item \((0,1,0,0)\)
            \item \((0,0,1,0)\)
            \item \((0,0,0,1)\)
            \columnbreak
            \item \((0,\frac{d-1}{d+1},\frac{d-1}{d+1},0)\)
            \item \((0,0,\frac{d-1}{d+1},\frac{d-1}{d+1})\)
            \item \((\frac{d-1}{d+1},0,\frac{d-1}{d+1},0)\)
            \item \((\frac{d-1}{d+1},\frac{d-1}{d+1},\frac{d-5}{d+1},\frac{d-1}{d+1})\)
           \columnbreak
            \item \((\frac{d-1}{d+1}, \frac{d^{2}-5}{d^{2}-1},0,0)\)
            \item \((\frac{d^{2}-5}{d^{2}-1},\frac{d-1}{d+1},0,0)\)
            \item \((0,\frac{d-1}{d+1}, 0, \frac{d^{2}-5}{d^{2}-1})\)
             \item \((0,\frac{d^{2}-5}{d^{2}-1},0,\frac{d-1}{d+1})\)
             \columnbreak
             \item \((\frac{d-1}{d+1},0,0, \frac{d^{2}-5}{d^{2}-1})\)
            \item \((\frac{d^{2}-5}{d^{2}-1},0,0,\frac{d-1}{d+1})\)
            \item \((\frac{d-1}{d+1},\frac{d-1}{d+1},\frac{d-3}{d+1},0)\)
            \item \((0,\frac{d-1}{d+1},\frac{d-3}{d+1},\frac{d-1}{d+1})\)
            \item \((\frac{d-1}{d+1},0,\frac{d-3}{d+1},\frac{d-1}{d+1})\)
        \end{itemize}
\end{multicols}

\subsection{Exponent sharpness}
To establish that \(\frac{d}{2}\left(1-\frac{1}{p_{1}}-\frac{1}{p_{2}}-\frac{1}{p_{3}}-\frac{1}{p_{4}}\right)\) is the smallest exponent that we can guarantee, take \(f_{1}(x)=f_{2}(x)=f_{3}(x)=f_{4}(x)=B_{\lambda}(x)\). Then \[\|f_{1}\|_{p_{1}}\|f_{2}\|_{p_{2}}\|f_{3}\|_{p_{3}}\|f_{4}\|_{p_{4}}\approx\lambda^{\frac{d}{2}(\frac{1}{p_{1}}+\frac{1}{p_{2}}+\frac{1}{p_{3}}+\frac{1}{p_{4}})},\]
and
\begin{align*}
    \Ld_{Y}&(f_{1},f_{2},f_{3},f_{4})\\
    &\approx\frac{1}{|S_{\lambda}|^{3}}\sum_{x_{1},x_{2},x_{3},x_{4}\in\Z^{d}}\left(\prod_{i=1}^{4}B_{\lambda}(x_{i})\right)S_{\lambda}(x_{1}-x_{3})S_{\lambda}(x_{2}-x_{3})S_{\lambda}(x_{3}-x_{4})\\
    &=\frac{1}{|S_{\lambda}|^{3}}\sum_{x_{3}}B_{\lambda}(x_{3})\left(\sum_{x_{1}}B_{\lambda}(x_{1})S_{\lambda}(x_{1}-x_{3})\right)\left(\sum_{x_{2}}B_{\lambda}(x_{2})S_{\lambda}(x_{2}-x_{3})\right)\left(\sum_{x_{4}}B_{\lambda}(x_{4})S_{\lambda}(x_{4}-x_{3})\right)\\
    &=\sum_{x_{3}}B_{\lambda}(x_{3})A_{\lambda}B_{\lambda}(x_{3})\approx \sum_{x_{3}}B_{\lambda}(x_{3})\\
    &\approx \lambda^{\frac{d}{2}}\approx\lambda^{\frac{d}{2}(1-\frac{1}{p_{1}}-\frac{1}{p_{2}}-\frac{1}{p_{3}}-\frac{1}{p_{4}})}\|f_{1}\|_{p_{1}}\|f_{2}\|_{p_{2}}\|f_{3}\|_{p_{3}}\|f_{4}\|_{p_{4}}
\end{align*}

Thus, the exponent on \(\lambda\) is sharp.


\section{The \(k\)-chain, \(P_{k}\)}\label{chain section}

        For \(\lambda\in\mathcal{R}_{P_{k}}\), \(k\geq 1\) and  \(f_{i}\) functions on \(\Z^{d}\) for \(i=1,...,k+1\), we define \[\Ld_{P_{k}}(f_{1},...,f_{k}, f_{k+1})=\frac{1}{N_{P_{k}}}\sum_{x_{1},...,x_{k+1}\in\Z^{d}}\left(\prod_{i=1}^{k}S_{\lambda}(x_{i}-x_{i+1})\right)\left(\prod_{i=1}^{k+1}f_{i}(x_{i})\right).\]

    \begin{theorem}\label{kchainthm}
    If \(d\geq 5\), and \(\infty>p_{1},...,p_{k+1}>1\) are such that \(\sum_{i=1}^{k+1}\frac{1}{p_{i}}>1\), then there exist constants \(C\) such that for all \(\lambda\in\mathcal{R}_{P_{k}}\) we have the \(\ell^{p}\) improving inequality \[\Ld_{P_{k}}(f_{1},...,f_{k+1})\leq C\lambda^{\frac{d}{2}\left(1-\sum_{i=1}^{k+1}\frac{1}{p_{i}}\right)}\prod_{i=1}^{k+1}\|f_{i}\|_{p_{i}}\] provided that one of the following holds:
    \begin{enumerate}
    \item \(\frac{1}{p_{2}}+\frac{2}{d-1}\left(\frac{1}{p_{1}}+\sum_{i=3}^{k+1}\frac{1}{p_{i}}\right)<1\) and \(\frac{1}{p_{1}}, \sum_{i=3}^{k+1}\frac{1}{p_{i}}\leq \frac{d-1}{d+1}\).
    \item \(\frac{2}{d-1}\left(\sum_{i=3}^{k+1}\frac{1}{p_{i}}\right)+\frac{1}{p_{2}}+\frac{d-1}{2}\left(\frac{1}{p_{1}}\right)<\frac{d-1}{2}\) and \(\frac{1}{p_{1}}>\frac{d-1}{d+1}\geq\sum_{i=3}^{k+1}\frac{1}{p_{i}}\)
    \item \(\frac{2}{d-1}\left(\frac{1}{p_{1}}\right)+\frac{1}{p_{2}}+\frac{d-1}{2}\sum_{i=3}^{k+1}\frac{1}{p_{i}}<\frac{d-1}{2} \) and \(\sum_{i=3}^{k+1}\frac{1}{p_{i}}>\frac{d-1}{d+1}\geq\frac{1}{p_{1}}\).
    \item \(\frac{1}{p_{1}}+\sum_{i=3}^{k+1}\frac{1}{p_{i}}+\frac{2}{d-1}\left(\frac{1}{p_{2}}\right)<2-\frac{2}{d-1}\) and \(\sum_{i=3}^{k+1}\frac{1}{p_{i}},\frac{1}{p_{1}}>\frac{d-1}{d+1}\).
    \end{enumerate}
    \end{theorem}

\begin{proof}
    We proceed through induction.\\
    When \(k=1\), Theorem \ref{kchainthm} is equivalent to Theorem \ref{Lthm}, and when \(k=2\), Theorem \ref{kchainthm} is equivalent to Theorem \ref{2chain}.\\
    Now, assume that Theorem \ref{kchainthm} holds for \(k=n-1\) for some \(n\in\N\). Let \(k=n\), \(\lambda\in\Z\), and \(f_{i}\) be functions on \(\Z^{d}\) for \(i=1,...,k+1\) with \(d\geq5\), and \(\infty>p_{1},...,p_{k+1}>1\). Assume \(\sum_{i=1}^{k+1}\frac{1}{p_{k+1}}>1\) and one of the following holds:
    \begin{enumerate}
    \item \(\frac{1}{p_{2}}+\frac{2}{d-1}\left(\frac{1}{p_{1}}+\sum_{i=3}^{k+1}\frac{1}{p_{i}}\right)<1\) and \(\frac{1}{p_{1}}, \sum_{i=3}^{k+1}\frac{1}{p_{i}}\leq \frac{d-1}{d+1}\).
    \item \(\frac{2}{d-1}\left(\sum_{i=3}^{k=1}\frac{1}{p_{i}}\right)+\frac{1}{p_{2}}+\frac{d-1}{2}\left(\frac{1}{p_{1}}\right)<\frac{d-1}{2}\) and \(\frac{1}{p_{1}}>\frac{d-1}{d+1}\geq\sum_{i=3}^{k+1}\frac{1}{p_{i}}\)
    \item \(\frac{2}{d-1}\left(\frac{1}{p_{1}}\right)+\frac{1}{p_{2}}+\frac{d-1}{2}\sum_{i=3}^{k+1}\frac{1}{p_{i}}<\frac{d-1}{2} \) and \(\sum_{i=3}^{k+1}\frac{1}{p_{i}}>\frac{d-1}{d+1}\geq\frac{1}{p_{1}}\).
    \item \(\frac{1}{p_{1}}+\sum_{i=3}^{k+1}\frac{1}{p_{i}}+\frac{2}{d-1}\left(\frac{1}{p_{2}}\right)<2-\frac{2}{d-1}\) and \(\sum_{i=3}^{k+1}\frac{1}{p_{i}},\frac{1}{p_{1}}>\frac{d-1}{d+1}\).
    \end{enumerate}

    This is equivalent to saying that one of the following holds:
    \begin{enumerate}
    \item \(\frac{1}{p_{2}}+\frac{2}{d-1}\left(\frac{1}{p_{1}}+\sum_{i=3}^{k-1}\frac{1}{p_{i}}+\left(\frac{1}{p_{k}}+\frac{1}{p_{k+1}}\right)\right)<1\) and \(\frac{1}{p_{1}}, \sum_{i=3}^{k-1}\frac{1}{p_{i}}+\left(\frac{1}{p_{k}}+\frac{1}{p_{k+1}}\right)\leq \frac{d-1}{d+1}\).
    \item \(\frac{2}{d-1}\left(\sum_{i=3}^{k-1}\frac{1}{p_{i}}+\left(\frac{1}{p_{k}}+\frac{1}{p_{k+1}}\right)\right)+\frac{1}{p_{2}}+\frac{d-1}{2}\left(\frac{1}{p_{1}}\right)<\frac{d-1}{2}\) and \(\frac{1}{p_{1}}>\frac{d-1}{d+1}\geq\sum_{i=3}^{k-1}\frac{1}{p_{i}}+\left(\frac{1}{p_{k}}+\frac{1}{p_{k+1}}\right)\)
    \item \(\frac{2}{d-1}\left(\frac{1}{p_{1}}\right)+\frac{1}{p_{2}}+\frac{d-1}{2}\sum_{i=3}^{k-1}\frac{1}{p_{i}}+\left(\frac{1}{p_{k}}+\frac{1}{p_{k+1}}\right)<\frac{d-1}{2} \) and \(\sum_{i=3}^{k+1}\frac{1}{p_{i}}>\frac{d-1}{d+1}\geq\frac{1}{p_{1}}\).
    \item \(\frac{1}{p_{1}}+\sum_{i=3}^{k-1}\frac{1}{p_{i}}+\left(\frac{1}{p_{k}}+\frac{1}{p_{k+1}}\right)+\frac{2}{d-1}\left(\frac{1}{p_{2}}\right)<2-\frac{2}{d-1}\) and \(\sum_{i=3}^{k-1}\frac{1}{p_{i}}+\left(\frac{1}{p_{k}}+\frac{1}{p_{k+1}}\right),\frac{1}{p_{1}}>\frac{d-1}{d+1}\).
    \end{enumerate}
    That is to say, that the point \(\left(\frac{1}{p_{1}},\frac{1}{p_{2}},\dots,\frac{1}{p_{k-1}},\frac{1}{p_{k}}+\frac{1}{p_{k+1}}\right)\) is in the region of \(\ell^{p}\) improvement for \(\Lambda_{P_{k-1}}\) by the induction hypothesis.\\

     We have that
        \begin{align*}
            \Ld_{P_{k}}&(f_{1},f_{2},...f_{k+1})\\
            &\approx\frac{1}{|S_{\lambda}|^{k}}\sum_{x_{1},...,x_{k+1}\in\Z^{d}}\left(\prod_{i=1}^{k}S_{\lambda}(x_{i}-x_{i+1})\right)\left(\prod_{i=1}^{k+1}f_{i}(x_{i})\right)\\
            &=\frac{1}{|S_{\lambda}|^{k-1}}\sum_{x_{1},...,x_{k}\in\Z^{d}}\left(\prod_{i=1}^{k-1}S_{\lambda}(x_{i}-x_{i+1})\right)\left(\prod_{i=1}^{k-1}f_{i}(x_{i})\right)f_{k}(x_{k})\frac{1}{|S_{\lambda}|}\sum_{x_{k+1}\in\Z^{d}}S_{\lambda}(x_{k}-x_{k+1})f_{k+1}(x_{k+1})\\
            &=\frac{1}{|S_{\lambda}|^{k-1}}\sum_{x_{1},...,x_{k}\in\Z^{d}}\left(\prod_{i=1}^{k-1}S_{\lambda}(x_{i}-x_{i+1})\right)\left(\prod_{i=1}^{k-1}f_{i}(x_{i})\right)f_{k}(x_{k})A_{\lambda}f_{k+1}(x_{k})\\
            &\approx\Ld_{P_{k-1}}(f_{1},...f_{k-1},f_{k}\cdot A_{\lambda}f_{k+1}).\\
        \end{align*}
        
    Then, as the point \(\left(\frac{1}{p_{1}},\frac{1}{p_{2}},\dots,\frac{1}{p_{k-1}},\frac{1}{p_{k}}+\frac{1}{p_{k+1}}\right)\) is in the region of \(\ell^{p}\) improvement with the desired exponent for \(\Ld_{P_{k-1}}\), we have that
    \begin{align*}
        \Ld_{P_{k}}(f_{1},f_{2},...f_{k})&\approx\Ld_{P_{k-1}}(f_{1},...f_{k-1},f_{k}\cdot A_{\lambda}f_{k+1})\\
        &\lesssim\lambda^{\frac{d}{2}(1-\sum_{i=1}^{k+1}\frac{1}{p_{i}})}\|f_{1}\|_{p_{1}}...\|f_{k-1}\|_{p_{k-1}}\|f_{k}\cdot A_{\lambda}f_{k+1}\|_{(\frac{1}{p_{k}}+\frac{1}{p_{k+1}})^{-1}}\\
       &\leq \lambda^{\frac{d}{2}(1-\sum_{i=1}^{k+1}\frac{1}{p_{i}})}\|f_{1}\|_{p_{1}}...\|f_{k-1}\|_{p_{k-1}}\|f_{k}\|_{p_{k}}\|A_{\lambda}f_{k+1}\|_{p_{k+1}}\\
        &\lesssim \lambda^{\frac{d}{2}(1-\sum_{i=1}^{k+1}\frac{1}{p_{i}})}\|f_{1}\|_{p_{1}}...\|f_{k-1}\|_{p_{k-1}}\|f_{k}\|_{p_{k}}\|f_{k+1}\|_{p_{k+1}}.
    \end{align*}
    
    Thus, if Theorem \ref{kchainthm} holds for \(k=n-1\), then it holds for \(k=n\). \\
    Therefore, by induction, we have proved Theorem \ref{kchainthm}.

\end{proof}

\subsection{Exponent Sharpness}

To show that the exponent in Theorem \ref{kchainthm} is the best we can guarantee, take \(f_{i}=B_{\lambda}\), that is, a ball of radius \(\sqrt{\lambda}\) for \(1\leq i \leq k+1\). Then \[\prod_{i=1}^{k+1}\|f_{i}\|_{p_{i}}\approx \lambda^{\frac{d}{2}\sum_{i=1}^{k+1}\frac{1}{p_{i}}}.\]
Now,
\begin{align*}
\Ld_{P_{k}}(f_{1},...,f_{k+1})&\approx\frac{1}{|S_{\lambda}|^{k}}\sum_{x_{1},...,x_{k+1}\in\Z^{d}}\left(\prod_{i=1}^{k}S_{\lambda}(x_{i}-x_{i+1})\right)\left(\prod_{i=1}^{k+1}B_{\lambda}(x_{i})\right)\\
&=\frac{1}{|S_{\lambda}|^{k}}\sum_{x_{1},...,x_{k}\in\Z^{d}}\left(\prod_{i=1}^{k-1}S_{\lambda}(x_{i}-x_{i+1})\right)\left(\prod_{i=1}^{k}B_{\lambda}(x_{i})\right)\sum_{x_{k+1}\in\Z^{d}}S_{\lambda}(x_{k}-x_{k+1})B_{\lambda}(x_{k+1})\\
&=\frac{1}{|S_{\lambda}|^{k-1}}\sum_{x_{1},...,x_{k}\in\Z^{d}}\left(\prod_{i=1}^{k-1}S_{\lambda}(x_{i}-x_{i+1})\right)\left(\prod_{i=1}^{k}B_{\lambda}(x_{i})\right)A_{\lambda}B_{\lambda}(x_{k+1})\\
&\approx \frac{1}{|S_{\lambda}|^{k-1}}\sum_{x_{1},...,x_{k}\in\Z^{d}}\left(\prod_{i=1}^{k-1}S_{\lambda}(x_{i}-x_{i+1})\right)\left(\prod_{i=1}^{k}B_{\lambda}(x_{i})\right)\\
&\vdots\\
&\approx \frac{1}{|S_{\lambda}|}\sum_{x_{1}\in\Z^{d}}B_{\lambda}(x_{1})\sum_{x_{2}\in\Z^{d}}B_{\lambda}(x_{2})S_{\lambda}(x_{1}-x_{2})\\
&=\sum_{x_{1}\in\Z^{d}}B_{\lambda}(x_{1})A_{\lambda}(B_{R})(x_{1})\approx \sum_{x_{1}\in\Z^{d}}B_{\lambda}(x_{1})\\
&\approx \lambda^{\frac{d}{2}}\approx\lambda^{\frac{d}{2}(1-\sum_{i=1}^{k+1}\frac{1}{p_{i}})}\prod_{i=1}^{k+1}\|f_{i}\|_{p_{i}}.
\end{align*}

Thus, the exponent on \(\lambda\) is sharp.

\section{The \(k\)-simplex, \(K_{k}\)}\label{sim section}

Given \(k\in\N\), for \(d\geq 2k+1\), we define

    \[\Ld_{K_{k}}(f_{1},f_{2},...,f_{k})=\frac{1}{N_{K_{k}}}\sum_{x_{1},...x_{k}\in\Z^{d}}\prod_{i=1}^{k}f_{i}(x_{i})\left(\prod_{i=1}^{k-1}\prod_{j=i+1}^{k}S_{\lambda}(x_{i}-x_{j})\right).\]

We have the following theorem:
\begin{theorem}\label{simtheorem}
    If \(d\geq 2k+1\) and \(1<p_{1},p_{2},...,p_{k}<\infty\) with \(\sum_{i=1}^{k}\frac{1}{p_{i}}>1\), then there exist constants \(C\) such that for all \(\lambda\in\mathcal{R}_{K_{k}}\) we have the \(\ell^{p}\) improving inequality
    \begin{equation}\label{simeq}\Ld_{K_{k}}(f_{1},f_{2},...,f_{k})\leq C\lambda^{\frac{d}{2}(1-\frac{1}{p_{1}}-\sum_{i=1}^{k}\frac{1}{p_{i}})}\prod_{i=1}^{k}\|f_{i}\|_{p_{i}}\end{equation}
    provided \(\left(\frac{1}{p_{1}},\frac{1}{p_{2}},...,\frac{1}{p_{k}}\right)\) is in the convex hull of the points \(x_{1},...,x_{k}\) described by
    \begin{itemize}
        \item For \(1\leq i\leq k\), \(x_{i}=1\), and \(x_{j}=0\) for \(j\neq i\), \(1\leq j\leq k\).
        \item For \(1\leq i,j\leq k\), \(x_{i}=x_{j}=\frac{d-1}{d+1}\), and \(x_{n}=0\) for \(n\neq i,j\), \(1\leq n\leq k\).
    \end{itemize}

\end{theorem}

The cases where \(k=2,3,4\) are the 1-chain, the triangle, and the tetrahedron respectively. To establish this for an arbitrary \(k\in\N\), we will proceed through induction, just as we used the triangle to establish \(\ell^{p}\) improving results for the tetrahedron.

\begin{proof}
    Suppose that \(k\in\N\), \(k\geq 4\) is such that Theorem \ref{simtheorem} holds. Suppose that \(\left(\frac{1}{p_{1}},\frac{1}{p_{2}},...,\frac{1}{p_{k}},\frac{1}{p_{k+1}}\right)\) is in the convex hull of the points \(x_{1},...,x_{k},x_{k+1}\) described by
    \begin{itemize}
        \item For \(1\leq i\leq k+1\), \(x_{i}=1\), and \(x_{j}=0\) for \(j\neq i\), \(1\leq j\leq k+1\).
        \item For \(1\leq i,j\leq k+1\), \(x_{i}=x_{j}=\frac{d-1}{d+1}\), and \(x_{n}=0\) for \(n\neq i,j\), \(1\leq n\leq k+1\).
    \end{itemize}
    Then, for each index, \(1\leq m\leq k+1\), \(\left(\frac{1}{p_{1}},\frac{1}{p_{2}},...,\frac{1}{p_{m-1}},\frac{1}{p_{m+1}},...,\frac{1}{p_{k}},\frac{1}{p_{k+1}}\right)\) is in the region of \(\ell^{p}\) improving described in Theorem \ref{simtheorem}.\\
    We will show the case where \(m=k+1\). The other \(k\) cases follow by symmetry.\\
    We have that 
    \begin{align*}
        \Ld_{K_{k+1}}(f_{1},f_{2},...,f_{k+1})&\approx\frac{1}{\prod_{i=1}^{k}|S_{\lambda}^{d-i}|}\sum_{x_{1},...x_{k+1}\in\Z^{d}}\prod_{i=1}^{k+1}f_{i}(x_{i})\left(\prod_{i=1}^{k}\prod_{j=i+1}^{k+1}S_{\lambda}(x_{i}-x_{j})\right)\\
        &=\frac{1}{\prod_{i=1}^{k}|S_{\lambda}^{d-i}|}\sum_{x_{1},...x_{k}\in\Z^{d}}\prod_{i=1}^{k}f_{i}(x_{i})\left(\prod_{i=1}^{k-1}\prod_{j=i+1}^{k}S_{\lambda}(x_{i}-x_{j})\right)\\&\quad\quad\quad\cdot\sum_{x_{k+1}\in\Z^{d}}f_{k}(x_{k+1})\prod_{i=1}^{k}(x_{i}-x_{k+1})\\
        &\lesssim \frac{\|f_{k+1}\|_{\infty}}{\prod_{i=1}^{k-1}|S_{\lambda}^{d-i}|}\sum_{x_{1},...x_{k}\in\Z^{d}}\prod_{i=1}^{k}f_{i}(x_{i})\left(\prod_{i=1}^{k-1}\prod_{j=i+1}^{k}S_{\lambda}(x_{i}-x_{j})\right)\\
        &\approx\|f_{k+1}\|_{\infty}\Lambda_{K_{k}}(f_{1},...,f_{k})
    \end{align*}
    Thus, since we have that \((\frac{1}{p_{1}},...,\frac{1}{p_{k}})\) is in the region of \(\ell^{p}\) improving for \(\Lambda_{K_{k}}\), \[\Ld_{K_{k+1}}(f_{1},f_{2},...,f_{k+1})\lesssim\lambda^{\frac{d}{2}(1-\sum_{i=1}^{k}\frac{1}{p_{i}}+0)}\prod{i=1}^{k}\|f_{k}\|_{p_{k}}\|f_{k+1}\|_{\infty}.\]
    Then, by exploiting the symmetry of the simplex and interpolating between the results, as in the proofs of Theorem \ref{trithm} and Theorem \ref{tetrtheorem}, we get that 
    \[\Ld_{K_{k}}(f_{1},f_{2},...,f_{k})\lesssim\lambda^{\frac{d}{2}(1-\frac{1}{p_{1}}-\sum_{i=1}^{k}\frac{1}{p_{i}})}\prod_{i=1}^{k}\|f_{i}\|_{p_{i}}.\]

\end{proof}
\subsection{Exponent Sharpness}
To show the sharpness of the exponent on \(\lambda\), take \(f_{1}(x)=f_{2}(x)=...=f_{k}(x)=B_{\lambda}(x)\). Then \[\prod_{i=1}^{k}\|f_{i}\|_{p_{i}}\approx\lambda^{\frac{d}{2}(\sum_{i=1}^{k}\frac{1}{p_{i}})}\]
Additionally,
    \[\Ld_{K_{k}}(f_{1},f_{2},...,f_{k})\approx\frac{1}{\prod_{i=1}^{k-1}|S_{\lambda}^{d-i}|}\sum_{x_{1},...x_{k}\in\Z^{d}}\prod_{i=1}^{k}B_{\lambda}(x_{i})\left(\prod_{i=1}^{k-1}\prod_{j=i+1}^{k}S_{\lambda}(x_{i}-x_{j})\right).\]
Scaling up the reasoning we used for \(\Ld_{K_{3}}\), we have that \[\sum_{x_{1},...,x_{k}\in\Z^{d}}B_{\lambda}(x_{k})\prod_{i=1}^{k-1}S_{\lambda}(x_{i}-x_{k})\approx |S_{\lambda}^{d-k+1}|\prod_{i=1}^{k-1}B_{\lambda}(x_{i}).\]
Thus, 
\begin{align*}
    \Ld_{K_{k}}(f_{1},f_{2},...,f_{k})&\approx\frac{1}{\prod_{i=1}^{k-1}|S_{\lambda}^{d-i}|}\sum_{x_{1},...x_{k}\in\Z^{d}}\prod_{i=1}^{k}B_{\lambda}(x_{i})\left(\prod_{i=1}^{k-1}\prod_{j=i+1}^{k}S_{\lambda}(x_{i}-x_{j})\right)\\
    &=\frac{1}{\prod_{i=1}^{k-1}|S_{\lambda}^{d-i}|}\sum_{x_{1},...x_{k-1}\in\Z^{d}}\prod_{i=1}^{k-1}B_{\lambda}(x_{i})\left(\prod_{i=1}^{k-2}\prod_{j=i+1}^{k-1}S_{\lambda}(x_{i}-x_{j})\right)\\
    &\quad\quad\quad\quad\quad\cdot\sum_{x_{k}\in\Z^{d}}B_{\lambda}(x_{k})\prod_{i=1}^{k-1}S_{\lambda}(x_{i}-x_{k})\\
    &\approx \frac{1}{\prod_{i=1}^{k-2}|S_{\lambda}^{d-i}|}\sum_{x_{1},...x_{k-1}\in\Z^{d}}\prod_{i=1}^{k-1}B_{\lambda}(x_{i})\left(\prod_{i=1}^{k-2}\prod_{j=i+1}^{k-1}S_{\lambda}(x_{i}-x_{j})\right)\\
    &\vdots\\
    &\approx\frac{1}{|S_{\lambda}|}\sum_{x_{1},x_{2}\in\Z^{d}}B_{\lambda}(x_{1})B_{\lambda}(x_{2})S_{\lambda}(x_{1}-x_{2})\\
    &= \sum_{x_{1}\in\Z^{d}}B_{\lambda}(x_{1})A_{\lambda}B_{\lambda}(x_{2})\approx  \sum_{x_{1}\in\Z^{d}}B_{\lambda}(x_{1})\\
    &\approx\lambda^{\frac{d}{2}}=\lambda^{\frac{d}{2}(1-\sum_{i=1}^{k}\frac{1}{p_{i}})}\lambda^{\frac{d}{2}\sum_{i=1}^{k}\frac{1}{p_{i}}}\\
    &\approx\lambda^{\frac{d}{2}(1-\sum_{i=1}^{k}\frac{1}{p_{i}})}\prod_{i=1}^{k}\|f_{i}\|_{p_{i}}.
\end{align*}
Thus, the exponent on \(\lambda\) of \(\frac{d}{2}(1-\sum_{i=1}^{k}\frac{1}{p_{i}})\) is the best that we can guarantee.


\bibliographystyle{plain}
\bibliography{references}

@article {BIKP25,
    AUTHOR = {Bhowmik, Pablo and Iosevich, Alex and Koh, Doowon and Pham,
              Thang},
     TITLE = {Multi-linear forms, graphs, and {$L^p$}-improving measures in
              {$\Bbb F_q^d$}},
   JOURNAL = {Canad. J. Math.},
  FJOURNAL = {Canadian Journal of Mathematics. Journal Canadien de
              Math\'{e}matiques},
    VOLUME = {77},
      YEAR = {2025},
    NUMBER = {1},
     PAGES = {208--251},
      ISSN = {0008-414X,1496-4279},
   MRCLASS = {42B05},
  MRNUMBER = {4861735},
       DOI = {10.4153/S0008414X2300086X},
       URL = {https://doi.org/10.4153/S0008414X2300086X},
}

@article {KL20,
    AUTHOR = {Kesler, R. and Lacey, M. T.},
     TITLE = {{$\ell^p$}-improving inequalities for discrete spherical
              averages},
   JOURNAL = {Anal. Math.},
  FJOURNAL = {Analysis Mathematica},
    VOLUME = {46},
      YEAR = {2020},
    NUMBER = {1},
     PAGES = {85--95},
      ISSN = {0133-3852,1588-273X},
   MRCLASS = {42B25},
  MRNUMBER = {4064582},
MRREVIEWER = {Stefan\ Steinerberger},
       DOI = {10.1007/s10476-020-0019-9},
       URL = {https://doi.org/10.1007/s10476-020-0019-9},
}

@article {H20,
    AUTHOR = {Hughes, Kevin},
     TITLE = {{$\ell^p$}-improving for discrete spherical averages},
   JOURNAL = {Ann. H. Lebesgue},
  FJOURNAL = {Annales Henri Lebesgue},
    VOLUME = {3},
      YEAR = {2020},
     PAGES = {959--980},
      ISSN = {2644-9463},
   MRCLASS = {42B25 (11L07 11P05 11P55 37A44 37A46)},
  MRNUMBER = {4149830},
MRREVIEWER = {Yu\ Liu},
       DOI = {10.5802/ahl.50},
       URL = {https://doi.org/10.5802/ahl.50},
}

@article {CLM21,
    AUTHOR = {Cook, Brian and Lyall, Neil and Magyar, \'{A}kos},
     TITLE = {Multilinear maximal operators associated to simplices},
   JOURNAL = {J. Lond. Math. Soc. (2)},
  FJOURNAL = {Journal of the London Mathematical Society. Second Series},
    VOLUME = {104},
      YEAR = {2021},
    NUMBER = {4},
     PAGES = {1491--1514},
      ISSN = {0024-6107,1469-7750},
   MRCLASS = {42B25},
  MRNUMBER = {4339943},
MRREVIEWER = {Vjekoslav\ Kova\v{c}},
       DOI = {10.1112/jlms.12467},
       URL = {https://doi.org/10.1112/jlms.12467},
}

@article {MSW02,
    AUTHOR = {Magyar, A. and Stein, E. M. and Wainger, S.},
     TITLE = {Discrete analogues in harmonic analysis: spherical averages},
   JOURNAL = {Ann. of Math. (2)},
  FJOURNAL = {Annals of Mathematics. Second Series},
    VOLUME = {155},
      YEAR = {2002},
    NUMBER = {1},
     PAGES = {189--208},
      ISSN = {0003-486X,1939-8980},
   MRCLASS = {42B25 (11K70 43A07)},
  MRNUMBER = {1888798},
MRREVIEWER = {Loukas\ Grafakos},
       DOI = {10.2307/3062154},
       URL = {https://doi.org/10.2307/3062154},
}

@article {IPWZ25,
    AUTHOR = {Iosevich, A. and Palsson, E. and Wyman, E. and Zhai, Y.},
     TITLE = {Multi-linear forms, structure of graphs and {L}ebesgue spaces},
   JOURNAL = {Math. Z.},
  FJOURNAL = {Mathematische Zeitschrift},
    VOLUME = {310},
      YEAR = {2025},
    NUMBER = {4},
     PAGES = {Paper No. 76, 25},
      ISSN = {0025-5874,1432-1823},
   MRCLASS = {47A07 (05C76 42B20)},
  MRNUMBER = {4915248},
MRREVIEWER = {Duong\ Quoc\ Huy},
       DOI = {10.1007/s00209-025-03761-3},
       URL = {https://doi.org/10.1007/s00209-025-03761-3},
}

@article {AKP24,
    AUTHOR = {Anderson, Theresa C. and Kumchev, Angel V. and Palsson,
              Eyvindur A.},
     TITLE = {A framework for discrete bilinear spherical averages and
              applications to {$\ell^p$}-improving estimates},
   JOURNAL = {Colloq. Math.},
  FJOURNAL = {Colloquium Mathematicum},
    VOLUME = {175},
      YEAR = {2024},
    NUMBER = {1},
     PAGES = {55--76},
      ISSN = {0010-1354,1730-6302},
   MRCLASS = {42B25 (47A07)},
  MRNUMBER = {4731985},
       DOI = {10.4064/cm9216-1-2024},
       URL = {https://doi.org/10.4064/cm9216-1-2024},
}

@article {AKP22,
    AUTHOR = {Anderson, Theresa C. and Kumchev, Angel V. and Palsson,
              Eyvindur A.},
     TITLE = {Discrete maximal operators over surfaces of higher
              codimension},
   JOURNAL = {Matematica},
  FJOURNAL = {La Matematica. Official Journal of the Association for Women
              in Mathematics},
    VOLUME = {1},
      YEAR = {2022},
    NUMBER = {2},
     PAGES = {442--479},
      ISSN = {2730-9657},
   MRCLASS = {11K70 (11L07 11P55 42B25)},
  MRNUMBER = {4445931},
MRREVIEWER = {Xianchang\ Meng},
       DOI = {10.1007/s44007-021-00017-4},
       URL = {https://doi.org/10.1007/s44007-021-00017-4},
}

@article {IPS22,
    AUTHOR = {Iosevich, Alex and Palsson, Eyvindur Ari and Sovine, Sean R.},
     TITLE = {Simplex averaging operators: quasi-{B}anach and
              {$L^p$}-improving bounds in lower dimensions},
   JOURNAL = {J. Geom. Anal.},
  FJOURNAL = {Journal of Geometric Analysis},
    VOLUME = {32},
      YEAR = {2022},
    NUMBER = {3},
     PAGES = {Paper No. 87, 16},
      ISSN = {1050-6926,1559-002X},
   MRCLASS = {42B20},
  MRNUMBER = {4363760},
MRREVIEWER = {B.\ S.\ Rubin},
       DOI = {10.1007/s12220-021-00843-6},
       URL = {https://doi.org/10.1007/s12220-021-00843-6},
}

@inproceedings {L73,
    AUTHOR = {Littman, Walter},
     TITLE = {{$L^{p}-L^{q}$}-estimates for singular integral operators
              arising from hyperbolic equations},
 BOOKTITLE = {Partial differential equations ({P}roc. {S}ympos. {P}ure
              {M}ath., {V}ol. {XXIII}, {U}niv. {C}alifornia, {B}erkeley,
              {C}alif., 1971)},
    SERIES = {Proc. Sympos. Pure Math., Vol. XXIII},
     PAGES = {479--481},
 PUBLISHER = {Amer. Math. Soc., Providence, RI},
      YEAR = {1973},
   MRCLASS = {47G05},
  MRNUMBER = {358443},
MRREVIEWER = {L.\ Cattabriga},
}

@article {B86,
    AUTHOR = {Bourgain, J.},
     TITLE = {Averages in the plane over convex curves and maximal
              operators},
   JOURNAL = {J. Analyse Math.},
  FJOURNAL = {Journal d'Analyse Math\'{e}matique},
    VOLUME = {47},
      YEAR = {1986},
     PAGES = {69--85},
      ISSN = {0021-7670,1565-8538},
   MRCLASS = {42B25 (52A10)},
  MRNUMBER = {874045},
MRREVIEWER = {K.\ J.\ Falconer},
       DOI = {10.1007/BF02792533},
       URL = {https://doi.org/10.1007/BF02792533},
}

@article {S76,
    AUTHOR = {Stein, Elias M.},
     TITLE = {Maximal functions. {I}. {S}pherical means},
   JOURNAL = {Proc. Nat. Acad. Sci. U.S.A.},
  FJOURNAL = {Proceedings of the National Academy of Sciences of the United
              States of America},
    VOLUME = {73},
      YEAR = {1976},
    NUMBER = {7},
     PAGES = {2174--2175},
      ISSN = {0027-8424},
   MRCLASS = {42A40 (43A85)},
  MRNUMBER = {420116},
MRREVIEWER = {Alberto\ Torchinsky},
       DOI = {10.1073/pnas.73.7.2174},
       URL = {https://doi.org/10.1073/pnas.73.7.2174},
}

@article {S70,
    AUTHOR = {Strichartz, Robert S.},
     TITLE = {Convolutions with kernels having singularities on a sphere},
   JOURNAL = {Trans. Amer. Math. Soc.},
  FJOURNAL = {Transactions of the American Mathematical Society},
    VOLUME = {148},
      YEAR = {1970},
     PAGES = {461--471},
      ISSN = {0002-9947,1088-6850},
   MRCLASS = {47.70 (46.00)},
  MRNUMBER = {256219},
MRREVIEWER = {Z.\ Ditzian},
       DOI = {10.2307/1995383},
       URL = {https://doi.org/10.2307/1995383},
}

@article {M97,
    AUTHOR = {Magyar, Akos},
     TITLE = {{$L^p$}-bounds for spherical maximal operators on {$\bold
              Z^n$}},
   JOURNAL = {Rev. Mat. Iberoamericana},
  FJOURNAL = {Revista Matem\'atica Iberoamericana},
    VOLUME = {13},
      YEAR = {1997},
    NUMBER = {2},
     PAGES = {307--317},
      ISSN = {0213-2230},
   MRCLASS = {42B25},
  MRNUMBER = {1617657},
MRREVIEWER = {Sundaram\ Thangavelu},
       DOI = {10.4171/RMI/222},
       URL = {https://doi.org/10.4171/RMI/222},
}

@article {BIT16,
    AUTHOR = {Bennett, Michael and Iosevich, Alexander and Taylor, Krystal},
     TITLE = {Finite chains inside thin subsets of {$\Bbb{R}^d$}},
   JOURNAL = {Anal. PDE},
  FJOURNAL = {Analysis \& PDE},
    VOLUME = {9},
      YEAR = {2016},
    NUMBER = {3},
     PAGES = {597--614},
      ISSN = {2157-5045,1948-206X},
   MRCLASS = {28A75 (42B10)},
  MRNUMBER = {3518531},
MRREVIEWER = {Gareth\ Speight},
       DOI = {10.2140/apde.2016.9.597},
       URL = {https://doi.org/10.2140/apde.2016.9.597},
}

@article {OT22,
    AUTHOR = {Ou, Yumeng and Taylor, Krystal},
     TITLE = {Finite point configurations and the regular value theorem in a
              fractal setting},
   JOURNAL = {Indiana Univ. Math. J.},
  FJOURNAL = {Indiana University Mathematics Journal},
    VOLUME = {71},
      YEAR = {2022},
    NUMBER = {4},
     PAGES = {1707--1761},
      ISSN = {0022-2518,1943-5258},
   MRCLASS = {42B10 (28A78 52C10)},
  MRNUMBER = {4481098},
MRREVIEWER = {Rami\ Ayoush},
}

@article{pinnedtrees,
title = {Nonempty interior of pinned distance and tree sets},
journal = {Advances in Mathematics},
volume = {493},
pages = {110917},
year = {2026},
issn = {0001-8708},
doi = {https://doi.org/10.1016/j.aim.2026.110917},
url = {https://www.sciencedirect.com/science/article/pii/S0001870826001398},
author = {Tainara Borges and Benjamin Foster and Yumeng Ou and Eyvindur Palsson}
}

@article {GILP15,
    AUTHOR = {Greenleaf, Allan and Iosevich, Alex and Liu, Bochen and
              Palsson, Eyvindur},
     TITLE = {A group-theoretic viewpoint on {E}rd\"os-{F}alconer problems
              and the {M}attila integral},
   JOURNAL = {Rev. Mat. Iberoam.},
  FJOURNAL = {Revista Matem\'atica Iberoamericana},
    VOLUME = {31},
      YEAR = {2015},
    NUMBER = {3},
     PAGES = {799--810},
      ISSN = {0213-2230,2235-0616},
   MRCLASS = {42B20 (52C10)},
  MRNUMBER = {3420476},
MRREVIEWER = {Tuomas\ P.\ Hyt\"onen},
       DOI = {10.4171/RMI/854},
       URL = {https://doi.org/10.4171/RMI/854},
}

@article {IPPS22,
    AUTHOR = {Iosevich, Alex and Pham, Minh-Quy and Pham, Thang and Shen,
              Chun-Yen},
     TITLE = {Pinned simplices and connections to product of sets on
              paraboloids},
   JOURNAL = {Indiana Univ. Math. J.},
  FJOURNAL = {Indiana University Mathematics Journal},
    VOLUME = {74},
      YEAR = {2025},
    NUMBER = {3},
     PAGES = {647--668},
      ISSN = {0022-2518,1943-5258},
   MRCLASS = {52C10 (28A80 42B10)},
  MRNUMBER = {4946877},
}

@article {GI12,
    AUTHOR = {Greenleaf, Allan and Iosevich, Alex},
     TITLE = {On triangles determined by subsets of the {E}uclidean plane,
              the associated bilinear operators and applications to discrete
              geometry},
   JOURNAL = {Anal. PDE},
  FJOURNAL = {Analysis \& PDE},
    VOLUME = {5},
      YEAR = {2012},
    NUMBER = {2},
     PAGES = {397--409},
      ISSN = {2157-5045,1948-206X},
   MRCLASS = {42B15 (52C10)},
  MRNUMBER = {2970712},
MRREVIEWER = {Andreas\ Seeger},
       DOI = {10.2140/apde.2012.5.397},
       URL = {https://doi.org/10.2140/apde.2012.5.397},
}

@incollection {EHI13,
    AUTHOR = {Erdogan, Burak and Hart, Derrick and Iosevich, Alex},
     TITLE = {Multiparameter projection theorems with applications to
              sums-products and finite point configurations in the
              {E}uclidean setting},
 BOOKTITLE = {Recent advances in harmonic analysis and applications},
    SERIES = {Springer Proc. Math. Stat.},
    VOLUME = {25},
     PAGES = {93--103},
 PUBLISHER = {Springer, New York},
      YEAR = {2013},
      ISBN = {978-1-4614-4565-4; 978-1-4614-4564-7},
   MRCLASS = {28A80 (42B08)},
  MRNUMBER = {3066881},
MRREVIEWER = {Li-Feng\ Xi},
       DOI = {10.1007/978-1-4614-4565-4\_11},
       URL = {https://doi.org/10.1007/978-1-4614-4565-4_11},
}

@article {PRA23,
    AUTHOR = {Palsson, Eyvindur Ari and Romero Acosta, Francisco},
     TITLE = {A {M}attila-{S}j\"olin theorem for triangles},
   JOURNAL = {J. Funct. Anal.},
  FJOURNAL = {Journal of Functional Analysis},
    VOLUME = {284},
      YEAR = {2023},
    NUMBER = {6},
     PAGES = {Paper No. 109814, 20},
      ISSN = {0022-1236,1096-0783},
   MRCLASS = {28A78 (42B20 52C10)},
  MRNUMBER = {4530888},
MRREVIEWER = {Vladimir\ Eiderman},
       DOI = {10.1016/j.jfa.2022.109814},
       URL = {https://doi.org/10.1016/j.jfa.2022.109814},
}

@article {PRA25,
    AUTHOR = {Palsson, Eyvindur Ari and Romero Acosta, Francisco},
     TITLE = {A {M}attila-{S}j\"olin theorem for simplices in low
              dimensions},
   JOURNAL = {Math. Ann.},
  FJOURNAL = {Mathematische Annalen},
    VOLUME = {391},
      YEAR = {2025},
    NUMBER = {1},
     PAGES = {1123--1146},
      ISSN = {0025-5831,1432-1807},
   MRCLASS = {28A75 (28A78 42B20 52A20)},
  MRNUMBER = {4846807},
MRREVIEWER = {Bochen\ Liu},
       DOI = {10.1007/s00208-024-02948-z},
       URL = {https://doi.org/10.1007/s00208-024-02948-z},
}

@article {GGIP15,
    AUTHOR = {Grafakos, Loukas and Greenleaf, Allan and Iosevich, Alex and
              Palsson, Eyvindur},
     TITLE = {Multilinear generalized {R}adon transforms and point
              configurations},
   JOURNAL = {Forum Math.},
  FJOURNAL = {Forum Mathematicum},
    VOLUME = {27},
      YEAR = {2015},
    NUMBER = {4},
     PAGES = {2323--2360},
      ISSN = {0933-7741,1435-5337},
   MRCLASS = {42B15 (05D05)},
  MRNUMBER = {3365800},
       DOI = {10.1515/forum-2013-0128},
       URL = {https://doi.org/10.1515/forum-2013-0128},
}

@article {GIT22,
    AUTHOR = {Greenleaf, Allan and Iosevich, Alex and Taylor, Krystal},
     TITLE = {On {$k$}-point configuration sets with nonempty interior},
   JOURNAL = {Mathematika},
  FJOURNAL = {Mathematika. A Journal of Pure and Applied Mathematics},
    VOLUME = {68},
      YEAR = {2022},
    NUMBER = {1},
     PAGES = {163--190},
      ISSN = {0025-5793,2041-7942},
   MRCLASS = {28A75 (28A80 52C10 58J40)},
  MRNUMBER = {4405974},
MRREVIEWER = {Xiumin\ Du},
       DOI = {10.1112/mtk.12114},
       URL = {https://doi.org/10.1112/mtk.12114},
}

@article {GIT24,
    AUTHOR = {Greenleaf, Allan and Iosevich, Alex and Taylor, Krystal},
     TITLE = {Nonempty interior of configuration sets via microlocal
              partition optimization},
   JOURNAL = {Math. Z.},
  FJOURNAL = {Mathematische Zeitschrift},
    VOLUME = {306},
      YEAR = {2024},
    NUMBER = {4},
     PAGES = {Paper No. 66, 20},
      ISSN = {0025-5874,1432-1823},
   MRCLASS = {28A75 (28A80 52C10 58J40)},
  MRNUMBER = {4716767},
MRREVIEWER = {Stefan\ Steinerberger},
       DOI = {10.1007/s00209-024-03466-z},
       URL = {https://doi.org/10.1007/s00209-024-03466-z},
}

@article {GIP17,
    AUTHOR = {Greenleaf, Allan and Iosevich, Alex and Pramanik, Malabika},
     TITLE = {On necklaces inside thin subsets of {$\Bbb R^d$}},
   JOURNAL = {Math. Res. Lett.},
  FJOURNAL = {Mathematical Research Letters},
    VOLUME = {24},
      YEAR = {2017},
    NUMBER = {2},
     PAGES = {347--362},
      ISSN = {1073-2780,1945-001X},
   MRCLASS = {28A80 (43A46)},
  MRNUMBER = {3685274},
       DOI = {10.4310/MRL.2017.v24.n2.a4},
       URL = {https://doi.org/10.4310/MRL.2017.v24.n2.a4},
}

@misc{IMMM25,
      author={Alex Iosevich and Akos Magyar and Alex McDonald and Brian McDonald},
      title={The {VC}-dimension and point configurations in $\mathbb{R}^d$}, 
      year={2025},
      eprint={2510.13984},
      archivePrefix={arXiv},
      primaryClass={math.CA},
      note = {arXiv preprint arXiv:2510.13984},
      url={https://arxiv.org/abs/2510.13984}, 
}

@misc{BFOPRA26,
      title={Falconer-type results for any finite graph with multiple pins}, 
      author={Tainara Borges and Ben Foster and Yumeng Ou and Eyvindur Palsson and Francisco Romero Acosta},
      year={2026},
      eprint={2603.01954},
      archivePrefix={arXiv},
      primaryClass={math.CA},
      note = {arXiv preprint arXiv:2603.01954},
      url={https://arxiv.org/abs/2603.01954}, 
}

@book {K86,
    AUTHOR = {Kitaoka, Y.},
     TITLE = {Lectures on {S}iegel modular forms and representation by
              quadratic forms},
    SERIES = {Tata Institute of Fundamental Research Lectures on Mathematics
              and Physics},
    VOLUME = {77},
 PUBLISHER = {Tata Institute of Fundamental Research, Bombay; by
              Springer-Verlag, Berlin},
      YEAR = {1986},
     PAGES = {vi+227},
      ISBN = {3-540-16472-3},
   MRCLASS = {11F46 (11D85 11E45)},
  MRNUMBER = {843330},
MRREVIEWER = {O.\ M.\ Fomenko},
       DOI = {10.1007/978-3-662-00779-2},
       URL = {https://doi.org/10.1007/978-3-662-00779-2},
}

@article {R59,
    AUTHOR = {Raghavan, S.},
     TITLE = {Modular forms of degree {$n$} and representation by quadratic
              forms},
   JOURNAL = {Ann. of Math. (2)},
  FJOURNAL = {Annals of Mathematics. Second Series},
    VOLUME = {70},
      YEAR = {1959},
     PAGES = {446--477},
      ISSN = {0003-486X},
   MRCLASS = {10.20},
  MRNUMBER = {122803},
MRREVIEWER = {H.\ D.\ Kloosterman},
       DOI = {10.2307/1970325},
       URL = {https://doi.org/10.2307/1970325},
}

@article {S44,
    AUTHOR = {Siegel, Carl Ludwig},
     TITLE = {On the theory of indefinite quadratic forms},
   JOURNAL = {Ann. of Math. (2)},
  FJOURNAL = {Annals of Mathematics. Second Series},
    VOLUME = {45},
      YEAR = {1944},
     PAGES = {577--622},
      ISSN = {0003-486X},
   MRCLASS = {10.0X},
  MRNUMBER = {10574},
MRREVIEWER = {B.\ W.\ Jones},
       DOI = {10.2307/1969191},
       URL = {https://doi.org/10.2307/1969191},
}
\vspace{25mm}

\end{document}